\documentclass[aoap]{imsart}

\usepackage[text={5.3in, 7.3in},centering]{geometry}
\usepackage{amsmath,amssymb,mathrsfs,amsthm,amsfonts,mathtools}
\usepackage[inline]{enumitem} 
\usepackage[usenames,dvipsnames]{xcolor}
\usepackage{hyperref}
\usepackage{tikz}
\usepackage{bbm}
\usetikzlibrary{shapes,positioning}

\usepackage{comment}
\hypersetup{%
	colorlinks=true, linkcolor=ForestGreen,
	citecolor=ForestGreen
}
\usepackage{acronym}
\usepackage{dsfont}

\newtheorem{thm}{Theorem}[section]
\newtheorem{remark}[thm]{Remark}
\newtheorem{defn}[thm]{Definition}
\newtheorem{ppn}[thm]{Proposition}

\newtheorem{lem}[thm]{Lemma}

\mathtoolsset{showonlyrefs}

\newcommand{\vvert}[1]{{\left\vert\kern-0.25ex\left\vert\kern-0.25ex\left\vert #1 
    \right\vert\kern-0.25ex\right\vert\kern-0.25ex\right\vert}}

\renewcommand{\P}{\mathds{P}}

\newcommand{\E}{\mathds{E}}
\newcommand{\R}{\mathds{R}}

\newcommand{\Z}{\mathds{Z}}

\newcommand{\e}{\mathrm{e}}

\DeclareMathOperator{\sech}{sech}

\newcommand\cC{\mathcal{C}}

\newtheorem{stat}{Statement}[section]

\newtheorem{theorem}[stat]{Theorem}

\theoremstyle{definition}

\numberwithin{equation}{section}

\makeatother

\begin{document}

\begin{frontmatter}

\title{Universality of Persistence of Random Polynomials}
\runtitle{Universality of Persistence}

	\thankstext{m1}{Department of Statistics, University of Chicago. PG's research is Partially supported  by NSF Grant DMS-2346685.}
	\thankstext{m2}{Department of Statistics, University of Columbia University. SM's research is partially supported by NSF Grant DMS-1712037.}

\begin{aug}

\author{\fnms{Promit} \snm{Ghosal}\thanksref{m1}\ead[label=e1]{promit@uchicago.edu}}
		\and
		\author{\fnms{Sumit} \snm{Mukherjee}\thanksref{m2}\ead[label=e2]{sm3949@columbia.edu}}
		
\runauthor{Ghosal \& Mukherjee}		
			
%
%

			

\end{aug}



\begin{abstract}
We investigate the probability that a random polynomial with independent, mean-zero and finite variance coefficients  has no real zeros. Specifically, we consider a random polynomial of degree $2n$ with coefficients given by an i.i.d. sequence of mean-zero, variance-1 random variables, multiplied by an $\frac{\alpha}{2}$-regularly varying sequence for $\alpha>-1$. We show that the probability of no real zeros is asymptotically $n^{-2(b_{\alpha}+b_0)}$
 , where $b_{\alpha}$ is
 the persistence exponents of a mean-zero, one-dimensional stationary Gaussian processes with covariance function 
 $\sech((t-s)/2)^{\alpha+1}$. Our work generalizes the previous results of Dembo et al. \cite{DPSZ02} and Dembo \& Mukherjee \cite{DemboMukherjee2015} by removing the requirement of finite moments of all order or Gaussianity. In particular, in the special case $\alpha = 0$, our findings confirm a conjecture by Poonen and Stoll \cite[Section~9.1]{poonen1999cassels} concerning random polynomials with i.i.d. coefficients.
\end{abstract}





\end{frontmatter}

\section{Introduction}

In this paper, we consider the random polynomials $$Q_n(x): = \sum_{i=0}^{n}a_i x^i,$$ where $a_i=\sqrt{R(i)}\xi_i$, with $\{\xi_i\}_{0\le i\le n}$ i.i.d.~random variables $\mathbb{E}[\xi_i]=0$ and $\mathrm{Var}(\xi_i)=1$. Hhere $R(\cdot):\mathbb{Z}_+\mapsto (0,\infty)$ is a regularly varying function of order $\alpha>-1$, with $R(0)=1$. The regularly varying assumption implies we can write $R(i)=i^{\alpha}L(i)$ for some $\alpha\in \mathbb{R}$, and $\{L(i)\}_{i\geq 0}$ is slowly varying  at infinity, i.e., $$\lim_{i\to\infty}L(\mu i)/L(i)=1,\text{ for all } \mu>0.$$ We seek to study the decay of the persistence probability, i.e., the probability of $Q_n$ having no real zero. To this goals, we introduce the following:
\begin{align}
p^{(\alpha)}_{n}:=\mathbb{P}\big(Q_n(\cdot) \text{ has no real zero}\big).
\end{align} 

Since an odd degree polynomial necessarily has a real root, we have $p_n^{(\alpha)}=0$ for $n$ odd. For $n$ even, the quantity $\log p^{(\alpha)}_{n}/\log n$ is anticipated to converge to a constant as $n\to\infty$. This constant is known as the persistence exponent of the random polynomial $Q_n$. In the scenario when $R(i)=1$, i.e.~$a_i$ are i.i.d.~random variables,—Poonen and Stoll \cite[Section~9.1]{poonen1999cassels} conjectured that this constant remains invariant across different distributions of $a_i$, provided that $a_i$
  satisfy the following conditions: (1) mean zero, (2) $P(a_i\neq 0)>0$, and (3) $a_i$ belongs to the domain of attraction of the normal distribution. Essentially, this conjecture suggests that the persistence exponent of $Q_n$ is universal within the class of i.i.d. random variables with zero mean and finite variance.

Our main result, outlined below, resolves this conjecture and extends it by demonstrating the universality of persistence exponents for $Q_n$ across a significantly broader range of random polynomial classes. 

 \begin{theorem}\label{thm:Main}
 Fix $\alpha>-1$, and assume $a_i=\sqrt{R(i)}\xi_i $ with $\{\xi_i\}_{0\le i\le n}$~i.i.d. random variables with $\mathbb{E}[\xi_i]=0$ and $\mathrm{Var}(\xi_i) =1$. Then we have $$\lim_{n\to\infty}-\frac{\log p^{(\alpha)}_{2n}}{\log n}=2b_{\alpha}+ 2b_0,$$ where $b_{\alpha}$ is defined via no-zero crossing probability of a stationary Gaussian process $\{Y^{(\alpha)}_t\}_{t\geq 0}$, i.e.,  
 \begin{align}\label{eq:b_alpha}
 b_{\alpha}:=-\lim_{T\to \infty} \frac{1}{T}\log \mathbb{P}\big(Y^{(\alpha)}_t\leq 0,\text{ for all }t\in [0,T]\big),
\end{align} 
where $(Y^{(\alpha)}_t)_{t\geq 0}$ is a mean zero stationary Gaussian process with the following covariance function $\mathrm{Cov}(Y^{(\alpha)}_s, Y^{(\alpha)}_t) = \mathrm{sech}((s-t)/2)^{\alpha+1}$. 
 \end{theorem}

\begin{remark}
The limit in \eqref{eq:b_alpha} exists by Slepian's Lemma, on using stationarity and non-negativity of the covariance function (see \cite[Lem 1.1]{DemboMukherjee2015}). Theorem~\ref{thm:Main} can be extended to the case where $\xi_i$
  are independent but not identically distributed random variables. To establish this generalization, an appropriate uniform integrability condition on $\xi_i$
  is required. However, our proof techniques remain largely unchanged in this scenario. To avoid unnecessary complications, we present and prove the theorem for the case where $\xi_i$
  are identically and independently distributed.
\end{remark}

Significant strides have been made in the investigation of persistence in random polynomials, starting with \cite{DPSZ02}, where it was shown that when the coefficients $a_i$ of $Q_n$ are i.i.d. and 
  possess finite moments of all orders, the persistence exponents of $Q_n$
  are independent of the distribution of $a_i$
  and are expressed as $-4b_0$. This seminal result laid the foundation for subsequent research in this field.
In \cite{DemboMukherjee2015}, the persistence probability was examined under the assumption that the coefficients of random polynomials are independent Gaussian variables with mean zero and variances that vary regularly. It was demonstrated therein that the persistence exponent in this context is given by $-2b_{\alpha} - 2b_0$, as stated in Theorem~\ref{thm:Main}. 

The persistence probabilities of random polynomials and other stochastic systems have also been closely scrutinized in the physics literature; see, for instance, \cite{majumdar1999persistence, schehr2007statistics, schehr2008real} for further exploration.
In a recent contribution \cite{poplavskyi2018exact}, the authors computed $b_0 = \frac{3}{16}$ precisely by evaluating the probability of non-real eigenvalues within the random orthogonal matrix ensemble \cite{kanzieper2016probability}. In fact, \cite{poplavskyi2018exact} found also an alternative way to arrive at the same exponent by matching the persistence probabilities between a 2-dimensional heat equation started from random initial data (studied in \cite{DemboMukherjee2015}) and the Glauber dynamics of a semi-infinite Ising spin chain (studied in \cite{derrida1995exact,derrida1996exact}).

The exploration of the zero-count of random polynomials has a rich historical backdrop. A particularly well-studied scenario involves the coefficients ${a_i}_{i\geq 1}$ being independent and identically distributed as standard Gaussian random variables. Early investigations, dating back to the work of Littlewood and Offord \cite{LO1938, LO1938b, LO1938c}, yielded both upper and lower bounds for the expected total number of real roots, denoted as $N_n$. Their contributions also provided insights into the tails of this distribution, offering an initial estimate that $\mathbb{P}(N_n = 0)$ scales as $O\big(\frac{1}{\log n}\big)$ when the coefficients $a_i$ follow the distribution $\mathbb{P}(a_i = -\frac{1}{3}) = \mathbb{P}(a_i = 0) = \mathbb{P}(a_i = 1) = \frac{1}{3}$.

Kac \cite{Kac1939} made a significant stride by deriving the exact formula for the expected value of $N_n$ when $a_i$'s are independent and identically distributed as standard normal random variables. Subsequently, over several decades, researchers honed in on the precise asymptotic behavior of $\mathbb{E}[N_n]$, refining and extending Kac's findings. In a more recent development, as highlighted by \cite{NNV2016}, it was demonstrated that the asymptotic expectation of $\mathbb{E}[N_n]$ equals $\frac{2}{\pi} \log n + O(1)$ for the case where $a_i$ follows a distribution with $\mathbb{E}[a_i] = 0$ and $\mathbb{E}[a_i^{2+\varepsilon}] < \infty$.

 The challenge of determining the probability of a random polynomial having no real roots remained a formidable task until the seminal work of \cite{DPSZ02}. Their work established that for even values of $n$, $\mathbb{P}(N_n =0)$ exhibits polynomial decay, specifically as $n^{-4b_0+ o(1)}$. They also showed that the probability of $N_{n+k}$ to have exact $k$ (simple) zeros also behave in the same way.  
The study of non-identically distributed, but independent coefficients has been considered in \cite{DemboMukherjee2015}, where the authors verify Theorem~\ref{thm:Main}, under the assumption that all coefficients are Gaussian. 
Theorem~\ref{thm:Main} represents a significant generalization of the findings in \cite{DemboMukherjee2015} and \cite{DPSZ02}. It achieves this by removing the requirement that $a_i$ possess finite moments of all orders or follow Gaussian distributions. In other word, it shows universality of the persistence of random polynomials. This pivotal development resolves Poonen and Stole's conjecture on persistence of random polynomials. 

\section{Proof Ideas}

The study of random polynomials has been a dynamic field, continually yielding fresh insights and relevant methodologies. A significant portion of this research has been devoted to investigating the presence of real roots in random polynomials. In particular, the problem of a random polynomial having no real zeros has attracted widespread attention within the mathematics and physics communities.

Many prior approaches to establish the persistence of such random polynomials have leaned heavily on comparisons with Gaussian distributions, often employing techniques like Komlos-Major-Tusnady embedding and tools such as Slepian's inequality \cite{DPSZ02, DemboMukherjee2015}. However, a notable challenge in establishing the universality of this persistence lies in the absence of KMT-type results for random variables with only finite second moments. Moreover, the challenge becomes even more complex when dealing with non-identically distributed random variables. We seek to resolve this by innovating a new approach for showing the persistence. We divide the proofs into two parts which are given as follows: 
\begin{align}
\underbrace{\limsup_{n\to \infty} \frac{\log p^{(\alpha)}_{2n}}{\log n} \leq - 2b_{\alpha} - 2b_{0}}_{\mathfrak{LimSup}}, \qquad \underbrace{\liminf_{n\to \infty} \frac{\log p^{(\alpha)}_{2n}}{\log n} \geq - 2b_{\alpha} - 2b_{0}}_{\mathfrak{LimInf}}.
\end{align}
We show $\mathfrak{LimSup}$ and $\mathfrak{LimInf}$ in Sections~\ref{sec:upper_bd} and~\ref{sec:lower_bd} respectively.  
The proof ideas for showing these two claims are different and in what follows, we discuss those proof ideas in details. We divide the proofs of $\mathfrak{LimSup}$ and $\mathfrak{LimInf}$ into two steps. We begin with the steps of the proof of $\mathfrak{LimSup}$ and then proceed to the steps in the proof of $\mathfrak{LimInf}$. From this point forward, we assume that $n$ is an even integer. However, it turns out that the proofs concerning $\mathfrak{LimSup}$ and $\mathfrak{LimInf}$ do not make use of this assumption, with the exception of Propositions~\ref{ppn:r=-2} and \ref{ppn:r=2}.

\noindent\textbf{Step I of $\mathfrak{LimSup}$:} 
To establish $\mathfrak{LimSup}$, we initially bound $p^{(\alpha)}_{2n}$ by considering the probability of having no real zeros for $Q_n(x)$ within four small intervals. These intervals are located around two points: around $+1$ which is associated to the intervals $\mathfrak{D}_1:= [e^{-1/n^{\delta}}, e^{-1/n^{1-\delta}}]$ and $\mathfrak{D}_2: =[e^{1/n^{\delta}}, e^{1/n^{1-\delta}}]$, and $-1$ associated to intervals $\mathfrak{D}_3: = [-e^{-1/n^{1-\delta}}, -e^{-1/n^{\delta}}]$ and $\mathfrak{D}_4: =[-e^{1/n^{1-\delta}}, -e^{1/n^{\delta}}]$.

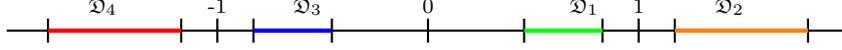
\begin{figure}
\begin{tikzpicture}[scale=0.8]

\draw[thick] (-7, 0) -- (7, 0);  

\draw[thick] (-6.3147, 0.2) -- (-6.3147, -0.2);  
\draw[thick] (-4.1005, 0.2) -- (-4.1005, -0.2);  
\draw[thick] (-3.5005, 0.2) -- (-3.5005, -0.2);
\draw[thick] (-2.9005, 0.2) -- (-2.9005, -0.2);  
\draw[thick] (-1.5971, 0.2) -- (-1.5971, -0.2);  
\draw[thick] (0, 0.2) -- (0, -0.2);  

\draw[thick] (1.5971, 0.2) -- (1.5971, -0.2);    
\draw[thick] (2.9005, 0.2) -- (2.9005, -0.2);    
\draw[thick] (3.5005, 0.2) -- (3.5005, -0.2);
\draw[thick] (4.1005, 0.2) -- (4.1005, -0.2);    
\draw[thick] (6.3147, 0.2) -- (6.3147, -0.2);    

\node at (-5.1005 - 0.3, 0.4) {$\mathfrak{D}_4$};
\node at (-1.6971 - 0.3, 0.4) {$\mathfrak{D}_3$};
\node at (0, 0.4) {0};
\node at (3.5005, 0.4) {1};
\node at (-3.5005, 0.4) {-1};
\node at (2.9005 - 0.3, 0.4) {$\mathfrak{D}_1$};
\node at (5.3147 - 0.3, 0.4) {$\mathfrak{D}_2$};

\draw[red, ultra thick] (-6.3147, 0) -- (-4.1005, 0);   
\draw[blue, ultra thick] (-2.9005, 0) -- (-1.5971, 0);  
\draw[green, ultra thick] (1.5971, 0) -- (2.9005, 0);   
\draw[orange, ultra thick] (4.1005, 0) -- (6.3147, 0);  

\end{tikzpicture}
\caption{Illustration of contributing intervals in $\mathfrak{LimSup}$ Case. The probability of no real zero of $Q_n(x)$ on the intervals $\mathfrak{D}_1$ or $\mathfrak{D}_3$ is $O(n^{-b_{\alpha}})$. Similarly, the probability of no real zero of $Q_n(x)$ on the intervals $\mathfrak{D}_2$ or $\mathfrak{D}_4$ is $O(n^{-b_{0}})$. Combining these bounds with `near' independence of these four events shows $p^{(\alpha)}_{2n}= O(n^{-2(b_{\alpha}+b_0)})$.}
\label{fig:1}
\end{figure}

Now, we need to analyze the probability of no real roots of $Q_n$ in each of these four intervals. However, it is important to note that the events of having no zeros in these intervals are not independent. To address this challenge, we demonstrate that the probability of having no real zeros of $Q_n$ in disjoint intervals $\mathfrak{D}_1, \mathfrak{D}_2, \mathfrak{D}_3, \mathfrak{D}_4$ are dictated by the disjoint components of $Q_n$. For instance, the probability of no zero of $Q_n$ in the interval $\mathfrak{D}_1$ is mainly governed by $\sum_{i\in [n^{\delta}, n^{1-\delta}]} a_i x^i$ and this holds roughly because the variance of $\sum_{i \in [n^{\delta}, n^{1-\delta}]} a_i x^i$ is greater than any other portions of $Q_n$ when $x\in \mathfrak{D}_1$ (see Theorem~).
Consequently, we can express the probability of having no real zeros in any of these four intervals as approximately the product of the probabilities of having no real zeros in each of these intervals.     
\begin{align}
     \P\big(Q_n(x) \text{ has no real zeros in } \mathbb{R}\big) &\leq \P\big(Q_n(x) \text{ has no real zeros in } \mathbb{R}\big) \nonumber \\
      \lesssim &\prod_{i=1}^{4}  \P\Big(Q_n(x) \text{ has no real zero in } \mathfrak{D}_i\Big) \label{eq:UpperBd_Split}
\end{align}

\noindent \textbf{Step II of $\mathfrak{LimSup}$:}
Now we need to upper bound the probability of no real zero in the each of the intervals $\mathfrak{D}_1$ and $\mathfrak{D}_2$. We show that probability of no zeros of $Q_n$ in the interval $\mathfrak{D}_1$ behaves as $n^{-b_{\alpha}}$ and the probability of no real zeros in the interval $\mathfrak{D}_2$ behaves as $n^{-b_0}$ (see \eqref{eq:b_alpha} for the definition of $b_{\alpha}$ for any $\alpha>-1$). By symmetry, the probabilities of no real zero in the interval $\mathfrak{D}_3$ and $\mathfrak{D}_4$ are respectively $n^{-b_\alpha}$ and $n^{-b_0}$. More concretely, we show the following: 
$$\limsup_{n\to \infty} \frac{1}{\log n}\log \P\Big(Q_n(x) \text{ has no real zero in }\mathfrak{D}_1 \text{ or, } \mathfrak{D}_3\Big)\leq- b_{\alpha}$$
and, 
$$\limsup_{n\to \infty} \frac{1}{\log n}\log \P\Big(Q_n(x) \text{ has no real zero in }\mathfrak{D}_2 \text{ or, } \mathfrak{D}_4\Big)\leq - b_{0}$$
Sum of these yields the desired upper bound following \eqref{eq:UpperBd_Split}. 

We  now explain why the probability of no zeros for the intervals $\mathfrak{D}_1$ and $\mathfrak{D}_2$ are different and how we bound those probabilities. Recall that the probability of no real zero of $Q_n$ in the interval $\mathfrak{D}_1$ depends on $\sum_{i\in [n^{\delta}, n^{1-\delta}]\cap \mathbb{Z}} a_i x^i$ whereas the probability of no real zero of $Q_n$ in $\mathfrak{D}_2$ is governed by $\sum_{i\in [n-n^{1-\delta}, n-n^{\delta}]\cap \mathbb{Z}} a_i x^i$. 
We further divide the polynomial $\sum_{i\in [n^{\delta}, n^{1-\delta}]\cap \mathbb{Z}} a_i x^i$ into almost $\log n$ number of sub-polynomials $\sum_{i \in \mathcal{I}^{(-)}_r} a_i x^i$ (see \eqref{eq:I_r-} for $\mathcal{I}^{(-)}_r$) of varying sizes. Each of these sub-polynomials (for instance, $\sum_{i \in \mathcal{I}^{(-)}_r} a_i x^i$) when restricted to a particular sub-interval ($\mathcal{J}^{(-)}_r$ of \eqref{eq:J_rDef} in the case of $\mathcal{I}^{(-)}_r$) of $\mathfrak{D}_1$, converges to a Gaussian process which is directly linked to $\{Y^{(\alpha)}_t\}_{t\geq 0}$. These sub-polynomials are constructed in Section~\ref{sec:Notations} (see \eqref{eq:sub_poly}). Such construction also ensures that those sub-polynomials when restricted to some other sub-intervals of $\mathfrak{D}_1$, the covariance function decays (see Lemma~\ref{lem:weak*}). As a result, the probability of no real zeros of $\sum_{i\in [n^{\delta}, n^{1-\delta}]\cap \mathbb{Z}} a_i x^i$ in any of those sub-intervals of $\mathfrak{D}_1$ can be bounded by the probability of the respective sub-polynomials not exceeding an arbitrary small negative number (Lemma~\ref{lem:clm}). Since all these sub-polynomials are independent of each other and they weakly converges (when restricted to the respective sub-intervals) to the same Gaussian process, we do have 
\begin{align*}
\P\Big(Q_n(x) \text{ has no real zeros in } \mathfrak{D}_1\Big) &\lesssim \prod_{r=1}^{K} \P\Big(\sum_{i \in \mathcal{I}^{(-)}_r} a_i x^i\leq -\delta, \text{ for all }x\in \mathcal{J}^{(-)}_r \Big) \\
&\lesssim  \P \big(Y^{(\alpha)}_{t} \leq 0, t\in [0, \log M]\big)^K
\end{align*}
where $K = O(\log n/\log M)$ for any fixed $M>0$. The constant of proportionality in the above equation is $o(\log n)$. Taking logarithm in both sides in the above display, dividing by $\log n$ and letting $n\to \infty$ shows that probability of no real zeros in $\mathfrak{D}_1$ is bounded by $n^{-b_{\alpha}+o(1)}$. On the other hand, to bound the similar probability for $\mathfrak{D}_2$, we divide $\sum_{i\in [n-n^{1-\delta}, n-n^{\delta}]\cap \mathbb{Z}} a_i x^i$ into almost $\log n$ many sub-polynomials corresponding to index sets $\mathcal{I}^{(+)}_{r}$ (see \eqref{eq:I_r}). However, unlike the $\mathfrak{D}_1$ case, each of these sub-polynomials (for instance, $\sum_{i \in \mathcal{I}^{(+)}_r} a_i x^i$) when restricted to a particular sub-interval ($\mathcal{J}^{(+)}_r$ of \eqref{eq:J_rDef} in the case of $\mathcal{I}^{(+)}_r$) of $\mathfrak{D}_2$, converges to a Gaussian process which is directly linked to $\{Y^{(0)}_t\}_{t\geq 0}$. Now, by a similar argument, we show 
\begin{align*}
\P\Big(Q_n(x) \text{ has no real zeros in } \mathfrak{D}_2\Big) &\lesssim \prod_{r=1}^{K} \P\Big(\sum_{i \in \mathcal{I}^{(+)}_r} a_i x^i\leq -\delta, \text{ for all }x\in \mathcal{J}^{(+)}_r \Big) \\
&\lesssim  \P \big(Y^{(0)}_{t} \leq 0, t\in [0, \log M]\big)^K
\end{align*}
where the right hand side of the above display decays as $n^{-b_0+o(1)}$. Combining these probabilities for $\mathfrak{D}_1,\mathfrak{D}_2, \mathfrak{D}_3$ and $\mathfrak{D}_4$ yields the proof of
$\mathfrak{LimSup}$. Although the steps to find the exact asymptotics are straightforward, the assumption of only finite second moments enforces careful accounting of probability bounds which is executed mainly in the proof of Lemma~\ref{lem:clm} through employing a combinatorial selection criterion `\emph{move, flush and repeat}' and the required probability bounds are assimilated from several important lemmas stated and proved in  Section~\ref{sec:Lemma_4.4_convergence}.  

\vspace{0.2cm}


\noindent\textbf{Step I of $\mathfrak{LimInf}$:} The proof of $\mathfrak{LimInf}$ is carried out by investigating the lower bound of the probability of $Q_n(x)<0$ for all $x\in \mathbb{R}$. This latter event can be rewritten as $\cap_{r}\{Q_{n}(\pm e^{u})<0, u \in A_r\}$ where the intervals $\{A_r:r\in \mathcal{R}\}$ (for $\mathcal{R}:=\{-2,-1,0,1,2\}$) are shown in \eqref{eq:A_s} (see Figure~\ref{fig:2}). 

\begin{figure}
\begin{tikzpicture}[scale=0.8]

\draw[thick] (-7, 0) -- (7, 0);  

\draw[thick] (-6.3147, 0.2) -- (-6.3147, -0.2);  
\draw[thick] (-4.1005, 0.2) -- (-4.1005, -0.2);  
\draw[thick] (-2.9005, 0.2) -- (-2.9005, -0.2);  
\draw[thick] (-1.5971, 0.2) -- (-1.5971, -0.2);  
\draw[thick] (0, 0.2) -- (0, -0.2);              
\draw[thick] (1.5971, 0.2) -- (1.5971, -0.2);    
\draw[thick] (2.9005, 0.2) -- (2.9005, -0.2);    
\draw[thick] (4.1005, 0.2) -- (4.1005, -0.2);    
\draw[thick] (6.3147, 0.2) -- (6.3147, -0.2);    

\node at (-7.5, 0) {$-\infty$};
\node at (-6.3147 - 0.3, 0.4) {$-\e^{\frac{h}{\log n}}$};
\node at (-4.1005 - 0.3, 0.4) {$-\e^{\frac{K}{n}}$};
\node at (-2.9005 - 0.3, 0.4) {$-\e^{\frac{-K}{n}}$};
\node at (-1.5971 - 0.3, 0.4) {$-\e^{-\frac{h}{\log n}}$};
\node at (0, 0.4) {0};
\node at (1.5971 - 0.3, 0.4) {$\e^{-\frac{h}{\log n}}$};
\node at (2.9005 - 0.3, 0.4) {$\e^{\frac{-K}{n}}$};
\node at (4.1005 - 0.3, 0.4) {$\e^{\frac{K}{n}}$};
\node at (6.3147 - 0.3, 0.4) {$\e^{\frac{h}{\log n}}$};
\node at (7.5, 0) {$\infty$};

\draw[red, ultra thick] (-7, 0) -- (-6.3147, 0);
\draw[blue, ultra thick] (-6.3147, 0) -- (-4.1005, 0);
\draw[green, ultra thick] (-4.1005, 0) -- (-2.9005, 0);
\draw[yellow, ultra thick] (-2.9005, 0) -- (-1.5971, 0);
\draw[purple, ultra thick] (-1.5971, 0) -- (0, 0);
\draw[purple, ultra thick] (0, 0) -- (1.5971, 0);
\draw[yellow, ultra thick] (1.5971, 0) -- (2.9005, 0);
\draw[green, ultra thick] (2.9005, 0) -- (4.1005, 0);
\draw[blue, ultra thick] (4.1005, 0) -- (6.3147, 0);
\draw[red, ultra thick] (6.3147, 0) -- (7, 0);

\end{tikzpicture}
\caption{Illustration of how $\mathbb{R}$ is divided into intervals in the $\mathfrak{LimInf}$ Case. Indeed, $[-\e^{\frac{h}{\log n}}, -\e^{\frac{K}{n}}]$, $[-\e^{\frac{-K}{n}}, -\e^{-\frac{h}{\log n}}]$, $[\e^{-\frac{h}{\log n}}, \e^{\frac{-K}{n}}]$ and $[\e^{\frac{-K}{n}}, \e^{\frac{h}{\log n}}]$ are the intervals which mainly contribute to the lower bound of the probability of no real zeros of $Q_{2n}(x)$ and they play same role as in the intervals $\mathfrak{D}_4, \mathfrak{D}_3, \mathfrak{D}_1$ and $\mathfrak{D}_2$ respectively as in the proof of $\mathfrak{LimSup}$.}
\label{fig:2}
\end{figure}
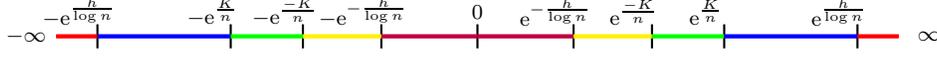

Like as in the proof of $\mathfrak{LimSup}$, we notice that $\P(Q_{n}(\pm e^{u})<0, u \in A_r)$ is dictated by the sub-polynomials $  Q^{(r)}_n(x):=\sum_{i \in B_r} a_i x^i$ (see \eqref{eq:B_s} for the definitions of the sets $B_r$). From the definition of $Q^{(r)}_n(\cdot)$, it follows that $Q_n(x) = \sum_{r\in \mathcal{R}}Q^{(r)}_n(x)$. Such decomposition of $Q_n(x)$ implies  
\begin{align*}
    \Big\{Q_n(x)<0,x\in \R\Big\}\supset \bigcap_{r\in \mathcal{R}}\Big\{\frac{Q_n^{(r)}(\pm e^u)}{\sigma_n(u)}<-\delta , u\in A_r, \frac{Q_n^{(r)}(\pm e^u)}{\sigma_n(u)}<\delta/4, u\notin A_r \Big\}
\end{align*}
where $\sigma_n(u)$ is the standard deviation of the random variable $Q_n(e^u)$. Since the sub-polynomials $Q^{(r)}_n$ are independent, the probability of the event of the right hand side is equal to the product of the the following probabilities;
\begin{align}\label{eq:lowerbd_components}
\P\Big(\frac{Q_n^{(r)}(\pm e^u)}{\sigma_n(u)}<-\delta , u\in A_r, \frac{Q_n^{(r)}(\pm e^u)}{\sigma_n(u)}<\delta/4, u\notin A_r \Big), \qquad r \in \mathcal{R}.
\end{align}
The main thrust of the proof of $\mathfrak{LimInf}$ lies in the lower bound of the above probabilities. Indeed, we derive these lower bounds in Proposition~\ref{prop:r=-1},~\ref{prop:r=1},~\ref{prop:r=0},~\ref{ppn:r=-2} and~\ref{ppn:r=2} respectively. We show in Proposition~\ref{prop:r=-1} and ~\ref{prop:r=1}, that the above probabilities for $r=-1$ and $r=1$ are bounded below by the probabilities of no-zero crossing of two Gaussian processes which later is bounded by the non-zero crossing probabilities of the Gaussian processes $Y^{(\alpha)}_t$ and $Y^{(0)}_t$ respectively (through Lemma~\ref{lem:tilde_Y}). Thus the probabilities of the above display for $r=-1$ and $r=1$ roughly contributes $n^{-2b_{\alpha}}$ and $n^{-2b_0}$ in the lower bound to the probability of $Q_n(x)<0$ for all $x\in \mathbb{R}$. Furthermore, Propositions~\ref{prop:r=0},~\ref{ppn:r=-2} and~\ref{ppn:r=2} shows that the probabilities of the above display corresponding to $r=0,-2,2$ are bound below by $n^{-o(1)}$. Assembling these bounds together yields the proof of $\mathfrak{LimInf}$. In the next step, we discuss the ideas that we use in lower bounding the probabilities in the above display for $r=-1$ and $r=1$. It is important to note that, while the probabilities in \eqref{eq:lowerbd_components} for $r=-2,2$ do not directly contribute to the exact value of the persistence exponent, the proof of their lower bound (which is of the order $n^{-o(1)}$) in Propositions~\ref{ppn:r=-2} and \ref{ppn:r=2} relies heavily on the fact that $Q_n$
  is a polynomial of even degree.

\vspace{0.2cm}

 \noindent \textbf{Step II of $\mathfrak{LimInf}$:} 
  
 We discuss here strategies to bound \eqref{eq:lowerbd_components} for $r= -1,1$. As in the proof of $\mathfrak{LimSup}$, we divide the intervals $A_{-1}$ and $A_{+1}$ into $T_n=O(\log n)$ many intervals. These are now denoted as $\{\tilde{\mathcal{J}}^{(-)}_p:1\leq p\leq T_n\}$ and $\{\tilde{\mathcal{J}}^{(+)}_p:1\leq p\leq T_n\}$ respectively. On the other hand, the index sets $B_{-1} $ and $B_{+1}$ are also divided into $\{\tilde{\mathcal{I}}^{(-)}_p:1\leq r\leq T_n\}$ and $\{\tilde{\mathcal{I}}^{(+)}_p:1\leq r\leq T_n\}$ (see \eqref{eq:tildeI-} and \eqref{eq:tildeI+}). We show that the probability in \eqref{eq:lowerbd_components} could be lower bounded by product of certain probabilities of the polynomials $\tilde{Q}^{(-1),p}_n(x) = \sum_{i \in \tilde{\mathcal{I}}^{(-)}_p} a_i x^i$ when $r=-1$ and $\tilde{Q}^{(-1),p}_n(x) = \sum_{i \in \tilde{\mathcal{I}}^{(-)}_p} a_i x^i$ when $r=+1$. We show below how this bound follows for $r=-1$ case:
 \begin{align*}
\bigcap_{p=1}^{T_n} (B_{1p}\cap B_{2p}\cap B_{3p}\cap B_{4p})\subset \Big\{\frac{Q_n^{(r)}(\pm e^u)}{\sigma_n(u)}<-\delta , u\in A_r, \frac{Q_n^{(r)}(\pm e^u)}{\sigma_n(u)}<\delta/4, u\notin A_r \Big\}.   
 \end{align*}
where $B_{1\cdot}, B_{2\cdot}, B_{3\cdot}$ and $B_{4\cdot}$ are defined in \eqref{eq:b1p}, \eqref{eq:b2p}, \eqref{eq:b3p} and \eqref{eq:b4p} respectively. In words, $B_{1p}$ signifies the event that $\tilde{Q}^{(-1),p}_n(e^u)$ (when scaled by $\sigma_n(u)$) to be less then a small negative constant for all $u$ inside and around a neighborhood of $\mathcal{J}^{(-)}_p$ whereas $B_{2\cdot}, B_{3\cdot}$ and $B_{4\cdot}$ signifies $\tilde{Q}^{(-1),p}_n(e^u)/\sigma_n(e^u)$ to be less than positive numbers which decays as $u$ moves away $\mathcal{J}^{(-)}_p$. We show that $\P(B_{1p}\cap B_{2p}\cap B_{3p}\cap B_{4p})\approx \P(B_{1p})$ and $\P(B_{1p})$ can be lower bounded by probability of the Gaussian process $Y^{(\alpha)}_t$ being smaller a small negative number everywhere inside an interval whose length grows as $ne^{-T_n}$. This last probability 
is indeed same as the no-zeros crossing probability of the Gaussian process $Y^{(\alpha)}_t$ when the small constant is taken to $0$. 
We show the above mentioned lower bound by combining Lemma~\ref{lem:B1p}, Lemma~\ref{lem:B2p}, Lemma~\ref{lem:B34p} with 
Lemma~\ref{lem:tilde_Y}. The first three lemma required careful derivation of the decay of probabilities of $\tilde{Q}^{(-1),p}_n(e^u)/\sigma_n(e^u)$ exceeding a positive constants as $u$ moves away from $\mathcal{J}^{(-)}_p$. Since we have assumed only second moments of $a_i$'s finite, the derivation of these probabilities poses several challenges which are overcome in Lemma~\ref{lem:defer}.
Taking the product of the aforementioned lower bound over all $p \in [1,T_n]\cap \mathbb{Z}$ shows that the probability in \eqref{eq:lowerbd_components} for $r=-1$ can be bounded from below by the product of $T_n$ many no-zero crossing probability of $Y^{(\alpha)}_t$. Similarly, the probability in \eqref{eq:lowerbd_components} for $r=+1$ can be bounded from below by the product of $T_n$ many no-zero crossing probability of $Y^{(0)}_t$. These explain the bounds obtained in Proposition~\ref{prop:r=-1} and~\ref{prop:r=1} which are the most import components in the proof of $\mathfrak{LimInf}$.      

  \section{Upper bound}\label{sec:upper_bd}
 
  \subsection{Fixation of notations}\label{sec:Notations} 
Fix $\delta>0$, and let $$A_+:=\Big[\frac{1}{n^{1-\delta}}, \frac{1}{n^\delta}\Big], \quad A_-:=-A_+.$$

Define the positive weight function $\sigma_n(u)$ as 
 \begin{eqnarray*}
 \sigma_n^2(u)
 :=&\frac{e^{nu} L(n)}{|u|}&\text{ if }u\in A_+\\
 =&\frac{L(1/|u|)}{|u|^{\alpha+1}}&\text{ if }u\in A_{-1}.
 \end{eqnarray*}

Also, define
\begin{align*}
  B_1:=\big(n-n^{1-\delta},n-n^\delta\big]\cap\mathbb{Z},\quad B_{-1}:=\big[n^\delta, n^{1-\delta}\big)\cap\mathbb{Z}.
\end{align*}

Fixing $M>0$, let $K:=\big\lfloor \frac{(1-2\delta)\log n }{\log M}\big\rfloor$. Define the sets $\{\mathcal{I}^{(-)}_r\}_{1\le r\le K}$, by setting
  \begin{align}\label{eq:I_r-}
   \mathcal{I}^{(-)}_{r} :=& \cap \big[n^\delta M^{r-1}, n^\delta  M^{r}\big),\quad \forall 1\le r\leq K,
   \end{align}
   and use the bound $M^K\le n^{1-2\delta}$ to note that $$ \cup_{r=1}^K \mathcal{I}^{(-)}_{r}=[n^\delta, n^\delta M^K)\cap \mathbb{Z}\subseteq \big[n^\delta,n^{1-\delta}\big)=B_{-1}.$$
   Similarly, define the sets $\{\mathcal{I}^{(+)}_r\}_{1\le r\le K}$, where
   \begin{align}\label{eq:I_r}
   \begin{split}
    \mathcal{I}^{(+)}_{r} :=& \mathbb{Z}\cap \big(n-n^\delta  M^{r},n-n^\delta  M^{r-1} \big], \quad \forall 1\le  r\leq  K,\\
   \end{split}
   \end{align}
   and again use $M^K\le n^{1-2\delta}$ to note that $$ \cup_{r=1}^K \mathcal{I}^{(+)}_{r}=\mathbb{Z}\cap (n-n^\delta M^K, n-n^\delta]\subseteq \Big(n-n^{1-\delta}, n-n^\delta\Big]=B_+.$$
We will now partition a proper subset $A_a$, for $a\in \{+,-\}$. 
For $1\le r\le K$, define
\begin{align}\label{eq:J_rDef}
\mathcal{J}_{r}^{(+)}:=\Big(\frac{1}{n^\delta M^{r-\delta}}, \frac{1}{n^\delta M^{r-1+\delta}}\Big],\quad \text{ and }\quad 
\mathcal{J}_{r}^{(-)}:=-\mathcal{J}_{r}^{(+)}.
\end{align}
Then we have $\cup_{r=1}^{K}\mathcal{J}_{r}^{(a)}\subset 
A_{a}$, for $a\in \{+,-\}.$
\\

For $a\in \{+,-\}$ set
\begin{align}\label{eq:sub_poly}
Q^{(a),r}_{n}(x) := \sum_{i \in \mathcal{I}^{(a)}_{r}} a_ix^i. 
\end{align}
%
%
%
%
We will create a back-log of objects. For $a\in \{+,-\}$, define
\begin{align}\label{eq:backlog}
\Xi^{(a)}_{n,r}(\delta) :=& \Big\{\max_{b\in\{+,-\}}\sup_{ u \in \mathcal{J}^{(a)}_{r}}\frac{Q^{(a),r}_n( b e^{u})}{\sigma_n(u)}<\delta\Big\}, \nonumber\\
\mathcal{K}^{(a)}_{n,\delta}(r_1,r_2):=&\Big\{\min_{b\in \{+,-\}}\inf_{  u \in \mathcal{J}^{(a)}_{r_2}}\frac{Q^{(a),r_1}_n(be^{u})}{\sigma_n(u)}<-\delta\Big\}. 
\end{align} 
We further define the index set 
\begin{align*}
\mathcal{I}^{(a)}_{\mathrm{ext}} := [n]_0/\cup_{r=1}^K \mathcal{I}^{(a)}_r,\qquad [n]_0:=\{0,1,2,\cdots,n\}.
\end{align*} 
Let us also define $Q_n^{(a,\mathrm{ext})}(x):=\sum_{i\in \mathcal{I}^{(a)}_{\mathrm{ext}}} a_i x^i$, and set
\begin{align*}
\mathcal{K}^{(a)}_{n,\delta}(\mathrm{ext} , r):=\Big\{\min_{b\in \{+,-\}}\inf_{u \in \mathcal{J}^{(a)}_r}\frac{Q_n^{(a),\mathrm{ext}}(b e^{u})}{\sigma_n(u)}< -\delta\Big\}.
\end{align*}

\subsection{Proof of $\mathfrak{LimSup}$}
We claim that it suffices to show that 
\begin{align}\label{eq:lowerbd_simple}
\limsup_{n\to \infty}\frac{1}{\log n}\log  \P\big(\sup_{x\in \R}Q_{2n}(x)<0\big)\leq -2b_0 - 2b_{\alpha}. 
\end{align}
Indeed, using the above display for the polynomial $\widetilde{Q}_{2n}(x) = \sum_{i=0}^n(-a_i) x^i$, we get
\[\limsup_{n\to \infty}\frac{1}{\log n}\log  \P\big(\inf_{x\in \R}Q_{2n}(x)>0\big)\leq -2b_0 - 2b_{\alpha}.\]
The upper bound $\limsup_{n\to \infty} \log p^{(\alpha)}_{2n}/\log n\leq -2b_{0} - 2b_{\alpha}$ follows by combining the above two bounds. We now proceed now to show \eqref{eq:lowerbd_simple}. We assume $n$ is an even integer. 
To this effect, first note the  trivial upper bound
\begin{align}\label{eq:lower_bound1}
\P\big(\sup_{x\in \R}Q_n(x)<0\big)\le &\P\Big(\cap_{a,b\in \{+,-\}}\sup_{u\in A_a}\frac{Q_n(b e^{u})}{\sigma_n(u)}<0\Big).
\end{align}
Recalling that
$\bigcup_{r=1}^{K}\mathcal{J}^{(a)}_{r}\subset A_{a}$
for $a\in \{-,+\}$
(see~\eqref{eq:J_rDef}), 
 a union bound gives, for any $\omega\in (0,1/2)$,
\begin{align}\label{eq:upp1}
\notag&\P\left(\max_{a,b\in \{+,-\}}\sup_{u\in A_a}\frac{Q_n(b e^u)}{\sigma_n(u)}<0\right)\\
\notag\le & \P\left( \max_{a,b\in \{-,+\}} \max_{\omega K\le r\le (1-\omega)K}\sup_{u\in \mathcal{J}^{(a)}_{r}}\frac{Q_n(b e^u)}{\sigma_n(u)}<0\right)\\
\le &\P\left(\bigcap_{a\in \{-,+\} }\bigcap_{r=\omega K}^{(1-\omega)K} \Big\{{\Xi}^{(a)}_{n,r}(\delta) \cup\Big( \bigcup_{\ell\in [K], \ell\ne r}\mathcal{K}^{(a)}_{n,|\ell-r|^{-2} C\delta}(\ell,r)\Big)\cup \mathcal{K}^{(a)}_{n,\delta/2}(\mathrm{ext},r)\Big\}\right),
 \end{align}
 where $C>0$ is chosen such that $2C\sum_{r=1}^\infty \frac{1}{r^2}= \frac{1}{2}$,
  and the last line follows on noting that for $u\in \mathcal{J}_r^{(a)}$ we can write
 \begin{align*}
 Q_n(b e^u)=&\sum_{i=0}^n a_i (b e^{u})^i=\sum_{i\in {\mathcal{I}}^{(a)}_r}  a_i (b e^{u})^i+\sum_{\ell \in [K]/\{r\} }\sum_{i\in {{\mathcal{I}}}^{(a)}_\ell} a_i (b e^u)^i+\sum_{i\in \mathcal{I}_{\mathrm{ext}}^{(a)}}a_i (b e^{u})^i\\
 =&Q_n^{(a),r}(be^u)+\sum_{\ell\in [K]/\{r\}} Q_n^{(a),\ell}(be^u)+Q_n^{(a), \mathrm{ext}}(be^u)
 \end{align*}
 The r.h.s. of \eqref{eq:upp1} can be further bounded above by 
 \begin{align}\label{eq:upp1.5}
 \P &\left(\bigcap_{a\in \{-,+\} }\bigcap_{r=\omega K}^{(1-\omega)K} \Big\{{\Xi}^{(a)}_{n,r}(\delta) \cup\Big( \bigcup_{\ell\in [K], \ell\ne r}\mathcal{K}^{(a)}_{n,|\ell-r|^{-2} C\delta}(\ell,r)\Big)\right) \\ &+\max_{a\in \{+,-\}}\sum_{\omega K\le r\le (1-\omega)K}\P\left(\mathcal{K}^{(a)}_{n,\delta/2}(\mathrm{ext},r)\right).
 \end{align}

We now state a lemma to bound the second term in the RHS of \eqref{eq:upp1.5}, deferring the proof of the Lemma to the end of the section.

\begin{lem}\label{lem:KextBound}
There exist $\theta= \theta(\omega, \delta, M), C= C(\omega,\delta, M)>0$ such that  
\begin{align}
\max_{\omega K\leq r\leq (1-\omega)K, a\in \{-,+\}} \left\{ \P\left( \mathcal{K}^{(a)_{n,\delta/2}(\mathrm{ext},r)}\right)\right\}\le e^{- Cn^{\theta}}.
\end{align}
\end{lem}

%
Lemma~\ref{lem:KextBound} is proved in Section~\ref{sec:proof_of_lemma_3.1}. As a result of Lemma~\ref{lem:KextBound}, the second term in the RHS of \eqref{eq:upp1.5} is bounded above by $Ke^{-Cn^{\theta}}$. Now we proceed to bound the first term. Notice that the events $\{{\Xi}^{(+)}_{n,r}(\delta) \}_{r\in [K]}$ and $\{\mathcal{K}^{(+)}_{n,|\ell-r|^{-2}} \}_{\ell, r\in [K]}$ are jointly independent of  $\{{\Xi}^{(+)}_{n,r}(\delta) \}_{r\in [K]}$ and $\{\mathcal{K}^{(+)}_{n,|\ell-r|^{-2}} \}_{\ell, r\in [K]}$. As a result, we can write the first term in the RHS of \eqref{eq:upp1.5} 
\begin{align*}
\prod_{a\in \{+,-\}} &\P\left(\bigcap_{r=\omega K}^{(1-\omega)K} \Big\{{\Xi}^{(a)}_{n,r}(\delta) \cup\Big( \bigcup_{\ell\in [K], \ell\ne r}\mathcal{K}^{(a)}_{n,|\ell-r|^{-2} C\delta}(\ell,r)\Big)\Big\}\right)
\\ &= \prod_{a\in \{+,-\}}\P\left(\bigcap_{r=\omega K}^{(1-\omega)K} \Big\{{\Xi}^{(a)}_{n,r}(\delta) \cup \mathcal{K}^{(a)}_{n,\delta}(r)\Big\}\right),
\end{align*}
where $\mathcal{K}^{(a)}_{n, \delta}(r):=\bigcup_{\ell\in [K], \ell\ne r}\mathcal{K}^{(a)}_{n,C\delta|\ell-r|^{-2}}(\ell,r)$. 
To bound the RHS of the above display, by union bound, we write 
\begin{align}
\P\left(\bigcap_{r=\omega K}^{(1-\omega)K} \Big\{{\Xi}^{(a)}_{n,r}(\delta) \cup \mathcal{K}^{(a)}_{n, \delta}(r)\Big\}\right)
&\leq  \sum_{\vec{\mathbf{r}}: \mathrm{length}(\vec{\mathbf{r}})\geq (1-3\omega)K}\mathbb{P}\Big(\bigcap_{r\in \vec{\mathbf{r}}}{\Xi}^{(a)}_{n,r}(\delta)\Big) \nonumber\\&+ \sum_{\vec{\mathbf{r}}:\mathrm{length}(\vec{\mathbf{r}})\geq \omega K} \mathbb{P}\Big(\bigcap_{r\in \mathfrak{J}_{\vec{\mathbf{r}}}}{\Xi}^{(a)}_{n,r}(\delta)\cap \bigcap_{r\in \vec{\mathbf{r}}} \mathcal{K}^{(a)}_{n,\delta}(r)\Big),
\label{eq:LastSplit}
\end{align}
where  $\mathfrak{J}_{\vec{\mathbf{r}}}:=[\omega K, (1-\omega)K]/\{\vec{\mathbf{r}}\}$. Notice that $\mathrm{length}(\vec{\mathbf{r}})$ is upper bounded by $K$. So the number of terms in the above sums are finite. 
We now proceed to bound the two terms in the right hand side of \eqref{eq:LastSplit}.

\begin{lem}\label{lem:clm}
(i) For any $h>0$, there exists a positive integer $n_0$ (depending on $M,\omega, h,\delta,\kappa,\alpha, $) such that for all $n\geq n_0$ we have
\begin{align}\label{eq:clm1}
\mathbb{P}\Big(\bigcap_{r\in \vec{\mathbf{r}}}{\Xi}^{(a)}_{n,r}(\delta)\Big) \le &(1+h)^{\mathrm{length}(\vec{\mathbf{r}})}  \P(\max_{b\in \{+,-\}}\sup_{t\in [0,  M^{1-2\kappa}]}Y^{(a)}_{0,M}(b,t)<\delta)^{|\vec{\mathbf{r}}|}.
\end{align}
Here $\{Y_{0,M}^{(a)}(b,.),b\in \{-,+\}\}$ are i.i.d.~centered Gaussian process with covariance Definition~\ref{def:Y}.
\\

(ii) For any $h>0$, there exists a positive integer $n_0$ (depending on $M,\omega, h,\delta,\kappa,\alpha$) and an absolute constant $C$ such that for all $n\geq n_0$ and $\vec{\mathbf{r}}\in [K]$ we have
\begin{align}\label{eq:clm2}
 \mathbb{P}\Big(\bigcap_{r\in \mathfrak{J}_{\vec{\mathbf{r}}}}{\Xi}^{(a)}_{n,r}(\delta)\cap \bigcap_{r\in \vec{\mathbf{r}}} \mathcal{K}^{(a)}_{n,\delta}(r)\Big)\le &  2[C(1+h)]^{|\tilde{\mathfrak{J}}_{\vec{\mathbf{r}},\lessgtr}|} \P(\sup_{t\in [1,M^{1-2\kappa}]}Y_{0,M}^{(a)}(b,t)<\delta)^{2 |\tilde{\mathfrak{J}}_{\vec{\mathbf{r}},\lessgtr}|}
\nonumber\\ & \times M^{-c_1\sum_{i\in[N]} (s_{\tilde{r}_i}-\tilde{r}_i)+c_2\omega K},
\end{align}
where $c_1=\alpha+1-3\delta, c_2=(\alpha+3)\delta$ if $a=-$, and $c_1=1, c_2=\delta$ if $a=+$. Here $|\tilde{\mathfrak{J}}_{\vec{\mathbf{r}},\lessgtr}|$ denotes $\max\{|\tilde{\mathfrak{J}}_{\vec{\mathbf{r}},>}|, |\tilde{\mathfrak{J}}_{\vec{\mathbf{r}},<}|\}$ where  $\tilde{\mathfrak{J}}_{\vec{\mathbf{r}},>}$ and $\tilde{\mathfrak{J}}_{\vec{\mathbf{r}},<}$ are defined as follows 
\begin{align*}
\tilde{\mathfrak{J}}_{\vec{\mathbf{r}},\lessgtr} := 
    \{r \in \vec{\mathbf{r}}: s_r\lessgtr r \text{ where } \mathcal{K}^{(a)}_{n, C\delta |s_r -r|^{-2}}(s_r,r) \text{ occurs}\}
\end{align*}
and they satisfy the relation 
\begin{align}\label{eq:mathfrakJ-s-r}
|\tilde{\mathfrak{J}}_{\vec{\mathbf{r}},>}|+\sum_{i\in  [N]}(s_{\tilde{r}_i}-\tilde{r}_i)+N=(1-\omega)K.
\end{align}
\end{lem}
Before proceeding to the proof of Lemma \ref{lem:clm} we first complete the proof of the upper bound. Owing to \eqref{eq:upp1}, \eqref{eq:upp1.5}, \eqref{eq:LastSplit}, \eqref{eq:clm1} and \eqref{eq:clm2} and Lemma~\ref{lem:KextBound}, we obtain 
\begin{align*}
    \mathbb{P}\Big(Q_n(x)<0, x\in \mathbb{R}\Big) &\leq \prod_{a\in \{-,+\}}(\mathfrak{C}^{(a)} +\mathfrak{D}^{(a)})
\end{align*}
where 
\begin{align*}
    \mathfrak{C}^{(a)}: &= (1+h)^K  |\vec{\mathbf{r}}: \mathrm{length}(\vec{\mathbf{r}})\geq K(1-3\omega)| \\ &\times\P\Big(\max_{b\in \{+,-\}}\sup_{t\in [1,  M^{1-2\kappa}]}Y^{(a)}_{0,M}(b,t)<\delta\Big)^{2(1-3\omega)K}  \\
    \mathfrak{D}^{(a)}: &=  2[C(1+h)]^{K}\\ &\times\prod_{a\in \{+,-\}}\sum_{\vec{\mathbf{r}}; |\vec{\mathbf{r}}|\geq \omega K} \P\Big(\sup_{t\in [1,M^{1-2\kappa}]}Y_{0,M}^{(a)}(b,t)<\delta\Big)^{2 |\tilde{\mathfrak{J}}_{\vec{\mathbf{r}},\lessgtr}|}
M^{-c_1\sum_{i\in[N]} (s_{\tilde{r}_i}-\tilde{r}_i)+c_2\omega K} 
\end{align*}
Taking logarithm on both sides and using the inequality $\log(\mathfrak{C}^{(a)}+\mathfrak{D}^{(a)})\leq  \log 2(\max\{\mathfrak{C}^{(a)},\mathfrak{D}^{(a)}\})$ to bound the logarithm of the right hand sides of the above display yields, 
\begin{align}\label{eq:FinalIneqForUpperBound}
    \log \mathbb{P}(Q_n(x)<0, x\in \mathbb{R}) \leq \sum_{a\in \{-,+\}}\Big(\log \max\Big\{\mathfrak{C}^{(a)},\mathfrak{D}^{(a)}\Big\} + \log 2\Big)
\end{align}
Now we divide both sides of \eqref{eq:FinalIneqForUpperBound} by $\log n$. Notice that 
\begin{align}\label{eq:mathfrakC}
\varlimsup_{n\to \infty}\frac{1}{n}\log \mathfrak{C}^{(a)} &= \varlimsup_{n\to \infty}\frac{K}{\log n}\log (1+h)\\ & + \varlimsup_{n\to \infty}\frac{1}{\log n} \log |\vec{\mathbf{r}}: \mathrm{length}(\vec{\mathbf{r}})\geq K(1-3\omega)| \\ &+ \varlimsup_{n\to \infty} \frac{(1-3\omega) K}{\log n} \log \P\Big(\max_{b\in \{+,-\}}\sup_{t\in [1,  M^{1-2\kappa}]}Y^{(a)}_{0,M}(b,t)<\delta\Big) \nonumber 
\end{align}
Since $K/\log n$ is arbitrarily small when $M$ gets large and furthermore, $|\vec{\mathbf{r}}: \mathrm{length}(\vec{\mathbf{r}})\geq K(1-3\omega)|\leq 2^K$, we have 
\begin{align*}
\varlimsup_{M\to \infty}&\varlimsup_{\delta \to 0}\varlimsup_{\kappa\to 0}\varlimsup_{n\to \infty}\frac{1}{\log n}\log \mathfrak{C}^{(a)}\\  & \leq \varlimsup_{M\to \infty}\varlimsup_{\delta \to 0}\varlimsup_{\kappa\to 0} \varlimsup_{n\to \infty} \frac{(1-3\omega)K}{\log n} \log \P\Big(\max_{b\in \{+,-\}}\sup_{t\in [1,  M^{1-2\kappa}]}Y^{(a)}_{0,M}(b,t)<\delta\Big) \\ & \leq (1-3\omega)\varlimsup_{M\to \infty}\frac{1}{\log M} \log  \P\Big(\max_{b\in \{+,-\}}\sup_{t\in [1,  M]}Y^{(a)}_{0,M}(b,t)<0\Big)\\  & \leq (1-3\omega)\begin{cases}\varlimsup_{M\to \infty}\frac{2}{\log M} \log  \P\Big(\sup_{t\in [1,  M]}Y^{(\alpha)}_{\log t}<0\Big) & \text{ if }a = - \\  \varlimsup_{M\to \infty}\frac{2}{\log M}\log  \P\Big(\sup_{t\in [1,  M]}Y^{(0)}_{\log t}<0\Big) \Big) & \text{ if }a = +
\end{cases}
\\ &= (1-3\omega) \begin{cases}
2b_{\alpha} & \text{ if }a = -\\
2b_0 & \text{ if } a= +
\end{cases}.
\end{align*}
The second to the last inequality follows since $\{Y^{(a)}_{0,M}(+,t):t \in \mathbb{R}\}$ and $\{Y^{(a)}_{0,M}(-,t): t\in \mathbb{R}\}$ are independent Gaussian processes, and, furthermore, as we show below, the covariance functions of $Y^{(a)}_{0,M}(-,t)$ and $Y^{(a)}_{0,M}(+,t)$ are bounded respectively by the covariance functions of $t^{-(\alpha+1)/2} Y^{(\alpha)}_{\log t}$ and $t^{-1/2}Y^{(0)}_{\log t}$. Here, $Y^{(\alpha)}_{t}$ and $Y^{(0)}_{ t}$ are the centered Gaussian processes same as in Theorem~\ref{thm:Main}. To see this, recall from Definition~\ref{def:Y} that 
\begin{align*}
\mathrm{Cov}(Y_{0,M}^{(a)}(b,t_1), Y_{0,M}^{(a)}(b,t_2))=h_{0,M}^{(a)}(t_1+t_2) = \begin{cases}
h_{0, M, \alpha} & \text{ if }a=-\\ 
h_{0, M, 0} & \text{ if }a=+
\end{cases}
\end{align*}
where 
$h_{0, M, \alpha} = \int^{M^{\delta}}_{M^{\delta -1}} x^{\alpha} e^{- xt} dt$. As $M\to \infty$, $\mathrm{Cov}(Y_{0,M}^{(a)}(-,t_1), Y_{0,M}^{(a)}(-,t_2))$ increases up to $(t_1+t_2)^{-(\alpha+1)}$ which is equal to $\mathrm{Cov}(t^{-(\alpha+1)/2}_1Y^{(\alpha)}_{\log t_1}, t^{-(\alpha+1)/2}_2Y^{(\alpha)}_{\log t_2})$. Similarly, $\mathrm{Cov}(Y_{0,M}^{(a)}(+,t_1), Y_{0,M}^{(a)}(+,t_2))$ increases up to $(t_1+t_2)^{-(\alpha+1)}$ which is equal to the covariance function  $\mathrm{Cov}(t^{-1/2}_1Y^{(0)}_{\log t_1}, t^{-1/2}_2Y^{(0)}_{\log t_2})$. Due to this, the second to the last inequality follows by the Slepian's inequality.

Since $\omega$ is arbitrary, letting $\omega$ to $0$ yields the right hand side of the above display to be $2b_{\alpha}$ when $a=-$ and $2b_{0}$ when $a=+$. On the other hand, we have 
\begin{align}\label{eq:mathfrakD}
\log \mathfrak{D}^{(a)} &\leq 2K \log C+ \log|\vec{\mathbf{r}}: |\vec{\mathbf{r}}|\geq \omega K|  \\ +\max_{\vec{\mathbf{r}}: |\vec{\mathbf{r}}|\geq \omega K} &\log \Big(\P\Big(\sup_{t\in [1,M^{1-2\kappa}]}Y_{0,M}^{(a)}(b,t)<\delta\Big)^{2 |\tilde{\mathfrak{J}}_{\vec{\mathbf{r}},\lessgtr}|}
M^{-c_1\sum_{i\in[N]} (s_{\tilde{r}_i}-\tilde{r}_i)+c_2\omega K} \Big)\nonumber
\end{align}
By using the same argument as in above (using Slepian's lemma), we can bound the probability $\P(\sup_{t\in [1,M^{1-2\kappa}]}Y_{0,M}^{(a)}(b,t)<0)^{2 |\tilde{\mathfrak{J}}_{\vec{\mathbf{r}},\lessgtr}|}$ by $\P(\sup_{t\in [1, M^{1-2\kappa}]} Y^{(\alpha)}_{\log t}<0)^{^{2 |\tilde{\mathfrak{J}}_{\vec{\mathbf{r}},\lessgtr}|}}$ when $a=-$ and by $\P(\sup_{t\in [1, M^{1-2\kappa}]} Y^{(0)}_{\log t}<0)^{^{2 |\tilde{\mathfrak{J}}_{\vec{\mathbf{r}},\lessgtr}|}}$ when $a=+$.

Following \cite[Lemma~2.5]{poonen1999cassels} and \cite{DemboMukherjee2015}, we know that 
\begin{align*}
\frac{\P \big(\sup_{t\in [1,M^{1-2\kappa}]}Y^{(0)}_{\log t}<0\big)}{ M^{-(1-\delta)(1-\kappa)/2}} \gtrsim 1,\qquad  \frac{\P\big(\sup_{t\in [1,M^{1-2\kappa}]}Y^{(\alpha)}_{\log t}<0\big)}{M^{-(\alpha+1-\delta)(1-\kappa)/2}} \gtrsim 1
\end{align*}
for $M$ large and $\kappa,\delta>0$ very small. On the other hand, both 
 $\P\big(\sup_{t\in [1,M^{1-2\kappa}]}Y^{(0)}_{\log t}<0\big)$ and $\P\big(\sup_{t\in [1,M^{1-2\kappa}]}Y^{(\alpha)}_{\log t}<0\big)$ is upper bounded by $M^{-\beta(1-2\kappa)/2}$ for all $M>0$ large and $\beta,\kappa>0$ small. Hence we get (first taking $\delta \to 0$),  
\begin{align*}
\varlimsup_{n\to \infty} &\frac{1}{\log n}\max_{\vec{\mathbf{r}}: |\vec{\mathbf{r}}|\geq \omega K}\log \Big(\P\Big(\sup_{t\in [1,M^{1-2\kappa}]}Y_{0,M}^{(a)}(b,t)<\delta\Big)^{2 |\tilde{\mathfrak{J}}_{\vec{\mathbf{r}},\lessgtr}|}
M^{-c_1\sum_{i\in[N]} (s_{\tilde{r}_i}-\tilde{r}_i)+c_2\omega K} \Big)\\
\leq \varlimsup_{n\to \infty} &\frac{1}{\log n} \max_{\vec{\mathbf{r}}: |\vec{\mathbf{r}}|\geq \omega K}\begin{cases}\log \Big( \mathbb{P}\Big(\sup_{t\in [1,M^{1-2\kappa}]}Y^{(\alpha)}_{\log t}<0\Big)^{\psi_{\vec{\mathbf{r}}}}\Big) & a= - \\
\log \Big( \mathbb{P}\Big(\sup_{t\in [1,M^{1-2\kappa}]}Y^{(0)}_{\log t}<0\Big)^{\psi_{\vec{\mathbf{r}}}}\Big) & a=+
\end{cases}
\\
 \leq \qquad &\begin{cases}
     - 2b_{\alpha} & a =-\\
      - 2b_0 & a=+
 \end{cases} 
 \quad \text{as } M\to \infty, \kappa, \omega\to 0
\end{align*}
where $\psi_{\vec{\mathbf{r}}}:=2(|\tilde{\mathfrak{J}}_{\vec{\mathbf{r}},\lessgtr}| +\frac{1}{1-2\kappa}(\sum_{i \in [N]}(s_{\tilde{r}_i} -\tilde{r}_i))-\frac{c_2}{\beta(1-2\kappa) }\omega K )$. The last inequality follows by combining \eqref{eq:mathfrakJ-s-r}  with also letting $\omega \to 0$.
Combining \eqref{eq:FinalIneqForUpperBound}, \eqref{eq:mathfrakC} and \eqref{eq:mathfrakD} yields that 
\[\varlimsup_{n\to \infty}\frac{1}{\log n} \log \mathbb{P}(Q_n(x)<0, x\in \mathbb{R})\leq - 2(b_{\alpha}+b_0).\]


\subsubsection{Proof of Lemma \ref{lem:clm}}
We first state the following definition and a lemma, which we will use to verify Lemma \ref{lem:clm}. We defer the proof of the lemma to Section~\ref{sec:Lemma_4.4_convergence}. 

\begin{defn}\label{def:Y}
Set for $a\in \{-,+\}$, $$h_{n,r,M }^{(a),\ell}(t):=\sum_{i\in \mathcal{I}_\ell^{(a)}} R(n) e^{it\tau},\text{ where }\tau:=\tau_n(r,M)=\frac{1}{n^\delta M^{r-\delta}}.$$
Also, for any integer $s\in \Z$ and $t\ge 0$ set $$h_{s,M,\alpha}(t):=\int_{M^{s+\delta-1}}^{M^{s+\delta}} x^\alpha e^{-xt} dx.$$

Fixing $a\in\{-,+\}, \ell, r \in [K]$, define a process on $\{-,+\}\times [1,M^{1-2\delta}]$ by setting
$$Y_{n,r,M}^{(a),\ell}(b,t):=\frac{Q_n^{(a),\ell}(b e^{at\tau})}{\sigma_n(at\tau)}.$$

Also let $\{Y_{s,M}^{(a)}(b,.)\}_{b\in \{+,-\}}$ be i.i.d.~centered Gaussian processes, with
\begin{align*}
Corr(Y_{s,M}^{(a)}(b,t_1), Y_{s,M}^{(a)}(b,t_2))=\frac{h_{s,M}^{(a)}(t_1+t_2)}{\sqrt{h_{s,M}^{(a)}(2t_1) h_{s,M}^{(a)}(2t_2)}} 
\end{align*}
where 
\begin{align*}
    h_{s,M}^{(a)}(t) := \begin{cases}
        h_{s,M,\alpha} (t) & \text{ if }a=-1,\\
        h_{s,M,0} (t) & \text{ if }a=0.
    \end{cases}
\end{align*}
\end{defn}

 \begin{lem}\label{lem:weak}
For any  $a\in \{+,-\}$ and $(r_n,\ell_n)\in [K]$ with $\ell_n-r_n\to s\in \Z$, we have 
$$\{Y_{n,r_n,M}^{(a),\ell_n}(b,t), b\in \{-,+\}, t\in [1,M^{1-2\delta]}\}\stackrel{d}{\to} \{Y_{s,M}^{(a)}(b,t), b\in \{-,+\}, t\in [1,M^{1-2\delta}]\}.$$
where the convergence is in the weak topology of $\mathcal{C}([1,M^{1-2\delta}]^{\otimes 2})$. 
\end{lem}

\begin{proof}[Proof of Lemma \ref{lem:clm}]
\begin{enumerate}
\item[(a)]

{\bf Proof of \eqref{eq:clm1}}

Recall the definition of the sets $\Xi^{(a)}_{n,r}(\delta)$ from \eqref{eq:backlog}. Using the independence of the sets $\{{\Xi}^{(a)}_{n,r}(\delta), r\in \vec{r}\}$  we have
$$\mathbb{P}\Big(\bigcap_{r\in \vec{\mathbf{r}}}{\Xi}^{(a)}_{n,r}(\delta)\Big)= \prod_{i=1}^{|\vec{\mathbf{r}}|}\mathbb{P}\Big(\bigcap_{r\in \vec{\mathbf{r}}}{\Xi}^{(a)}_{n,r}(\delta)\Big).$$
 Lemma \ref{lem:weak} implies that for any given  $a\in \{-,+\}$,  $$\max_{b\in \{-,+\}}\sup_{t\in [1,M^{1-2\kappa}]}Y_{n,r,M}^{(a),r}(b,t)\stackrel{d}{\to}\max_{b\in \{-,+\}}\sup_{t\in [1,M^{1-2\kappa}]}Y_{0,M}^{(a)}(b,t).$$ As a result, given any $h>0$ there exists $n_0$ depending on $(M, \omega, h)$, such that for all $n\ge n_0$ 
\begin{align*}
\mathbb{P}\Big(\Xi^{(a)}_{n,r}(\delta)\Big) &=\P\Big(\max_{b\in \{-,+\}}\sup_{t\in [1,M^{1-2\kappa}]}Y_{n,r,M}^{(a),r}(b,t)<\delta\Big)
\\ &\le (1+h) \P\Big(\max_{b\in \{+,-\}}\sup_{t\in [1,M^{1-2\kappa}]}Y_{0,M}^{(a)}(b,t)<\delta\Big).
\end{align*}
Combining the last two displays, the desired conclusion follows.
\\

\item[(b)]

{\bf Proof of \eqref{eq:clm2}}

We prove the lemma for $a=-$, noting that the proof of $a=+$ follows by a similar argument.
\\

For any set finite set $A$, we denote the number of elements in $A$ by $|A|$. If the event $\Big\{\bigcap_{r\in \mathfrak{J}_{\vec{\mathbf{r}}}}{\Xi}^{(a)}_{n,r}(\delta)\cap \bigcap_{r\in \vec{\mathbf{r}}} \mathcal{K}^{(a)}_{n,\delta}(r)\Big\}$ occurs, for every $r\in  \vec{\mathbf{r}}$ there exists $s_r\in [K]$ such that $\mathcal{K}^{(a)}_{n,\delta}(s_r,r)$ occurs. Also, since $\mathrm{length}(\vec{\mathbf{r}})\ge \omega K$, we either have $|r\in \vec{\mathbf{r}}:s_r>r|\ge \frac{\omega K}{2}$ or  $|r\in \vec{\mathbf{r}}:s_r<r|\ge \frac{\omega K}{2}$. Without loss of generality assume that we are in the first case. In this case we construct a subset $\widetilde{\mathcal{L}}_{\vec{\mathbf{r}},>}$ of the set $\mathcal{L}_{\vec{\mathbf{r}},>}:=\{r\in\vec{\mathbf{r}}:s_r>r\}$, using the following algorithm which we call as `\emph{move, flush and repeat}':

\begin{itemize}
\item 
\textbf{Move:} Pick the smallest element in the set $\mathcal{L}_{\vec{\mathbf{r}},>}$, say $r$, and put it in the subset $\widetilde{\mathcal{L}}_{\vec{\mathbf{r}},>}$.  By definition of $\mathcal{L}_{\vec{\mathbf{r}},>}$ we have $s_r>r$. 

\item
\textbf{Flush:} Remove all elements of $\mathcal{L}_{\vec{\mathbf{r}},>}$ in the interval $[r,s_r]$.

\item
\textbf{Repeat:} If $\mathcal{L}_{\vec{\mathbf{r}},>}$ is empty, stop and output the subset $\widetilde{\mathcal{L}}_{\vec{\mathbf{r}},>}$. If not, go back to step 1.

\end{itemize}

With this construction, we have $\widetilde{\mathcal{L}}_{\vec{\mathbf{r}},>}\subset {\mathcal{L}}_{\vec{\mathbf{r}},>}$. Let $N:=|\widetilde{\mathcal{L}}_{\vec{\mathbf{r}},>}|$, and let $\{\tilde{r}_1,\ldots,\tilde{r}_N\}$ denote the elements of $\widetilde{\mathcal{L}}_{\vec{\mathbf{r}},>}$. Also, set 
$\mathfrak{K}_{\vec{\mathbf{r}},>}:=\{s_{\tilde{r}_i}, 1\le i\le N\},$ and note that the set $\mathfrak{K}_{\vec{\mathbf{r}},>}$ necessarily consists of distinct elements in $[K]$. Thus the number of choices for the set $\mathfrak{K}_{\vec{\mathbf{r}},>}$ is at most $2^K$. Set
\begin{align}\label{eq:J-definition}
\tilde{\mathfrak{J}}_{\vec{\mathbf{r}},>}:=\bigcup_{i=0}^{N} [s_{\tilde{r}_i}+1,\tilde{r}_{i+1}],
\end{align}
where  $s_{\tilde{r}_0}:=\omega K/2$, and $\tilde{r}_{N+1}:=(1-\omega/2)K$.
Setting $c_1:=1+\alpha-\delta$ and $c_2:= (\alpha+3)\delta$ we have 
\begin{align}\label{eq:clm4}
\notag&\P\left(\bigcap_{r\in \mathfrak{J}_{\vec{\mathbf{r}}}}{\Xi}^{(a)}_{n,r}(\delta)\cap \bigcap_{r\in \vec{\mathbf{r}}} \mathcal{K}^{(a)}_{n,\delta}(r)\right)\\
\notag\le &2\sum_{ \mathfrak{K}_{\vec{\mathbf{r}},>}} \P\left( \bigcap_{r\in {\mathfrak{J}}_{\vec{\mathbf{r}},>}}{\Xi}^{(a)}_{n,r} \cap  \bigcap_{i=1}^N \mathcal{K}_{n,\delta}^{(a)}(s_{\tilde{r}_i},\tilde{r}_i)\right)\\
\notag= &2\sum_{ \mathfrak{K}_{\vec{\mathbf{r}},>}} \prod_{r\in   {\mathfrak{J}}_{\vec{\mathbf{r}},>}}\P({\Xi}^{(a)}_{n,r} )  \prod_{i=1}^N \P( \mathcal{K}_{n,\delta}^{(a)}(s_{\tilde{r}_i},\tilde{r}_i))\\
\notag\le & 2[(1+h)C]^{ |\tilde{\mathfrak{J}}_{\vec{\mathbf{r}},>}|} \P(\sup_{t\in [1,M^{1-2\kappa}]}Y_{M,0}^{(a)}(b,t)<\delta)^{2| |\tilde{\mathfrak{J}}_{\vec{\mathbf{r}},>}|}\\
\notag& \times \prod_{i\in  [N]:s_{\tilde{r}_i}-\tilde{r}_i\le \omega^{-1}}\frac{1}{M^{  c_1(s_{\tilde{r}_i}-\tilde{r}_i)+c_2}} 
\prod_{i\in [N]: s_{\tilde{r}_i}-\tilde{r}_i>\omega^{-1}}\frac{1}{M^{  c_1(s_{\tilde{r}_i}-\tilde{r}_i)-c_2}}
\end{align}
The right hand side of the last inequality could be further written as 
\begin{align*}
 & 2[(1+h)C]^{ |\tilde{\mathfrak{J}}_{\vec{\mathbf{r}},>}|} \P(\sup_{t\in [1, M^{1-2\kappa}]}Y_{0,M}^{(a)}(b,t)<\delta)^{2| |\tilde{\mathfrak{J}}_{\vec{\mathbf{r}},>}|} \\
 &\exp\Big(-\log M\Big[\sum_{i\in[N]:  s_{\tilde{r}_i}-\tilde{r}_i\le \omega^{-1}}[c_1( s_{\tilde{r}_i}-\tilde{r}_i)+c_2]\\ &+\sum_{i\in [N]:  s_{\tilde{r}_i}-\tilde{r}_i> \omega^{-1}}[c_1(s_{\tilde{r}_i}-\tilde{r}_i)-c_2]\Big]\Big).   
\end{align*}
The fits inequality in the above display follows from the union bound. The following equality holds true $\{{\Xi}^{(a)}_{n,r}(\delta): r\in \mathfrak{J}_{\vec{\mathbf{r}}}\}$ and $\{\mathcal{K}_{n,\delta}^{(a)}(s_{\tilde{r}_i},\tilde{r}_i): r\in \mathfrak{J}_{\vec{\mathbf{r}}}\}$ are mutually independent by the construction of set $\widetilde{\mathcal{L}}_{\vec{\mathbf{r}},>}$. The second inequality follows by noticing that $$\P(\Xi^{(a)}_{n,r}(\delta)) \to \P(\sup_{t\in [1,M^{1-2\kappa}]}Y_{M,0}^{(a)}(b,t)<\delta)$$ as $n\to \infty$. Here there exists $n_0\in \mathbb{N}$ depending on $(M,\omega, h)$ such that $\P(\Xi^{(a)}_{n,r}(\delta))$ is less than $(1+h)\P(\sup_{t\in [1,M^{1-2\kappa}]}Y_{M,0}^{(a)}(b,t)<\delta)$ for all $n\geq n_0$. Furthermore, $\P( \mathcal{K}_{n,\delta}^{(a)}(s_{\tilde{r}_i},\tilde{r}_i))$ are bounded by $M^{-( c_1(s_{\tilde{r}_i}-\tilde{r}_i)+c_2)}$ for $s_{\tilde{r}_i}-\tilde{r}_i\le \omega^{-1}$ by Lemma~\ref{lem:tail_bound} and bounded by $M^{ -(c_1(s_{\tilde{r}_i}-\tilde{r}_i)-c_2)}$ for $s_{\tilde{r}_i}-\tilde{r}_i\ge \omega^{-1}$ by Lemma~\ref{lem:ex_sup}.   
To bound the exponent in the RHS of \eqref{eq:clm4}, using $s_{\tilde{r}_0}=\omega K/2$, and $\tilde{r}_{N+1}=(1-\omega/2)K$ we note 
\begin{align}
|\tilde{\mathfrak{J}}_{\vec{\mathbf{r}},>}|+\sum_{i\in  [N]}(s_{\tilde{r}_i}-\tilde{r}_i)+N=(1-\omega)K.
\end{align}
As a result, we get
$$\sum_{i\in [N]:  s_{\tilde{r}_i}-\tilde{r}_i> \omega^{-1}}(  s_{\tilde{r}_i}-\tilde{r}_i)\le K,\text{ which implies }\Big| i\in [N]:  s_{\tilde{r}_i}-\tilde{r}_i> \omega^{-1}\Big|\le \omega K.$$
Combining the above two displays, the exponent of $\log M$ in the second term in the RHS of \eqref{eq:clm4} equals
\begin{align*}
&c_1\sum_{i\in [N]}(  s_{\tilde{r}_i}-\tilde{r}_i)-c_2\Big|i\in [N]:s_{\tilde{r}_i}-\tilde{r}_i\le   \omega^{-1}\Big|-c_2\Big|i\in [N]:  s_{\tilde{r}_i}-\tilde{r}_i>  \omega^{-1}\Big|\\
=&c_1\sum_{i\in [N]}( s_{\tilde{r}_i}-\tilde{r}_i)+ c_2N-2c_2\Big|i\in [N]:s_{\tilde{r}_i}-\tilde{r}_i> \omega^{-1}\Big|\\
\ge &c_1\sum_{i\in [N]}( s_{\tilde{r}_i}-\tilde{r}_i)-2c_2\omega K.
\end{align*}
Along with \eqref{eq:clm4}, this gives
\begin{align*}
&\P\left(\bigcap_{r\in \mathfrak{J}_{\vec{\mathbf{r}}}}{\Xi}^{(a)}_{n,r}(\delta)\cap \bigcap_{r\in \vec{\mathbf{r}}} \mathcal{K}^{(a)}_{n,\delta}(r)\right)\\
\le& 2[C(1+h)]^{ |\tilde{\mathfrak{J}}_{\vec{\mathbf{r}},>}|} \P\Big(\sup_{t\in [1,M^{1-2\kappa}]}Y_{0,M}^{(a)}(b,t)<\delta\Big)^{2| |\tilde{\mathfrak{J}}_{\vec{\mathbf{r}},>}|}\\ & \times 
M^{-c_1\sum_{i\in[N]} (s_{\tilde{r}_i}-\tilde{r}_i)+c_2\omega K},
\end{align*}
which completes the proof of part (ii).
\end{enumerate}
\end{proof}

\subsubsection{Proof of Lemma \ref{lem:weak}}\label{sec:Lemma_4.4_convergence}
Before proceeding to the proof Lemma~\ref{lem:weak}, we derive few facts about the process $Y^{(a),\ell}_{n,r, M}(b,t)$ (see Definition~\ref{def:Y}) in Lemma~\ref{lem:weak*}. We use these facts in our proof of Lemma~\ref{lem:weak}. We first state Lemma~\ref{lem:weak*} and use it prove Lemma~\ref{lem:weak} and thereafter, complete the proof of Lemma~\ref{lem:weak*}. In the proof of Lemma~\ref{lem:weak}, one needs tightness of the processes $Y^{(a),\ell}_{n,r, M}(b,t)$ which we prove in Lemma~\ref{lem:tightness}.

\begin{lem}\label{lem:weak*}

Set for $a\in \{-,+\}$. Recall $h_{n,r,M }^{(a),\ell}(t)$, $h_{s,M,\alpha}(t)$ and $\tau=\tau_n(r,M)$ from Definition~\ref{def:Y}.


\begin{enumerate}
\item[(1)] Consider the case $a=-$.

(i) For all $n$ large enough (depending on $\delta$ and $L(.)$) we have
 $$M^{-(|\ell-r|+1)\delta}\frac{L(n^\delta M^\ell)}{\tau^{\alpha+1} }h_{n,r,M}(t) \lesssim h_{n,r,M}^{(-),\ell}(t)\lesssim M^{(|\ell-r|+1)\delta}\frac{L(n^\delta M^\ell)}{\tau^{\alpha+1} }h_{\ell-r,M,\alpha}(t).$$

(ii)  For any positive integer $\kappa$  we have 
\begin{align}
\lim_{n\to\infty}\max_{r,\ell\in [K]: |r-\ell|\le \kappa}\left|\frac{\tau^{\alpha+1} h_{n,r,M}^{(-),\ell}(t)}{L(n^\delta M^r)h_{\ell-r,M,\alpha}(t)}-1\right|=0.
\end{align}

(iii) If $b_1=b_2$, for any positive integer $\kappa$ and $t_1,t_2\in[1,M^{1-2\delta}]$ we have
$$\lim_{n\to\infty}\max_{r,\ell\in [K], |r-\ell|\le \kappa}\left|\frac{Cov\Big( \frac{Q_n^{(-),\ell}(b_1 e^{at_1 \tau})}{\sigma_n(at_1 \tau)},  \frac{Q_n^{(-),\ell}(b_2 e^{at_2\tau})}{\sigma_n(at_2\tau)}\Big)}{C_{s,M,\alpha}(t_1,t_2)}-1\right|=0,$$
{ where } $C_{s,M,\alpha}(t_1,t_2):= \frac{h_{s,M,\alpha}(t_1+t_2)}{\sqrt{h_{0,M,\alpha }(2t_1) h_{0,M,\alpha }(2t_2)}}$ for $t_1,t_2\in [1,M^{1-2\delta}]$.

\item[(2)] Consider the case $a=+$.

 (i) For all $n$ large enough (depending on $\delta$ and $L(.)$) we have
 $$ h_{n,r,M}^{(+),\ell}(t)\asymp  \frac{R(n) e^{nt\tau}}{\tau} h_{\ell-r,M,0}(t).$$

(ii) For any $M>0$,  $t\in [1,M^{1-2\delta}]$ and positive integer $\kappa$  we have 
$$\lim_{n\to\infty}\max_{r,\ell\in [K]: |r-\ell|\le \kappa}\left|\frac{\tau h_{n,r,M}^{(+),\ell}(t)}{R(n)e^{nt\tau}h_{\ell-r,M,0}(t)}-1\right|=0.$$

(iii) If $b_1=b_2$, for any positive integer $\kappa$ and $t_1,t_2\in [1,M^{1-2\delta}]$ we have
$$\lim_{n\to\infty}\max_{r,\ell\in [K]:|r-\ell|\le \kappa}\left|\frac{Cov\Big( \frac{Q_n^{(+),\ell}(b_1 e^{at_1 \tau})}{\sigma_n(at_1 \tau)},  \frac{Q_n^{(+),\ell}(b_2 e^{at_2\tau})}{\sigma_n(at_2\tau)}\Big)}{C_{\ell-r,M,0}(t_1,t_2)}-1\right|=0. $$

\item[(3)] If $b_1\ne b_2$ then for any $t_1,t_2\in [1,M^{1-2\delta}]$ and positive integer $\kappa$ we have
$$\lim_{n\to \infty}\max_{r,\ell\in [K]:|r-\ell|\le \kappa}\Big|Cov\Big( \frac{Q_n^{(a),\ell}(b_1 e^{at_1 \tau})}{\sigma_n(at_1 \tau)},  \frac{Q_n^{(a),\ell}(b_2 e^{at_2\tau})}{\sigma_n(at_2\tau)}\Big)\Big|= 0.$$

\end{enumerate} 

\end{lem}

\begin{proof}[Proof of Lemma \ref{lem:weak}]
We show the desired convergence of the stochastic process by checking convergence of finite dimensional distributions and tightness.

 \noindent{\bf Step 1 - Convergence of finite dimensional distributions:}
For showing convergence of finite dimensional distributions, 
we will show that for any real vector $(t_1,\cdots,t_{k})\in [1,M^{1-2\delta}]^{k}$ we have
$$\{Y_{n,r_n,b,M}^{(a),\ell_n}(t_1),\cdots,{Y}_{n,r_n,b,M}^{(a),\ell_n}(t_k), b\in \{+,-\}\}\stackrel{D}{\rightarrow}\{Y_{s,b,M}^{(a)}(t_1),\cdots,Y_{s,b,M}^{(a)}(t_k), b\in \{+,-\}\}.$$
For this, fixing  $\gamma:=(\gamma_{1,b},\cdots,\gamma_{k,b}, b\in \{+,-\})\in\R^{2k}$  it suffices to show that
\begin{align*}
A_n(\gamma) &:=\sum_{b\in \{+,-\}}\sum_{d=1}^k \gamma_{d,b} Y_{n,r_n,b,M}^{(a),\ell_n}(t_d)\\ &\stackrel{D}{\rightarrow} N\Big(0,\sum_{b\in \{+,-\}} \sum_{d,d'=1}^k \gamma_{d,b} \gamma_{d',b} Cov(Y_{s,b,M}^{(a)}(t_d), Y_{s,b,M}^{(a)}(t_{d'}))\Big).
\end{align*}
To this effect, note that
\begin{align*}
A_{n}(\gamma)=&\sum_{b\in \{+,-\}}\sum_{d=1}^k \gamma_{d,b}\sum_{i\in {\mathcal{I}}_{\ell_n}^{(a)} }\frac{\xi_i b^i e^{iat_d \tau}}{\sigma_n(at_d \tau(r_n,M))}\\
=&\sum_{i\in {\mathcal{I}}_{\ell_n}^{(a)} } \xi_i \left[\sum_{b\in \{+,-\}}\sum_{d=1}^k \gamma_{d,b}\frac{b^i e^{iat_d\tau_n(r_n,M)}}{\sigma_n(at_d \tau_n(r_n,M))}\right]\\
=&\sum_{i\in {\mathcal{I}}_{\ell_n}^{(a)} } \xi_iV_n(i), \quad V_n(i):= \left[\sum_{b\in \{+,-\}}\sum_{d=1}^k \gamma_{d,b}\frac{b^i e^{iat_d\tau_n(r_n,M)}}{\sigma_n(at_d \tau_n(r_n,M))}\right]
\end{align*}
 is a sum of independent components with mean $0$. To verify the convergence in distribution of $A_{n}(\gamma)$ we will use the Lindeberg Feller CLT.  Since
\begin{align*}
Var(A_{n}(\gamma))
=& \sum_{b,b'\in \{+,-\}} \sum_{d,d'=1}^k \gamma_{d,b}\gamma_{d',b'} Cov\Big( \frac{Q_n^{(a),r_n}(b e^{t_d \tau})}{\sigma_n(at_d \tau_n(r_n,M))},  \frac{Q_n^{(a),r}(b' e^{t_{d'}\tau_n(r_n,M)})}{\sigma_n(at_{d'}\tau)}\Big),
\end{align*}
on taking $n\to\infty$ and using parts (a), (b), (c) of Lemma \ref{lem:weak*} we get
\begin{align}\label{eq:var}
\lim_{n\rightarrow\infty}Var(A_n(\gamma))=\sum_{b\in \{+,-\}} \sum_{d,d'=1}^k \gamma_{d,b} \gamma_{d',b} Cov\Big( Y^{(a)}_{s,b,M}(t_d), Y^{(a)}_{s,b,M}(t_{d'})\Big).
\end{align} 
To complete the proof, it suffices to verify the Lindeberg Feller condition, which in this case is equivalent to verifying that for every $\psi>0$ we have
\begin{align*}
\lim_{n\rightarrow\infty}\sum_{i \in {\mathcal{I}}_{\ell_n}^{(a)}}R(i)V_n(i)^2\E \tilde{\xi_i}^2 1\{|\tilde{\xi}_i|\sqrt{R(i)}V_n(i)>\psi\}=0,
\end{align*}
where $\tilde{\xi}_i:=\frac{\xi_i}{\sqrt{R(i)}}$ has mean $0$ and variance $1$.
We now claim that 
\begin{align}\label{eq:psi}
\psi_n:=\max_{ i\in {\mathcal{I}}_{\ell_n}^{(a)}}|V_n(i)|\sqrt{R(i)}\rightarrow 0.
\end{align}

Before proceeding to the proof of \eqref{eq:psi}, let us show how this this claim proves the Lindeberg Feller condition. Notice that 
\begin{align*}
\sum_{i\in {\mathcal{I}}_{\ell_n}^{(a)}} & V_n(i)^2R(i)\E \tilde{\xi}_i^2 1\{|\tilde{\xi}_iV_n(i) \sqrt{R(i)}|>\psi\} \\ &\le \Big( \max_{i\in {\mathcal{I}}_{\ell_n}^{(a)}}\E \tilde{\xi_i}^2 1\Big\{|\tilde{\xi}_i|>\frac{\psi}{\psi_n}\Big\} \Big)\Big(\sum_{i\in {\mathcal{I}}_{\ell_n}^{(a)}} V_n(i)^2R(i)\Big).
 \end{align*}
 By our assumption, $\{\tilde{\xi}^2_i\}_{i\geq 0}$ are uniformly integrable which implies that $$\max_{i\in {\mathcal{I}}_{\ell_n}^{(a)}}\E \tilde{\xi_i}^2 1\Big\{|\tilde{\xi}_i|>\frac{\psi}{\psi_n}\Big\} \to 0$$ as $n\to \infty$. To show the Lindeberg Feller condition holds, it suffices now to prove that $\sum_{i\in {\mathcal{I}}_{\ell_n}^{(a)}} V_n(i)^2R(i)$ has a finite limit as $n\to \infty$. Notice that (a), (b), (c) of Lemma \ref{lem:weak*} implies $\mathrm{Var}(A_n(\gamma)) = \sum_{i\in {\mathcal{I}}_{\ell_n}^{(a)}} V_n(i)^2R(i)$ converges to a finite limit. This proves the Lindeberg-Feller condition. It remains to show \eqref{eq:psi}.


For verifying \eqref{eq:psi}, we split the proof into two cases, depending on the value of $a$. If $a=-$, using the regular variation of $R(.)$ along with part (b) of Lemma \ref{lem:weak*}, for any $i\in {\mathcal{I}}_{\ell_n}^{(-)}$ we have 
\begin{align*}
|V_n(i)\sqrt{R(i)}|\lesssim &  \tau_n(r_n,M)^{\frac{\alpha+1}{2}}i^{\frac{\alpha}{2}}\sqrt{\frac{L(i)}{L(1/\tau)}}\sum_{d=1}^k e^{-i\tau t_d }\lesssim \tau_n(r_n,M)^{1/2},
\end{align*}
where we use the fact that $i\tau_n(r_n,M) \asymp 1$.
Similarly, if $a=+$,  using the regular variation of $R(.)$ along with part (c) of Lemma \ref{lem:weak*}, for any $i\in {\mathcal{I}}_{\ell_n}^{(+)}$ we have %
\begin{align*}
|V_n(i)\sqrt{R(i)}|\lesssim &  \tau_n(r_n,M)^{1/2} \sqrt{\frac{R(i)}{R(n)}}\sum_{d=1}^k e^{(i-n)\tau_n(r_n,M) t_d }\lesssim \tau_n(r_n,M)^{1/2}.
\end{align*}
Since $\tau_n(r_n,M)\le n^{-\delta}$, on combining the above two displays \eqref{eq:psi} follows.
\\

%

{\bf Step 2 - Tightness:}
Proceeding to show tightness in $\cC[0,M^{1-2\delta}]^{\otimes 2}$ we will invoke the Kolmogorov-Chentsov criterion, for which it suffices to show that for any $t_1,t_2$ in this interval we have
$$\E[Y_{n,r_n,b,M}^{(a),\ell_n}(t_1)-Y_{n,r_n,b,M}^{(a),\ell_n}(t_2]^2\lesssim_M (t_1-t_2)^2.$$
But this follows from Lemma \ref{lem:tightness} for any $\ell_n,r_n$.

\end{proof}

\begin{proof}[Proof of Lemma \ref{lem:weak*}]

 To begin, for any $t_1,t_2\in [1, M^{1-2\delta}]$, we have 
\begin{align}\label{eq:corr_pn}
\notag Cov\Big( \frac{Q_n^{(a),\ell}(b_1 e^{at_1 \tau})}{\sigma_n(at_1 \tau)},  \frac{Q_n^{(a),\ell}(b_2 e^{at_2\tau})}{\sigma_n(at_2\tau)}\Big)=&\frac{\sum_{i\in \mathcal{I}_\ell^{(a)}} R(i) (b_1 b_2)^i e^{ia(t_1+t_2)\tau}}{\sqrt{\sum_{i\in \mathcal{I}_r^{(a)}} R(i)e^{2iat_1\tau}}\sqrt{\sum_{i\in \mathcal{I}_r^{(a)}} R(i) e^{2iat_2\tau}}}\\
= \frac{\sum_{i\in \mathcal{I}_\ell^{(a)}} R(i) (b_1b_2)^i e^{ia(t_1+t_2)\tau}}{\sum_{i\in \mathcal{I}_\ell^{(a)}} R(i)  e^{ia(t_1+t_2)\tau}}& \frac{h_{n,r,M}^{(a),\ell}(t_1+t_2)}{\sqrt{h_{n,M,r,\alpha}^{(a),r}(2t_1) h_{n,M,r,\alpha}^{(a),r}(2t_2)}}.
\end{align}

Now we proceed to prove the claim made in the parts (1), (2), (3) of Lemma~\ref{lem:weak*}.

\begin{enumerate}
\item[(1)] 
We start by recalling that $$\mathcal{I}_\ell^{(-)}=\mathbb{Z}\cap [n^\delta M^{\ell-1}, n^\delta M^{\ell})\subset [n^\delta M^{\ell-1}, n^\delta M^{\ell}), \qquad \tau=\frac{1}{n^\delta M^{r-\delta}}.$$ Writing $R(i)=L(i)i^\alpha$ with $L(.)$ slowly varying, fixing $\varepsilon>0$ for all $n$ large (depending on $\delta$, and the function $L(.)$) we have
\begin{align}\label{eq:lslow}
M^{-(|\ell-r|+1)\omega}\le \frac{L(i)}{L(n^\delta M^r)}\le M^{(|\ell-r|+1)\omega}, \quad\quad \forall i \in \mathcal{I}_\ell^{(-)}.
\end{align}
for some $\omega>0$ which is close to $0$ as $n$ approaches to $\infty$. 
Also, noting that 
\begin{align}
\lim_{n\to\infty}\sup_{M\ge 1}\sup_{ t\in [1,M^{1-2\delta}]}\sup_{x\in [i-1,i+1]} \max_{r,\ell\in [K]}\max_{ i\in [n^\delta M^{\ell-1}, n^\delta M^\ell)}\Big| \frac{x^\alpha e^{-xt\tau}}{i^\alpha e^{-it\tau}}-1\Big|=0
\end{align}
shows us 
\begin{align}\label{eq:calculus}
\lim_{n\to\infty}\sup_{M\ge 1}\sup_{ t\in [1,M^{1-2\delta}]} \max_{r,\ell\in [K]}\left|\frac{\sum_{i\in [n^\delta M^{\ell-1}, n^\delta M^\ell)} i^\alpha e^{-it\tau}}{\int_{n^\delta M^{\ell-1}}^{n^\delta M^\ell} x^\alpha e^{-xt\tau}dx}-1\right|=0.
\end{align}
Finally, a change of variable gives
\begin{align}\label{eq:change}
\int_{n^\delta M^{\ell-1}}^{n^\delta M^\ell} x^\alpha e^{-xt\tau}dx=\frac{1}{\tau^{\alpha+1}}  \int_{ M^{\ell-r+\delta-1}}^{M^{\ell-r+\delta}} x^{\alpha} e^{-xt } dx=\frac{1}{\tau_{\alpha+1}} h_{\ell-r,M,\alpha}(t).
\end{align}

Combining \eqref{eq:lslow}, \eqref{eq:calculus} and \eqref{eq:change}, we get that for all $n$ large (depending only on $\varepsilon,\delta $ and $L(.)$) , any $M\ge 1, t\in [1,M^{1-2\delta}], r,\ell\in [K]$ we have
\begin{align*}
(1-\varepsilon) M^{-(|\ell-r|+1)\omega} \le \frac{\tau^{\alpha+1}h_{n,r,M}^{(-),\ell}(t)}{L(n^\delta M^r) h_{\ell-r,M,\alpha}(t)}
\le (1+\varepsilon) M^{(|\ell-r|+1)\omega}.
\end{align*}
The conclusion of part (i) now follows from the above inequalities since $\omega$ approaches to $0$ as $n$ goes to $\infty$. 

The conclusion of part (ii) follows on noting that $|r-\ell|\le \kappa$ stays bounded, and $M$ stays fixed. The conclusion of part (iii) follows on noting that if $b_1=b_2$ the first term in the RHS of \eqref{eq:corr_pn} equals 1, and taking ratios using part (ii).

%
%

\item[(2)] 
Recall that $$\mathcal{I}_r^{(+)}=\mathbb{Z}\cap  (n-n^\delta M^\ell, n-n^\delta M^{\ell-1}]\subset(n-n^{1-\delta}, n-n^\delta],$$ and use the fact that $R(.)$ is regularly varying to conclude that
\begin{align}\label{eq:r_reg}
\lim_{n\to\infty}\max_{r,\ell\in [K]}\max_{i\in (n-n^\delta M^\ell, n-n^\delta M^{\ell-1}]}\Big|\frac{R(i)}{R(n)}-1\Big|=0.
\end{align}
Also, note that
\begin{align}\label{eq:calculus2}
 \sum_{i \in (n-n^\delta M^\ell, n-n^\delta M^{\ell-1}]} e^{it\tau}=e^{nt\tau}\sum_{i\in [n^\delta M^{\ell-1},n^\delta M^\ell)} e^{-it\tau}.
\end{align}
Combining \eqref{eq:r_reg}, \eqref{eq:calculus2}, and \eqref{eq:calculus} with $\alpha=0$, for all $n$ large (depending only on $\varepsilon,\delta $ and $L(.)$), any $M\ge 1, t\in [1,M^{1-2\delta}], r,\ell\in [K]$ we have
\begin{align*}
(1-\varepsilon)  \le \frac{\tau h_{n,r,M}^{(+),\ell}(t)}{R(n) e^{nt\tau} h_{\ell-r,M,\alpha}(t)}
\le (1+\varepsilon).
\end{align*}
As before, all the three conclusions (i), (ii) and (iii) follow from the above display.


\item[(3)] If $b_1\ne b_2$, the first term in the RHS of \eqref{eq:corr_pn} converges to $0$, on using the fact that $$\lim_{n\to\infty}\max_{\ell \in [K]}\max_{i\in \mathcal{I}_\ell^{(a)}} \Big|\frac{R(i)}{R(i-1)}-1\Big|= 0.$$
The desired conclusion follows, on noting that the second term converges to a finite number as $n\to\infty$, on invoking parts (a) and (b).

\end{enumerate}

\end{proof}

The next two lemmas aims to show the tightness (Lemma~\ref{lem:tightness}) and the tail probabilities (Lemma~\ref{lem:ex_sup}) of the processes $Y^{(a), \ell_n}_{n,r,M}(b,t)$ for $a,b\in \{+,-\}$. The latter is shown using Proposition~\ref{ppn:gen} which derives a maximal inequality of a stochastic process based on the pointwise bound on the second moments of the stochastic process.  Finally the tail probabilities of $Y^{(a), \ell_n}_{n,r,M}(b,t)$ are used to derive bounds $ \P(\mathcal{K}_{n,\delta}^{(-)}(\ell,r))$ (recall the definitions of $\mathcal{K}_{n,\delta}^{(-)}(\ell,r)$ from \eqref{eq:backlog}) in Lemma~\ref{lem:tail_bound}, a crucial input needed to complete of the proof of the upper bound in Theorem~\ref{thm:Main}.

\begin{lem}\label{lem:tightness}

Setting ${Y}_{n,r, b,M}^{(a),\ell_n}(t)=\frac{Q_n^{(a),\ell}(b e^{at\tau})}{\sigma_n(at\tau)}$ for $t\in [1,M^{1-2\delta}]$, for all $n,M$ large enough (depending only on $\delta$) we have the following bounds:
\begin{align*}
\E\Big[{Y}_{n,r, b,M}^{(-),\ell}(t_1)-{Y}_{n,r, b,M}^{(-),\ell}(t_2)\Big]^2\lesssim & M^{-|\ell-r|(\alpha+1-3\delta)+(\alpha+3)\delta}(t_1-t_2)^2,\\
\E\Big[{Y}_{n,r, b,M}^{(+),\ell}(t_1)-{Y}_{n,r, b,M}^{(+),\ell}(t_2)\Big]^2\lesssim & M^{(\delta-|\ell-r|)} (t_1-t_2)^2.
\end{align*}

\end{lem}

\begin{proof}[Proof of Lemma \ref{lem:tightness}]

Recalling that $\sigma_n^2(at\tau)=\sum_{j\in \mathcal{I}_r^{(a)}} R(j) e^{2jat\tau}$, 
a direct computation gives
\begin{align}\label{eq:tight1}
\notag&\E\Big[{Y}_{n,r, b,M}^{(a),\ell}(t_1)-{Y}_{n,r,b,M}^{(a),\ell}(t_2)\Big]^2\\
\notag=&\sum_{i\in \mathcal{I}_{\ell}^{(a)}} R(i)\Big[ \frac{e^{iat_1 \tau}}{\sigma_n(at_1\tau)}-\frac{e^{iat_2 \tau}}{\sigma_n(at_2\tau)}\Big]^2\\
\notag=&(t_1-t_2)^2 \sum_{i\in \mathcal{I}_\ell^{(a)}}\left[\frac{\sigma_n(a\zeta \tau) ia\tau e^{ia\zeta \tau}-e^{ia\zeta \tau}a\tau\frac{2\sum_{j\in \mathcal{I}_r^{(a)}} jR(j) e^{2ja\zeta \tau}}{2\sigma_n(a\zeta \tau)}}{\sigma_n(a\zeta\tau)^2}\right]^2\\
=&\frac{(t_1-t_2)^2\tau^2}{\sigma_n(a\zeta \tau)^2}\sum_{i\in \mathcal{I}_{\ell}^{(a)}} R(i) e^{2ia\tau \zeta} \left[i-\frac{\sum_{j\in \mathcal{I}_{r}^{(a)}} R(j) j  e^{2ja\tau \zeta}}{\sigma_n(a\zeta \tau)^2}\right]^2.
\end{align}
where in the second equality we have used the mean value theorem. Note that $\zeta \in (t_1, t_2)$.
We now analyze the right hand side of \eqref{eq:tight1} by splitting the argument depending on whether $a=-$ or $a=+$.
\vspace{0.5cm}

\noindent {\bf Case: $a= -$} In this case Lemma \ref{lem:weak*} part (1)(i) shows that for all $n$ large enough (depending on $L(.)$ and $\delta$), 
\begin{align*}
\sum_{j\in \mathcal{I}_{r}^{(-)}}R(j) je^{-2j\tau \zeta}&\lesssim  M^{(|\ell-r|+1)\delta} \frac{L(n^\delta M^r)}{\tau^{\alpha+2}}h_{0,M,\alpha+1}^{(-)}(2\zeta),\\  
 \sigma_n^2(-\zeta \tau) &\gtrsim M^{-\delta} \frac{L(n^\delta M^{r})}{\tau^{\alpha+1}}h_{0,M,\alpha}^{(-)}(2\zeta),
\end{align*}
which on taking ratio gives
\begin{align*}
\frac{\sum_{j\in \mathcal{I}_{r}^{(-)}}R(j) je^{-2j\tau \zeta}}{\sigma_n^2(a\tau\zeta)} \lesssim  \frac{M^{(|\ell-r|+2)\delta}}{\tau}\frac{h_{0,M,\alpha+1}^{(-)}(2\zeta)}{h_{0,M,\alpha}^{(-)}(2\zeta)}\lesssim_\alpha \frac{M^{(|\ell-r|+2)\delta}}{\tau\zeta}, 
\end{align*}
where the last inequality uses the estimate
\begin{align}\label{eq:estimate}
{h_{0,M,k}(2\zeta)}=\int_{M^{\delta-1}}^{M^\delta} e^{-2x\zeta} x^{k}dx\asymp_k  \frac{1}{\zeta^{k+1}}
\end{align} for
$\zeta\in [1,M^{1-2\delta}], M\ge 1$, for any $k>-1$, for the particular choices $k=\alpha,\alpha+1$.
Along with \eqref{eq:tight1}, this gives
\begin{align}\label{eq:kc1}
\notag&\frac{\E\Big[{Y}_{n,r,b,M}^{(-),\ell}(t_1)-{Y}_{n,r,b,M}^{(-),\ell}(t_2)\Big]^2}{(t_1-t_2)^2}\\
\notag\lesssim& \frac{ \tau^{\alpha+3} }{L(n^\delta M^{r}) h_{0,M,\alpha}^{(-)}(2\zeta)} \sum_{i\in \mathcal{I}_{\ell}^{(-)}} R(i) e^{-2i\tau \zeta} \Big[i+\frac{M^{(|\ell-r|+2)\delta}}{\tau \zeta}\Big]^2\\
\notag\lesssim & \frac{ L(n^\delta M^r)\tau^{\alpha+3} M^{|\ell-r|\delta}}{L(n^\delta M^{r}) h_{0,M,\alpha}^{(-)}(2\zeta)}\Big[\frac{h_{\ell-r,M,\alpha+2}(2\zeta)}{\tau^{\alpha+3}} +\frac{M^{(|\ell-r|+2)2\delta} h_{\ell-r,M,\alpha+2}(2\zeta)}{\tau^{\alpha+3}\zeta^2}\Big]\\
\notag\lesssim & \frac{M^{2|\ell-r|\delta}}{ h_{0,M,\alpha}^{(-)}(2\zeta)}\Big[h_{\ell-r,M,\alpha+2}(2\zeta) +M^{(|\ell-r|+2)2\delta} h_{\ell-r,M,\alpha+2}(2\zeta)\Big]\\
\lesssim & M^{2|\ell-r|\delta}\zeta^{\alpha+1}\Big[h_{\ell-r,M,\alpha+2}(2\zeta)+M^{(|\ell-r|+2)2\delta} h_{\ell-r,M,\alpha+2}(2\zeta)\Big] .
\end{align}
We obtained the second inequality by using Using Lemma \ref{lem:weak*} part (a)(i), the third inequality by using \eqref{eq:lslow}, and $\zeta\ge 1$ and the last inequality by  \eqref{eq:estimate}.

We now consider $2$ cases, depending on the relative values of $(\ell,r)$.
\vspace{0.5cm}

\noindent {\bf Sub-case: $\ell\le r$} A change of variable gives that for any $k>-1$ we have 
 \begin{align}\label{eq:kc3}
 h_{\ell-r,M,k}(2\zeta)\le \int_{0}^{M^{\ell-r+\delta}} x^{k}dx
 \lesssim_k M^{(\ell-r+\delta)(k+1)}.
 \end{align}
 Combining \eqref{eq:estimate} and \eqref{eq:kc3} with $k=\alpha,\alpha+2$, and using the bound $\zeta\le M^{1-2\delta}$ the RHS of \eqref{eq:kc1} can be bounded by
 \begin{align*}
&M^{2(r-\ell)\delta+(1-2\delta)(\alpha+1)}[M^{(\ell-r+\delta)(\alpha+3)}+M^{(r-\ell+2)2\delta }M^{(\ell-r+\delta)(\alpha+1)}]\\
 \lesssim &M^{2(r-\ell)\delta+(r-\ell+2)\delta +(\ell-r+\delta)(\alpha+1)}
 =M^{(\ell-r)(\alpha+1-3\delta)+(\alpha+3)\delta}.
 \end{align*}
%
and so the conclusion of the lemma holds in this case.
\vspace{0.5cm}

\noindent {\bf Sub-case: $\ell>r$}
To this effect, for any $k>-1$ we have
 \begin{align}\label{eq:kc3.1}
 h_{\ell-r,M,k}(2\zeta)\lesssim_k \int_{M^{\ell-r+\delta-1}}^{\infty} e^{-x\zeta} dx\le e^{-M^{\ell-r+\delta-1}}\le \frac{1}{M^{(k+1)(\ell-r-\delta)}},
 \end{align}
 where the last inequality holds for all $M$ large enough (depending only on $k,\delta$). Combining \eqref{eq:estimate} and \eqref{eq:kc3.1} with $k=\alpha,\alpha+2$, the RHS of \eqref{eq:kc1} can be bounded by
\begin{align*}
M^{2(\ell-r)\delta+(1-2\delta)(\alpha+1)} &[M^{(\alpha+3)(r-\ell+\delta)}+M^{(r-\ell+2)2\delta+(r-\ell+\delta)(\alpha+1)}]\\ &\lesssim  M^{(r-\ell)(\alpha+1-3\delta)+(\alpha+3)\delta},
\end{align*}
%
and so again the conclusion of the lemma holds in this case.
\vspace{0.5cm}

\noindent {\bf Case: $a=+$} As before, we need to bound the RHS of \eqref{eq:tight1}. Note that
\begin{align*}
\frac{\sum_{j\in \mathcal{I}_{r}^{(+)}} R(j) j  e^{2j\tau \zeta}}{\sigma_n(a\zeta \tau)^2}=\frac{\sum_{j\in \mathcal{I}_{r}^{(-)}} R(n-j) (n-j)  e^{-2j\tau \zeta}}{\sum_{j\in \mathcal{I}_{r}^{(-)}} R(n-j)  e^{-2j\tau \zeta}}=n-\frac{\sum_{j\in \mathcal{I}_{r}^{(-)}} R(n-j) j  e^{-2j\tau \zeta}}{\sum_{j\in \mathcal{I}_{r}^{(-)}} R(n-j)  e^{-2j\tau \zeta}},
\end{align*}
which gives the following bound to the RHS of \eqref{eq:tight1} (without the factor $(t_1-t_2)^2$):
\begin{align}\label{eq:kc5}
\notag&\frac{\tau^2}{\sigma_n(\zeta \tau)^2}\sum_{i\in \mathcal{I}_{\ell}^{(+)}} R(i) e^{2i\tau \zeta} \left[i-\frac{\sum_{j\in \mathcal{I}_{r}^{(a)}} R(j) j  e^{2j\tau \zeta}}{\sigma_n(\zeta \tau)^2}\right]^2\\
\notag=&\frac{\tau^2}{\sigma_n(\zeta \tau)^2}\sum_{i\in \mathcal{I}_{\ell}^{(+)}} R(i) e^{2i\tau \zeta} \left[n-i-\frac{\sum_{j\in \mathcal{I}_{r}^{(-)}} R(n-j) j  e^{-2j\tau \zeta}}{\sum_{j\in \mathcal{I}_{r}^{(-)}} R(n-j)  e^{-2j\tau \zeta}}\right]^2\\
\notag=&\frac{\tau^2e^{2n\tau \zeta}}{\sigma_n(\zeta \tau)^2}\sum_{i\in \mathcal{I}_{\ell}^{(-)}} R(n-i) e^{-2i\tau\zeta} \left[i-\frac{\sum_{j\in \mathcal{I}_{r}^{(-)}} R(n-j) j  e^{-2j\tau \zeta}}{\sum_{j\in \mathcal{I}_{r}^{(-)}} R(n-j)  e^{-2j\tau \zeta}}\right]^2\\
\lesssim&\frac{\tau^2e^{2n\tau  \zeta}R(n)}{\sigma_n(\zeta \tau)^2}\sum_{i\in \mathcal{I}_{\ell}^{(-)}}  e^{-2i\tau\zeta} \left[i+\frac{\sum_{j\in \mathcal{I}_{r}^{(-)}}  j  e^{-2j\tau \zeta}}{\sum_{j\in \mathcal{I}_{r}^{(-)}}  e^{-2j\tau \zeta}}\right]^2{\footnotesize\text{ [Using regular variation of }R(.)]}.
\end{align}
Using Lemma \ref{lem:weak*} part (b)(i) along with \eqref{eq:estimate} gives \begin{align}\label{eq:split1}
\sigma_n(\zeta\tau)^2=\sum_{j\in \mathcal{I}_r^{(+)}} R(j) e^{2j\zeta \tau}\asymp \frac{R(n) e^{2n\zeta \tau}}{\tau} h_{0,M,0}(2\zeta)\asymp \frac{R(n)e^{2n\zeta \tau}}{\zeta\tau}.
\end{align}
On the other hand, invoking \eqref{eq:calculus} and \eqref{eq:change}  gives
\begin{align}\label{eq:split2}
\sum_{j\in \mathcal{I}_{r}^{(-)}}  j^k  e^{-2j\tau \zeta}\asymp_k \frac{h_{0,M,k}(2\zeta)}{\tau^{k+1}}\asymp_k  \frac{1}{(\zeta\tau)^{k+1}},\\
\label{eq:split2.1}
\sum_{j\in \mathcal{I}_{\ell}^{(-)}}  j^k  e^{-2j\tau \zeta}\asymp_k \frac{h_{\ell-r,M,k}(2\zeta)}{\tau^{k+1}}\lesssim_k  \frac{M^{(k+1)(\delta-|\ell-r|)}}{\tau^{k+1}},
\end{align}
where the last estimate of the second inequality uses \eqref{eq:kc3} and \eqref{eq:kc3.1}. Using \eqref{eq:split1} and \eqref{eq:split2} with $k=0,1$ the RHS of \eqref{eq:kc5} can be bounded by
\begin{align*}
\zeta\tau^3\sum_{i\in \mathcal{I}_\ell^-}e^{-2i\tau\zeta}\Big[i+\frac{1}{\zeta \tau}\Big]^2\lesssim 
 \zeta \Big[M^{3(\delta-|\ell-r|)}+\frac{M^{\delta-|\ell-r|}}{\zeta^2}\Big]
 \lesssim M^{\delta-|\ell-r|}.
 \end{align*}
This completes the proof of the lemma.
\end{proof}

\begin{lem}\label{lem:ex_sup}
For any $\ell,r\in [K]$ and $b\in \{-,+\}$ we have the following bounds:

\begin{enumerate}
\item[(i)] If $a=-$, then 
\begin{align*}
\P\Big(\sup_{t\in [1,M^{1-2\delta}]}Y_{n,r,M}^{(-),\ell}(b,t)>\delta\Big)\lesssim &M^{-|\ell-r|(1+\alpha-\delta)+(\alpha+3)\delta},\\
\E\sup_{t\in [1,M^{1-2\delta}]}|Y_{n,r,M}^{(-),\ell}(b,t)|\lesssim &M^{-|\ell-r|(1+\alpha-\delta)+(\alpha+3)\delta}.
\end{align*}

\item[(ii)] If $a=+$, then

\begin{align*}
\P\Big(\sup_{t\in [1,M^{1-2\delta}]}Y_{n,r,M}^{(-),\ell}(b,t)>\delta\Big)\lesssim & M^{-|\ell-r|+\delta},\\
\E\sup_{t\in [1,M^{1-2\delta}]}|Y_{n,r,M}^{(-),\ell}(b,t)|\lesssim &M^{-|\ell-r|+\delta}.
\end{align*} 
\end{enumerate}

\end{lem}

\begin{proof}

To prove this lemma we will use Proposition \ref{ppn:gen} for the choices 
$[T_1,T_2]=[1,M^{1-2\delta}]$,and $W(.):=Y_{n,r,M}^{(-),\ell}(b,.).$
We split the proof into two cases, depending on the value of $a$.
\\
\begin{enumerate}
\item[(i)] $a=-$

In this case we have
\begin{align*}
\E Y_{n,r,M}^{(-),\ell}(b,1)^2=\frac{h^{(-),\ell}_{n,r,M}(2)}{h^{(-),r}_{n,r,M}(2)} &\lesssim 
M^{(|\ell-r|+2)\delta}\frac{h_{\ell-r,M,\alpha}(2)}{h_{0,M,\alpha}(2)}\\ 
&\lesssim M^{(|\ell-r|+2)\delta} M^{(-|\ell-r|+\delta)(1+\alpha)}\\ &=M^{-|\ell-r|(1+\alpha-\delta)+(\alpha+3)\delta},
\end{align*}
where the inequality in the first line uses Lemma \ref{lem:weak*} part (a)(i) and in the second line uses \eqref{eq:estimate} for the denominator, and \eqref{eq:kc3} and \eqref{eq:kc3.1} for the numerator. Thus we can 
 take $\gamma_1=C M^{-|\ell-r|(1+\alpha-\delta)+(\alpha+3)\delta},$, where $C>0$ is a universal constant. Proceeding to estimate $\gamma_2$ of Proposition \ref{ppn:gen}, using Lemma \ref{lem:tightness} we can take 
$\gamma_2=C M^{-|\ell-r|(\alpha+1-3\delta)+(\alpha+3)\delta}$. Invoking Proposition \ref{ppn:gen} the desired conclusions follow.
\\

\item[(ii)] $a=+$

In this case we have
\begin{align*}
\E Y_{n,r,M}^{(+),\ell}(b,1)^2=\frac{h^{(+),\ell}_{n,r,M}(2)}{h^{(+),r}_{n,r,M}(2)}\asymp &
\frac{h_{\ell-r,M,\alpha}(2)}{h_{0,M,\alpha}(2)}\text{ [ Using Lemma \ref{lem:weak*} part (b)(i)]}\\
\lesssim  &M^{(-|\ell-r|+\delta)},
\end{align*}
where the inequality in the second line uses \eqref{eq:estimate} for the denominator, and \eqref{eq:kc3} and \eqref{eq:kc3.1} for the numerator. Thus we can 
 take $\gamma_1=C M^{-|\ell-r|+\delta},$, where $C>0$ is a universal constant. Using Lemma \ref{lem:tightness} we can take 
$\gamma_2=C M^{\delta-|\ell-r|}$. Invoking Proposition \ref{ppn:gen} the desired conclusions follow as before.
\end{enumerate}

\end{proof}

\begin{lem}\label{lem:tail_bound}
 For any $h\in (0,1)$  and $L>0$ , for all $n\ge n_0(M,\delta,h,L,\alpha)$ and $M\ge M_0(\delta, L, \alpha)$ we have 
\begin{enumerate}
\item when $|\ell - r|\leq s$ for some fixed integer $s$, 
 \begin{align*}
 \P(\mathcal{K}_{n,\delta}^{(-)}(\ell,r))\le M^{-(\alpha+1 -3\delta)|\ell - r| -(\alpha+3)\delta}, 
 \\
  \P(\mathcal{K}_{n,\delta}^{(+)}(\ell,r))\le M^{-|\ell-r|-\delta}.
  \end{align*}
\item for $\ell, r$ such that $|\ell - r|>s$,
\begin{align*}
 \P(\mathcal{K}_{n,\delta}^{(-)}(\ell,r))\le  M^{-(\alpha+1 -3\delta)|\ell - r| +(\alpha+3)\delta} , 
 \\
  \P(\mathcal{K}_{n,\delta}^{(+)}(\ell,r))\le M^{-|\ell-r|+\delta}.
  \end{align*}
\end{enumerate}

\end{lem}

\begin{proof}[Proof of Lemma \ref{lem:tail_bound}]

Use \eqref{eq:backlog} and definition \ref{def:Y} to note that  
\begin{align*}
{\Xi}_{n,r}^{(a)}(\delta)=&\Big\{\max_{b\in \{+,-\}} \sup_{u\in  \mathcal{J}_r^{(a)}} \frac{Q_n^{(a),r}(b e^u)}{\sigma_n(u)}<\delta\Big\}
=\Big\{\max_{b\in \{+,-\}} \sup_{t\in [1, M^{1-2\delta}]} Y_{n,r,M}^{(a),r}(b,t)<\delta\Big\},\\
\mathcal{K}_{n,\delta}^{(a)}(\ell,r)= &\Big\{\min_{b\in \{+,-\}}\inf_{  u \in \mathcal{J}^{(a)}_{r}}\frac{Q^{(a),\ell}_n(be^{u})}{\sigma_n(u)}<-\delta\Big\}= \Big\{\min_{b\in \{+,-\}}\inf_{  t \in [1,M^{1-2\delta}]}Y_{n,r,M}^{(a),\ell}(b,t)<-\delta\Big\}.
\end{align*}
where $\tau(n,M)=\frac{1}{n^\delta M^{r-\delta}}$ as in Lemma \ref{lem:weak*}.

Using Lemma \ref{lem:weak}, we get that whenever $\ell_n-r_n\to s$, we have
\begin{align}\label{eq:weak-convergence}
\Big\{\min_{b\in \{+,-\}}\inf_{  t \in [1,M^{1-2\delta}]}Y_{n,r,M}^{(a),\ell}(b,t)<-\delta\Big\}\stackrel{d}{\to}  \Big\{\min_{b\in \{+,-\}}\inf_{  t \in [1,M^{1-2\delta}]}Y_{s,M}^{(a)}(b,t)<-\delta\Big\}.
\end{align}
We show how the conclusions in (1) follows from the above convergence.  Lemma~\ref{lem:ex_sup} shows that $\inf_{  t \in [1,M^{1-2\delta}]}Y_{n,r,M}^{(a),\ell}(b,t)$ are uniformly tight. Furthermore, for any $t\in [1, M^{1-\delta}]$ and $s\in \mathbb{Z}$, we have $\mathrm{Var}\big(Y_{s,M}^{(a)}(b,t)\big) = \int^{s+\delta}_{s+\delta-1} x^\alpha e^{-tx} dx$. Recall that $\alpha >0$ for $a= -$ and $\alpha= 0$ if $a= +$. In the case when $s\neq 0$, we get the following bound 
\begin{align*}
\mathrm{Var}\big(Y_{s,M}^{(a)}(b,t)\big) \leq \begin{cases} M^{-(\alpha +1-3\delta)|s|- 3\delta} e^{-M^{s+1-\delta}} & s\leq -1, \\
M^{(\alpha +1)s} e^{-M^{s-\delta}} & s\geq 1.
\end{cases}
\end{align*}
Without loss of generality, one can bound the right hand side of the above equation by $M^{-(\alpha +1-3\delta)|s|- 3\delta} e^{-M^{s+1-\delta}}$. By the maximal inequality of Gaussian processes, the probability $\mathbb{P}\big(\inf_{  t \in [1,M^{1-2\delta}]}Y_{s,M}^{(a)}(b,t)<-\delta\big)$ could now be bounded by $M^{-(\alpha +1-3\delta)|s|- 3\delta} $. The conclusion in (1) now follows from this bound and the weak convergence in the display \eqref{eq:weak-convergence}.

To establish conclusion (2), it suffices to show that for any $s\in [-L,L]$ there exists $M_0$, such that for all $M>M_0$ we have 
\begin{align}
\P\big(\inf_{t\in [1,M^{1-2\delta}]} Y_{s,M}^{(a)}(b,t)<-\delta\big)\le \begin{cases} 
 M^{-(\alpha+1 -3\delta)|s| +(\alpha+3)\delta} & a = -\\
  M^{-|s|+\delta} & a = +
\end{cases}.
\end{align}
To this effect, we use Lemma \ref{lem:ex_sup} to get the appropriate $\gamma>0$ (depending on $\alpha$), such that 
$$\E\sup_{t\in [1,M^{1-2\delta}]} |Y_{s,M}^{(a)}(b,t)|\lesssim  M^{-\gamma}.$$
The desired conclusion then follows by an application of Borel-TIS inequality (\cite[Chapter 2]{adler2009random}) for all $M$ large enough, depending on $(\alpha,\delta,L)$.
\end{proof}

\begin{ppn}\label{ppn:gen}

Suppose $\{W(t)\}_{t\in [T_1,T_2]}$ be a continuous time stochastic process with mean $0$ and continuous sample paths, such that for some constants $\gamma_1,\gamma_2$ positive and $\beta>1$ we have
$$\E W(T_1)^2\le \gamma_1,\quad \E(W(t)-W(s))^2\le \gamma_2(s-t)^\beta.$$ Then there
 exists a constant $K$ depending only on $\beta$ such that the following hold:

$$\P\left(\sup_{t\in [T_1,T_2]}|W(t)|>\delta\right)\le \frac{K(\gamma_1+\gamma_2)}{\delta^2},\quad \E \sup_{t\in [T_1,T_2]}|W(t)|\le K(\gamma_1+\gamma_2). $$

\end{ppn}

\begin{proof}
Since the second conclusion follows from the first one, it suffices to verify the first conclusion. To this effect, for every $q\ge 1$, partition the interval $[T_1,T_2]$ into dyadic rationals of the form $i_j:=T_1+(T_2-T_1)\frac{j}{2^j}$ for  $0\le j\le 2^q$. 
Then continuity of sample paths gives
$$\sup_{t\in [T_1,T_2]}|W(t)|\le |W(T_1)|+\sum_{q=1}^\infty \max_{1\le j\le 2^q}|W(i_j)-W(i_{j-1})|$$
and so
\begin{align*}
\P\Big(\sup_{t\in [T_1,T_2]}|W(t)|>\delta\Big)\le&\P\Big( |W(T_1)|>\frac{\delta}{2}\Big)+\sum_{q=1}^\infty \sum_{j=1}^{2^q} \P\Big( |W(i_j)-W(i_j)|>\frac{c\delta}{q^2}\Big).
\end{align*}
where $c>0$ is chosen such that $c\sum_{q=1}^\infty \frac{1}{q^2}= \frac{1}{2}$. The desired conclusion then follows by an application of Chebyshev's inequality, on noting that
$$\P\Big( |W(i_j)-W(i_j)|>\frac{c\delta}{q^2}\Big)\le \frac{\gamma_2q^4}{c^2 \delta^22^{q\beta}}.$$

\end{proof}

\subsection{Proof of Lemma~\ref{lem:KextBound}}\label{sec:proof_of_lemma_3.1}
Recall that 
\begin{align*}
    \mathcal{K}^{(a)}_{n,\delta}(\mathrm{ext} , r)=\Big\{\min_{b\in \{+,-\}}\inf_{u \in \mathcal{J}^{(a)}_r}\frac{Q_n^{(a),\mathrm{ext}}(b e^{u})}{\sigma_n(u)}< -\delta\Big\}.
\end{align*}
where 
$Q_n^{(a,\mathrm{ext})}(x):=\sum_{i\in \mathcal{I}^{(a)}_{\mathrm{ext}}} a_i x^i$. To bound $\P(\mathcal{K}^{(a)}_{n,\delta}(\mathrm{ext} , r))$ uniformly for $\omega K \leq r\leq (1-\omega )K$, notice that $\{\inf_{u \in \mathcal{J}^{(a)}_r}\frac{Q_n^{(a),\mathrm{ext}}(b e^{u})}{\sigma_n(u)}\}$ behaves very similarly as $\{\inf_{u \in \mathcal{J}^{(a)}_r}\frac{Q_n^{(a),\ell}(b e^{at\tau})}{\sigma_n(at\tau)}\leq -\delta \}$ for the case when $\ell-r>\omega K$. By computing the variance, we get 
\begin{align*}
\mathrm{Var}\Big(\frac{Q_n^{(a),\ell}(b e^{at\tau})}{\sigma_n(at\tau)}\Big) \leq M^{(\alpha +1)s} e^{-M^{\ell - r-\delta}}  \qquad \text{ for } \ell -r\geq 1. 
\end{align*}
Plugging $\ell -r\geq \omega K$ and using the variance bound to first bound the probability of $\P(\frac{Q_n^{(a),\mathrm{ext}}(b e^{u})}{\sigma_n(u)}<-\delta)$ and then, to bound $\P(\mathcal{K}^{(a)}_{n,\delta}(\mathrm{ext} , r))$ using Borel-TIS \cite{adler2009random} proves the result.   

\section{Lower Bound}\label{sec:lower_bd}

\subsection{Fixation of Notations}
Before proceeding to the proof, we introduce few notations. To begin, fix  positive reals $M, \tilde{M}, h$.
Define the sets
\begin{align}
A_0:=\Big[-\frac{K}{n},\frac{K}{n}\Big],\quad A_{+1}:=\Big(\frac{K}{n},&\frac{h}{\log n}\Big],\quad A_{+2}:=\Big(\frac{h}{\log n},\infty\Big), \nonumber\\\quad A_{-1}:=-A_1,&\quad A_{-2}:=-A_2 \label{eq:A_s}\\
 B_0:=\Big(\frac{n}{D}, n-\frac{n}{D}\Big],\quad B_1:=\big(n-\frac{n}{D}, &n-L\log n\big],\quad B_2:=(n-L\log n,n],\quad \nonumber\\ B_{-1}:=\Big(L\log n, \frac{n}{D}\Big],&\quad B_{-2}:=[0, L \log n], \label{eq:B_s}
\end{align}
and note that $\cup_{r\in \mathcal{R}}A_r=\R, \cup_{r\in \mathcal{R}}B_r=[0,n]$ are disjoint partitions, where 
$\mathcal{R}:=\{0,\pm1, \pm 2\}$. We intend to partition the interval $B_1$ and $B_{-1}$ into sub-intervals (blocks) of size $M$. 
The choice of $K,M,D$ are made in a way such that   $K\gg M, D\gg K.$ For any $r\in \mathcal{R}$, define 
$$Q^{(r)}_n(x) := \sum_{i\in B_r\cap \mathbb{Z}} a_i x^i.$$

We now divide $B_{-1}=(L \log n, \frac{n}{D}]$ and $B_{1}= (n-\frac{n}{D},n - L\log n]$ into $T_n:=\lceil \frac{\log \frac{nh}{K}}{\log M}\rceil$ many sub-intervals. Define the sets $\{\mathcal{I}^{(-)}_{1\leq r\leq T_n}\}$ by setting 
\begin{align}\label{eq:tildeI-} 
\tilde{\mathcal{I}}^{(-)}_{r} = \mathbb{Z}\cap [L M^{r-1}\log n , L M^{r}\log n ), \quad 1\leq r\leq T_n
\end{align}
and note that $$B_{-1}\subset \cup^{T_n}_{r=1} \tilde{I}^{(-)}_{r}.$$
Similarly, define the sets $\{\tilde{\mathcal{I}}^{(+)}_{r}\}_{1\leq r\leq T_n}$ where 
\begin{align}\label{eq:tildeI+}
\tilde{\mathcal{I}}^{(+)}_{r} = \mathbb{Z}\cap (n - L M^{r}\log n , n - LM^{r-1}\log n ], \quad 1\leq r\leq T_n,
\end{align}
and note that 
$$B_{+1} \subset  \cup^{T_n}_{r=1} \tilde{\mathcal{I}}^{(-)}_{r}.$$

For $1\leq p \leq T_n$, define 
\begin{equation}\label{eq:tildeQ}
 \tilde{Q}_n^{(-1),p}(x):=\sum_{i\in \tilde{\mathcal{I}}^{(-)}_r} a_i x^i,\quad \tilde{Q}_n^{(+1),p}(x):=\sum_{i\in \tilde{\mathcal{I}}^{(+)}_r} a_i x^i \quad \text{ for  }\quad 1\le p\le T_n, 
 \end{equation}


Define the positive weight function $\sigma_n(u)$ as 
 \begin{align}\label{eq:dom}
 \sigma_n^2(u) :=\begin{cases}
 n^{\alpha+1}L(n)e^{2nu} & \text{ if }u\in A_0\\
 \frac{e^{nu} L(n)}{u}& \text{ if }u\in A_1\cup A_2\\
 |u|^{-\alpha-1}L(1/|u|)& \text{ if }u\in A_{-1}\cup A_{-2}.
 \end{cases}
 \end{align}

\subsection{Proof of $\mathfrak{LimInf}$}
 
For any $\delta>0$, we have
\begin{align}\label{eq:lower_bound}
\P(Q_n(x)<0,x\in \R)=&\P\Big(\frac{Q_n(\pm e^u)}{\sigma_n(u)}<0, u\in \R\Big)\\
\notag\ge &\prod_{r\in \mathcal{R}}\P\Big(\frac{Q_n^{(r)}(\pm e^u)}{\sigma_n(u)}<-\delta , u\in A_r, \frac{Q_n^{(r)}(\pm e^u)}{\sigma_n(u)}<\delta/4, u\notin A_r \Big).\label{eq:lower_bound*}
\end{align}
where the last line uses the fact that $\{Q_n^{(r)}(.), r\in \mathcal{R}\}$ are independent. For any $r\in \mathcal{R}$, define  
$$
\mathbb{P}_{\delta}[Q_n;r]:=\P\Big(\frac{Q_n^{(r)}(\pm e^u)}{\sigma_n(u)}<-\delta , u\in A_r, \frac{Q_n^{(r)}(\pm e^u)}{\sigma_n(u)}<\delta/4, u\notin A_r \Big).
$$
Below we find lower bound to $\mathbb{P}_{r}[Q_n]$ for each $r\in \mathcal{R}$.    
\begin{ppn}\label{prop:r=-1} 
Consider any $p\in \mathbb{Z}_{>0}\cap [1,T_n]$. There exists $\Gamma = \Gamma(\delta)>0$ such that for all $n\geq n_0$ and $M$ large, 
\begin{align}
\mathbb{P}_{\delta}[Q_n;-1] \geq 2^{-\Gamma T_n} \mathbb{P}\big(\sup_{t\in [1, M]} \tilde{Y}^{(-)}_{0,M}(t)\leq -\delta\big)^{2 T_n}
\end{align}  
where the Gaussian process $\tilde{Y}^{(-)}_{s,M}$ is defined in Definition~\ref{def:tildeY}.
\end{ppn}
 Proposition~\ref{prop:r=-1} is proven in Section~\ref{sec:r=-1}. 
 
\begin{ppn}\label{prop:r=1} 
Consider any $p\in \mathbb{Z}_{>0}\cap [1,N]$. There exists $\Gamma = \Gamma(\delta)>0$ such that for all $n\geq n_0$ and $M$ large, 
\begin{align}
\P_{\delta}[Q_n;+1]\geq 2^{-\Gamma T_n} \mathbb{P}\big(\sup_{t\in [1, M]} \tilde{Y}^{(+)}_{0,M}(t)\leq -\delta\big)^{2T_n}. 
\end{align} 
 The Gaussian process $\tilde{Y}^{(+)}_{s,M}$ is defined in Definition~\ref{def:tildeY}.
\end{ppn}
Proposition~\ref{prop:r=-1} is proven in Section~\ref{sec:r=1}. 

\begin{ppn}\label{prop:r=0}
 There exists $\delta_0>0$ and $C_1,C_2>0$ such that for all $n$ and $K$ large and $\delta<\delta_0$, 
\begin{align}
\P_{\delta}[Q_n;0]\geq 4^{-C_1\delta^{-1-\theta} K \log K} - e^{-C_2K^2}. 
\end{align}  
\end{ppn}
Proposition~\ref{prop:r=-1} is proven in Section~\ref{sec:r=0}. 

\begin{ppn}\label{ppn:r=2}
 Assume that there exists $\rho>0, \eta \in (0,1)$ and a small $c=c(\rho,\eta)>0$ such that for $\xi_i:=a_i/\sqrt{R(i)}$,
$$\P(\xi_i\leq -\rho)\geq c,\quad  \P(\xi_i\in [-\eta\rho,0])\geq c, \quad \text{or,} \quad \P(\xi_i\in [0,\rho])\geq c, \quad 
\forall 1\leq i\leq n.$$
Then for all large $n$,
\begin{align}\label{eq:Q2LBd}
\P_{\delta}[Q_n;+2]\geq c^{L\log n}
\end{align}
where $L>0$ is an arbitrarily small number. 
\end{ppn}
Proposition~\ref{prop:r=-1} is proven in Section~\ref{sec:r=2}. 

\begin{ppn}\label{ppn:r=-2}
 Assume that there exists $\rho>0, \eta \in (0,1)$ and $c=c(\rho)>0$, such that for $\xi_i=a_i/\sqrt{R(i)}$,
$$\P(\xi_i\leq -\rho)\geq c, \quad \P(\xi_i \in [-\eta\rho,0])\geq c, \quad \text{or, }\quad \P(\xi_i \in [0, \rho])\geq c, \quad 
\forall 1\leq i\leq n.$$
Then for all large $n$, 
\begin{align}\label{eq:Q-2LBd}
\P_{\delta}[Q_n;-2]\geq c^{L\log n}
\end{align}
for some arbitrarily small number $L>0$.
\end{ppn}
Proposition~\ref{prop:r=-1} is proven in Section~\ref{sec:r=-2}. 
\vspace{0.2cm}


\noindent \textsc{Final Stage of the proof:}
Continuing \eqref{eq:lower_bound}, we can write 
\begin{align*}
\P(Q_n(x)<0,x\in \R) &\geq \prod_{r \in \mathcal{R}}\P_{\delta}[Q_n;r] \\
&\geq 2^{-2\Gamma T_n} \mathbb{P}\big(\sup_{t\in [1, M]} \tilde{Y}^{(-)}_{0,M}(t)\leq -\delta\big)^{2 T_n} \mathbb{P}\big(\sup_{t\in [1, M]} \tilde{Y}^{(+)}_{0,M}(t)\leq -\delta\big)^{2T_n}\\
& \times c^{2\log n}( 4^{-C_1\delta^{-1-\theta} K \log K} - e^{-C_2K^2}).
\end{align*}
Taking logarithm on both sides and dividing both sides by $\log n$ shows 
\begin{align}\label{eq:LowerBoundFinal}
&\frac{\log \P(Q_n(x)<0,x\in \R)}{\log n}  \geq -\frac{2\Gamma T_n}{\log n} + 2L\log c+ \frac{1}{\log n} \log\big( 4^{-C_1\delta^{-1-\theta} K \log K} - e^{-C_2K^2}\big)    \nonumber\\ &+ \frac{2T_n}{\log n}\log \mathbb{P}\big(\sup_{t\in [1, M]} \tilde{Y}^{(-)}_{0,M}(t)\leq -\delta\big) + \frac{2T_n}{\log n}\log\mathbb{P}\big(\sup_{t\in [1, M]} \tilde{Y}^{(+)}_{0,M}(t)\leq -\delta\big) 
\end{align}
As $n,M \to \infty$ and $\delta \to 0$, we know 
\begin{align*}
\lim_{M\to \infty} &\lim_{\delta\to 0}\lim_{n\to \infty} \frac{1}{\log M }\log \mathbb{P}\big(\sup_{t\in [1, M]} \tilde{Y}^{(-)}_{0,M}(t)\leq -\delta\big) \\ &= \lim_{M\to \infty} \frac{1}{\log M }\log \mathbb{P}\big(\sup_{t\in [1, M]} \tilde{Y}^{(-)}_{0,M}(t)\leq 0\big) = 2b_{\alpha} 
\end{align*}
and 
\begin{align*}
\lim_{M\to \infty} &\lim_{\delta\to 0}\lim_{n\to \infty} \frac{T_n}{\log n }\log \mathbb{P}\big(\sup_{t\in [1, M]} \tilde{Y}^{(+)}_{0,M}(t)\leq -\delta\big) \\ &= \lim_{M\to \infty} \frac{1}{\log M }\log \mathbb{P}\big(\sup_{t\in [1, M]} \tilde{Y}^{(+)}_{0,M}(t)\leq 0\big) = 2b_{\infty}. 
\end{align*}
Combining these limits and plugging back into \eqref{eq:LowerBoundFinal} yields the results $$\liminf_{n\to \infty}p_{2n}\geq -2b_{\infty} - 2b_{\alpha}. $$ It remains to show that the conditions of Proposition~\ref{ppn:r=2} and~\ref{ppn:r=-2} are satisfied when $\xi_i= a_i/\sqrt{R(i)}$ are i.i.d. random variables with $\mathbb{E}[\xi_i]=0$ and $\mathrm{Var}(\xi_i) =1$. Since $\mathbb{E}[\xi_i]=0$, there must exists $\rho>0$ such that $\P(\xi_i\leq -\rho)>0$. If any of the probabilities $\P(\xi_i \in (-\rho,0])$ or $\P(\xi_i \in [0,\rho])$ is positive, then the conditions of the propositions are satisfied. Otherwise there must exists $\rho'>\rho$ such that $\mathbb{P}(\xi\geq \rho')>0$ since $\mathbb{E}[\xi_i] =0 $. Therefore we arrive at the condition $\mathbb{P}(\xi_i\geq \rho')>c$ and $\mathbb{P}(\xi_i\in [0, \rho'])>c$ for some small constant. Under these revised conditions, using the same arguments as in Proposition~\ref{ppn:r=2} and~\ref{ppn:r=-2}, we obtain 
\begin{align*}
\P\Big(\frac{Q_n^{(r)}(\pm e^u)}{\sigma_n(u)}>\delta , u\in A_{r}, \frac{Q^{(r)}_n(\pm e^u)}{\sigma_n(u)}>-\delta/4, u\notin A_{2} \Big)
\end{align*}
for $r\in \{-2,+2\}$.
By flipping the sign of the coefficients $a_i$ for $i \in B_{-1}\cup B_0\cup B_{+1}$ and using Proposition~\ref{prop:r=-1},~\ref{prop:r=0} and~\ref{prop:r=1} respectively, we obtain for $r\in \{-1,+1\}$,
\begin{align*}
\P\Big( &\frac{Q_n^{(r)}(\pm e^u)}{\sigma_n(u)}>\delta , u\in A_r, \frac{Q_n^{(r)}(\pm e^u)}{\sigma_n(u)}>-\delta/4, u\notin A_r \Big)\\ &\geq 2^{-\Gamma T_n} \mathbb{P}\big(\sup_{t\in [1, M]} \tilde{Y}^{(\pm)}_{0,M}(t)<- \delta\big)^{2 T_n} 
\end{align*}
and, 
\begin{align*}
\P\Big( &\frac{Q_n^{(0)}(\pm e^u)}{\sigma_n(u)}>\delta , u\in A_0, \frac{Q_n^{(0)}(\pm e^u)}{\sigma_n(u)}>-\delta/4, u\notin A_0 \Big)\\ &\geq 
4^{-C_1\delta^{-1-\theta} K \log K} - e^{-C_2K^2}.  
\end{align*}
These bounds together shows that 
$$\liminf_{n\to \infty} \frac{1}{\log n}\log \P(Q_{2n}(x) \text{ has no real zero})\geq  - 2b_{\alpha} - 2b_{0}.$$

\subsection{Proof of Proposition~\ref{prop:r=-1}: $r=- 1$ Case}\label{sec:r=-1}

 We consider the following decomposition $Q_n^{(-1)}(x)=\sum_{p=1}^{T_n} \tilde{Q}_n^{(-1),p}(x)$, and so
 \begin{align}
\notag& \P\Big(\frac{Q_n^{(-1)}(\pm e^u)}{\sigma_n(u)}<-\delta, u\in A_{-1}, \frac{Q_n^{(-1)}(\pm e^u)}{\sigma_n(u)}<\delta/4, u\notin A_{-1}\Big)\\
\notag \ge &\prod_{p=1}^N \P\Big(B_{1p}\cap B_{2p}\cap B_{3p}\cap B_{4p}\Big)\\
\label{eq:b-1} \ge &\prod_{p=1}^N\Big[\P(B_{1p})-\P(\neg B_{2p})-\P(\neg B_{3p})-\P(\neg B_{4p})\Big],
 \end{align}
 where for a non negative sequence $\rho(.)$ satisfying $\sum_{i=1}^\infty \rho(i)<1/2$ and a large but fixed integer $\Gamma$, we define 
\begin{align}
 \label{eq:b1p}B_{1p}:=&\bigcap_{q:|q-p|\le \Gamma} B^{(q)}_{1,p}, \quad \text{where }B^{(q)}_{1p}:=\Big\{\sup_{u\in I_q}\frac{\tilde{Q}_n^{(-1),p}(\pm e^{-u})}{\sigma_n(-u)}<-2\delta, |p-q|\le \Gamma\Big\},\\
 \label{eq:b2p}B_{2p}:=&\bigcap_{q:|q-p|>\Gamma}B^{(q)}_{2p}, \quad \text{where }B^{(q)}_{2p}:=\Big\{ \sup_{u\in I_q}\frac{\tilde{Q}_n^{(-1),p}(\pm e^{-u})}{\sigma_n(-u)}<\delta\rho(p-q)\Big\},\\
\label{eq:b3p} B_{3p}:=&\Big\{\sup_{u\in A_0\cup A_1\cup A_2}\frac{\tilde{Q}_n^{(-1),p}(\pm e^{u})}{\sigma_n(u)}<\delta\rho(N+1-p)\Big\},\\
\label{eq:b4p} B_{4p}:=&\Big\{\sup_{u\in A_{2} }\frac{\tilde{Q}_n^{(-1),p}(\pm e^{-u})}{\sigma_n(-u)}<\delta \rho(p)\Big\},
 \end{align}
 For simplicity, we will choose to work with $\rho(i) = \frac{\kappa}{i^2}$ for some small constant $\kappa>0$.

 Proceeding to estimate the probability of the sets above, we claim the following Lemma, whose proof we defer to the end of the section.
 \begin{lem}\label{lem:defer}
 Fix any $p\in \{1,\ldots , T_n\}$. Fix an interval $[c,d]$ such that $[c,d]$ is \emph{far away from} $\tilde{\mathcal{I}}^{(-)}$, i.e.,  
 either $c>\frac{1}{\log n M^{p-1-\omega}}$ or $d<\frac{1}{\log n M^{p+\omega}}$ for some $\omega>0$. Denote $s_{n,p}:= M^p\log n$. 
 \begin{enumerate}
 \item[(a)]
 If $d<\frac{M^{-p-\omega}}{\log n}$, for any $\lambda>0$ we have
   \begin{align}\label{eq:c<d}
   \P\Big(\sup_{u\in [c,d]}|\tilde{Q}_n^{(-1),p}(\pm e^{-u})|>\lambda\Big)\lesssim_{\alpha, \omega} \frac{\max_{i\in  \tilde{\mathcal{I}}^{(-)}}L(i)}{\lambda^2}s^{\alpha+1}_{n,p} \left[s^2_{n,p} (c-d)^2+1\right]. 
 \end{align}
  
  \item[(b)]
  If $c>\frac{M^{-p+1+\omega}}{\log n }$, for any $\lambda>0$ we have
  \begin{align}\label{eq:c>d}
  \P\Big(\sup_{u\in [c,d]}|\tilde{Q}_{n}^{(-1), p}(\pm e^{-u})|>\lambda\Big)\lesssim_{\alpha,\omega}\frac{e^{-cs_{n,p}/M} \max_{i\in  \tilde{\mathcal{I}}^{(-)}}L(i)}{\lambda^2} \left[\frac{(c-d)^2}{c^{\alpha+3}}+\frac{1}{c^{\alpha+1}}\right].
  \end{align}
  
  \end{enumerate}
  \end{lem}
  
  \begin{proof}[Proof of Lemma \ref{lem:defer}]
  We prove the lemma for $\tilde{Q}_{n}^{(-1), p}(e^{-u})$. Proof in the other case follows from same argument.
  \begin{enumerate}
  \item[(a)]
  To begin, for any $u,v\in [c,d]$ with $d<\frac{M^{-p-\omega}}{\log n}$, we have
 \begin{align*}
 \notag\E\Big[ \tilde{Q}_n^{(-1,p)}( e^{-u})- \tilde{Q}_n^{(-1,p)}(e^{-v})\Big]^2 &=\sum_{i\in I_p^{-1}}i^\alpha L(i)\Big(e^{-iu}-e^{-iv}\Big)^2\\
\notag &\le (u-v)^2 \sup_{\xi\in [c,d]}\sum_{i\in \tilde{\mathcal{I}}^{(-)}} i^{\alpha+2}L(i) e^{-2i\xi}\\
\notag &\leq (u-v)^2\max_{i\in  \tilde{\mathcal{I}}^{(-)}}L(i) \sum_{i\in  \tilde{\mathcal{I}}^{(-)}} i^{\alpha+2} e^{-2ic}\\
&\lesssim_{\alpha,\omega} (u-v)^2 (\log n M^p)^{\alpha +3} \max_{i\in  \tilde{\mathcal{I}}^{(-)}}L(i) .
  \end{align*}
  We obtain the last inequality by using Lemma~\ref{lem:TailSumBoundLemma}. 
 Invoking Lemma \ref{lem:kol_cent} with $C^2= (\log n M^p)^{\alpha +3} \max_{i\in  \tilde{\mathcal{I}}^{(-)}}L(i) $ gives
  \begin{align*}
   \P\Big(\sup_{u\in [c,d]}|\tilde{Q}_{n}^{(-1),p}(u)-\tilde{Q}_{n,}^{(-1),p}(d)|>\lambda\Big)\lesssim (\log n M^p)^{\alpha +3} \max_{i\in  \tilde{\mathcal{I}}^{(-)}}L(i) \frac{(c-d)^2}{\lambda^2}.
   \end{align*}

  Finally we use Chebyshev's inequality to get
  \begin{align*}
  \P(|\tilde{Q}_{n}^{(-1),p}(d)|>\lambda)
   \lesssim_{\alpha,\omega} \frac{1}{\lambda^2} (\log n M^p)^{\alpha +1} \max_{i\in  \tilde{\mathcal{I}}^{(-)}}L(i),
  \end{align*}
  which along with the bound in the previous display gives the desired estimates of \eqref{eq:c<d}.
\\

\item[(b)]
As before, fixing $u,v\in [c,d]$ with $c>\frac{M^{-p+1+\omega}}{\log n }$ we have
    \begin{align*}
\notag  \E\Big[\tilde{Q}_{n}^{(-1),p}(u )-\tilde{Q}_{n}^{(-1),p}(v )\Big]^2 &=\sum_{i\in I_p^{-1}} i^{\alpha} L(i)\Big( e^{-iu}-e^{-iv}\Big)^2\\
\notag  & \le (u-v)^2 \max_{i \in \tilde{\mathcal{I}}^{(-)}} L(i)\sum_{i \in \tilde{\mathcal{I}}^{(-)}} i^{\alpha}e^{-ci}\\
 \label{eq:a>b}  & \lesssim_{\alpha,\omega}   (u-v)^2\max_{i \in \tilde{\mathcal{I}}^{(-)}} L(i)\Big(\frac{1}{c}\Big)^{\alpha+3}e^{-cs_{n,p}/M},
  \end{align*}
  where we have used Lemma~\ref{lem:TailSumBoundLemma} to obtain the last inequality. The above bound along with Lemma \ref{lem:kol_cent} gives
 \[  \P\Big(\sup_{u\in [c,d]}|\tilde{Q}_{n}^{(-1),p}(u)-\tilde{Q}_{n,}^{(-1),p}(c)|>\lambda\Big)\lesssim_{\alpha,\omega} \frac{1}{\lambda^2} e^{-cs_{n,p}/M}\frac{(c-d)^2}{c^{\alpha+3}}.\]
 Again, using Chebyshev's inequality and Lemma~\ref{lem:TailSumBoundLemma} gives
  \begin{align*}
  \P(|\tilde{Q}_{n}^{(-1),p}(c)|>\lambda)
   \lesssim_{\alpha,\omega} \frac{1}{\lambda^2c^{\alpha+1}}e^{-cs_{n,p}/M},
  \end{align*}
  from which the desired conclusion follows as before.
 
\end{enumerate}
 \end{proof}

  Armed with Lemma \ref{lem:defer}, we now deal with each of these terms in the RHS of \eqref{eq:b-1} separately. 
 
 \begin{lem}\label{lem:TailSumBoundLemma}
 Consider the following function
 \begin{align}
     H^{(+)}(R,u) := \sum_{i=R}^{\infty}i^{\alpha} e^{-iu}, \quad  H^{(-)}(R,u) := \sum_{i=1}^{R}i^{\alpha} e^{-iu}.
 \end{align}
Fix some positive number $\omega>0$. There exists $M_0=M_0(\omega)>0$ such that for all $M>M_0$, $R_1\in (\frac{\alpha M^{\omega}}{u},\infty)$ and $ R_2\in (0,\frac{\alpha}{u M^{\omega}})$,
\begin{align}
    H^{(+)}(R_1,u) \lesssim_{\alpha,\omega} \frac{e^{-\frac{R_1u}{2}}}{u^{\alpha+1}} , \quad H^{(-)}(R_2,u) \lesssim_{\alpha,\omega} R^{\alpha+1}_2.
\end{align}
 \end{lem}
 \begin{proof}
 Note that $f:\mathbb{R}_{>0}\to \mathbb{R}_{>0}$ which we define as $f(x)= x^\alpha e^{-\alpha x}$ is increasing on the interval $(0,\frac{\alpha}{u M^{\omega}})$ and decreasing on the interval $(\frac{\alpha M^{\omega}}{u},\infty)$. Bounding the sum by its integral approximation shows 
 \begin{align}
     H^{(+)}(R_1,u)\leq \int^{\infty}_{R_1-1} x^{\alpha} e^{-xu}dx, \quad H^{(-)}(R_2,u) \leq \int^{R_2}_0 x^{\alpha} e^{-x u}dx. 
 \end{align}
 Since $R_1\geq \frac{\alpha M^{\omega}}{u}$ and $R_2\leq \frac{\alpha}{u M^{\omega}}$, we know  
 $$ \int^{\infty}_{R_1-1} x^{\alpha} e^{-xu}dx \lesssim_{\alpha,\omega} \frac{e^{-\frac{R_1u}{2}}}{u^{\alpha+1}} ,\quad  \int^{R_2}_0 x^{\alpha} e^{-x u}dx\lesssim_{\alpha,\omega} R^{\alpha+1}_2. $$
 This completes the proof.
 \end{proof}

 \subsubsection{Lower Bound on $\P(B_{1p})$}

 \begin{lem}\label{lem:B1p}
 Consider the Gaussian process $\{\tilde{Y}^{(-1)}_{0,M}(b,t)\}_{b\in \{+,-\},t\in [1,M]}$ defined in \ref{def:tildeY}. For all large $n$, we have 
 \begin{align}\label{eq:PrincipleEvent}
 \mathbb{P}\big(B_{1p}\big)\geq 2^{-4\Gamma}\mathbb{P}\big(\sup_{t\in [1, M]} \tilde{Y}^{(-1)}_{0,M}(t)\leq -\delta\big)^2.
 \end{align} 
 \end{lem}

\begin{proof}
Recall the definition of $\tilde{Q}$ from \eqref{eq:tildeQ}. 
By Lemma~\ref{lem:weakAlt}, we have  
\begin{align}\label{eq:MyEq}
\Big\{ u\in \mathcal{I}^{(-)}_{q}, b\in \{+,-\}: \frac{\tilde{Q}^{(-1),p}_n(be^{-u})}{\sigma_n(-u)}\Big\} & \stackrel{d}{\to} \big\{t\in [M^{-1},1],b\in \{+,-\}: \tilde{Y}^{(-1)}_{q-p,M}(b,t)\big\}   
\end{align} 
as $n\to \infty$ where $\tilde{Y}^{(-1)}_{q-p,M}(b,\cdot)$ is a centered Gaussian process as defined in \ref{def:tildeY}.  
 
 Now we proceed to prove \eqref{eq:PrincipleEvent} using \eqref{eq:MyEq}. By the weak convergence, for all large $n$
  \begin{align}
   \mathbb{P} & \Big(B_{1p}\cap B_{2p} \cap B_{3p} \cap B_{4p}\Big)
   \\
   &\geq \mathbb{P}\Big(\bigcap_{b\in \{+,-\}}\big\{\sup_{t\in [M^{-1},1]}\tilde{Y}^{(-1)}_{0,M}(b,t) \leq -\delta\big\}\bigcap\bigcap_{b\in \{+,-\}} \\ &\qquad \qquad \bigcap_{q:|q-p|\leq \Gamma} \Big\{\inf_{t \in [M^{-1},1]}\tilde{Y}^{(-1)}_{q-p,M}(b,t)\leq \rho(|p-q|)\delta\Big\} \Big)\nonumber\\
 & \geq\mathbb{P}\Big(\bigcap_{b\in \{+,-\}}\Big\{\sup_{t\in [M^{-1},1]} \tilde{Y}^{(-1)}_{0,M}(b,t)\leq -\delta\Big\}\Big) \\ &\times \prod_{q:|p-q|\leq \Gamma, p\neq q}     \mathbb{P}\Big(\bigcap_{b\in \{+,-\}}\Big\{\inf_{t \in [M^{-1},1]}\tilde{Y}^{(-1)}_{q-p,M}(b,t)\leq \rho(|p-q|)\delta\Big\}\Big)\label{eq:Latline}
  \end{align}
   where the last inequality follows by applying the Slepian's inequality for the Gaussian processes. 
   
   Recall that $\tilde{Y}^{(-1)}_{q-p,M}(+,\cdot)$ and $\tilde{Y}^{(+1)}_{q-p,M}(-,\cdot)$ are independent centered Gaussian processes. As a result, we have 
   \begin{align*}
   \mathbb{P}\Big(\bigcap_{b\in \{+,-\}}\Big\{\inf_{t \in [M^{-1},1]}\tilde{Y}^{(-1)}_{q-p,M}(+,t)\leq \rho(|p-q|)\delta\Big\}\Big)\geq \frac{1}{4}.
\end{align*}      
Plugging this bound in the last line of \eqref{eq:Latline} yields 
\begin{align*}
\text{r.h.s. of \eqref{eq:Latline}} &\geq 2^{-4\Gamma} \mathbb{P}\Big(\bigcap_{b\in \{+,-\}}\Big\{\sup_{t\in [M^{-1},1]}\tilde{Y}^{(-1)}_{0,M}(b,t)\leq -\delta\Big\}\Big) \\ &= 2^{-4\Gamma} \mathbb{P}\Big(\Big\{\sup_{t\in [M^{-1},1]}\tilde{Y}^{(-1)}_{0,M}(+,t)\leq -\delta\Big\}\Big)^2
\end{align*}
for all large $n$. The equality in the last display follows since $\tilde{Y}^{(-1)}_{0,M}(+,\cdot)$ and $\tilde{Y}^{(-1)}_{0,M}(-,\cdot)$ are independent Gaussian process and their marginal laws are same. This completes the proof.

\end{proof}

%
%
%
%
 \subsubsection{Upper Bound on $\P(\neg B_{2p})$}
 \begin{lem}\label{lem:B2p}
Recall $B_{2p}$ from \eqref{eq:b2p}. For all large $n$, we have 
\begin{align}
\mathbb{P}\big(\neg B_{2p}\big)\lesssim M^{-\Gamma+1}.
\end{align}
\end{lem} 
  
\begin{proof} 
 Recall that 
 \[B_{2p} = \bigcap_{q:|p-q|>\Gamma} B^{(q)}_{2p}, \quad \text{where } B^{(q)}_{2p}= \Big\{ \sup_{u\in \tilde{\mathcal{J}}^{(-)}_q}\frac{\tilde{Q}_n^{(-1),p}(\pm e^{-u})}{\sigma_n(-u)}<\delta\rho(p-q)\Big\}.\] 
Notice that 
$$\neg B^{(q)}_{2p}\subset  \Big\{\sup_{u\in \tilde{\mathcal{J}}^{(-)}_q}|\tilde{Q}_{n}^{(-1),p}(e^{-u})|>\frac{\delta \rho(p-q)}{\sqrt{s_{q-1}^{\alpha+1}}}\Big\}.$$
where $s_{q-1} := (\max_{u \in \tilde{\mathcal{J}}^{(-)}_q}\sigma^2_n(u))^{-1/(\alpha+1)}$. 
 For fixed $p,q\in [1,N]$ with $p\ge q+2$,  
 applying part (a) of Lemma \ref{lem:defer} with $[c,d]=\tilde{\mathcal{J}}^{(-)}_q$ gives
 \begin{align}\label{eq:QsqBd1}
 \P\Big(\sup_{u\in \tilde{\mathcal{J}}^{(-)}_q}|\tilde{Q}_{n}^{(-1),p}(e^{-u})|>\frac{\delta \rho(p-q)}{\sqrt{s_{q-1}^{\alpha+1}}}\Big)&\lesssim\frac{s_{q-1}^{\alpha+1}}{\delta^2\rho^2(p-q)}\left[s^2_{n,q}s_{n,p}^{\alpha+3}+s_{n,p}^{\alpha+1}\right]
\nonumber \\ &\le   \frac{M^{(q-p+1)(\alpha+3)}+M^{(q-p+1)(\alpha+1)}}{\delta^2\rho^2(p-q)}.
 \end{align}
 where the second inequality follows since $s_{r} =O(M^{-r}/L\log n)$ for all $r\in [1,T_n]$ and $s_{n,p}= LM^p\log n$.
 On the other hand, for $p,q\in [1,T_n]$ such that $p\le q-2$, using part (b) of Lemma \ref{lem:defer} with $[c,d]=\tilde{\mathcal{J}}^{(-)}_q$ we have
  \begin{align}\label{eq:QsqBd2}
  \P\Big(\sup_{u\in \tilde{\mathcal{J}}^{(-)}_q}|\tilde{Q}_{n}^{(-1),p}(e^{-u})|>\frac{\delta \rho(p-q)}{\sqrt{s_{q-1}^{\alpha+1}}}\Big)&\lesssim \frac{s_{q-1}^{\alpha+1}}{\delta^2\rho^2(p-q)}e^{-M^{q-p-1}}\left[s_{n,q+1}^2s_{n,q}^{\alpha+3}+s_{n,q-1}^{\alpha+1}\right]
\\ &\le e^{-M^{q-p-1}} \frac{(M^2+1)}{\delta^2\rho^2(p-q)}.
  \end{align}
Using union bound and combining the inequalities in \eqref{eq:QsqBd1} and \eqref{eq:QsqBd2} as $q$ varies over the set $\{q:|p-q|>\Gamma\}$ yields
 \begin{align}\label{eq:a>b3}
\P(\neg B_{2p})\le 2 \sum_{q\in [1,T_n],|p-q|\ge \Gamma} \P\Big(\sup_{u\in \tilde{\mathcal{J}}^{(-)}_q}|\tilde{Q}_{n}^{(-1),p}(e^{-u})|>\frac{\delta \rho(p-q)}{\sqrt{s_{q-1}^{\alpha+1}}}\Big)\lesssim M^{-\Gamma+1}.
 \end{align}
 Note that the above inequality follows since $\sum_{i}i^2 M^{-\Gamma-i}\lesssim M^{-\Gamma+1}$.
This completes the proof.
\end{proof}

  \subsubsection{Upper Bound on $\P(\neg B_{3p})$ \& $\P(\neg B_{4p})$}
  
\begin{lem}\label{lem:B34p}
We have 
\begin{align}
\mathbb{P}(\neg B_{3p}) \lesssim_{\alpha,\delta, K}\frac{(LM^{p}\log n)^{\alpha+1}}{n^{\alpha+1}}, \qquad \mathbb{P}(\neg B_{4p}) \lesssim_{\alpha,\delta}  M^{(-\Gamma-p+2)(\alpha+3)}+M^{(-\Gamma-p+2)(\alpha+1)}.
\end{align}
\end{lem}  
  
\begin{proof} \textbf{Bound on $\P(\neg B_{3p})$:} 
We write $ A_1\cup A_2=[\frac{K}{n},\infty)= \cup_{\ell\geq K}J_\ell$ where $J_\ell=[\frac{\ell}{n},\frac{\ell+1}{n}]$ for $\ell\ge 0$. By the union bound, we have   
 \begin{align}
 \mathbb{P}(\neg B_{3p}) &\leq \underbrace{\sum_{\ell\geq 0} \P\Big(\sup_{u\in J_{\ell}}\tilde{Q}_n^{(-1),p}(e^u)\ge \delta \inf_{u\in J_{\ell}}\sigma_n(u)\Big)}_{=:(\mathbf{I})} \\ &+ \underbrace{\P\Big(\sup_{u \in A_{0}}\tilde{Q}_n^{(-1),p}(e^u)\ge \delta \inf_{u\in A_0}\sigma_n(u)\Big)}_{=:(\mathbf{II})}
 \end{align}
 Now we first bound $(\mathbf{I})$ and then, bound $(\mathbf{II})$.

Fix any $\ell\in \mathbb{Z}_{\geq 0}$. For any $u,v\in J_{\ell}$,  
\begin{align*}
\E\Big[\tilde{Q}_n^{(-1),p}(e^{u})-\tilde{Q}_n^{(-1),p}(e^{v})\Big]^2\le& (u-v)^2 \sup_{\xi\in J_\ell}\sum_{i\in \tilde{\mathcal{I}}^{(-)}_p} i^{\alpha+2}L(i) e^{2i\xi}\\
\le & (u-v)^2 \exp\Big(\frac{2\ell\log n}{nM^{p}}\Big) \sum_{i\in \tilde{\mathcal{I}}^{(-)}_p} i^{\alpha+2}L(i)\\
\lesssim_{\alpha} &(u-v)^2 \exp\Big(\frac{2\ell \log n}{n M^{p}}\Big)  \max_{i\in \tilde{\mathcal{I}}^{(-)}_p}L(i) (L\log n M^{p+1})^{\alpha+3}.
\end{align*}
A similar calculation gives
\begin{align*}
\sup_{u\in J_\ell} &\E( \tilde{Q}_n^{(-1),p}(u))^2\le \sum_{i\in \tilde{\mathcal{I}}^{(-)}_p} i^\alpha L(i) e^{2(\ell+1)i/n} \\ &\lesssim \exp\Big(\frac{2\ell \log n}{n M^{p}}\Big)  \max_{i\in \tilde{\mathcal{I}}^{(-)}_p}L(i) (L\log n M^{p+1})^{\alpha+1}.
\end{align*}
Applying part (b) of Lemma \ref{lem:kol_cent} with $\lambda = \delta\inf_{u \in J_{\ell}}\sigma_n(u)=\delta e^{n\ell}(n/\ell)^{\alpha+1}$ gives
\begin{align*}
\P\big( &\sup_{u\in J_\ell}Q_n^{(-1),p}(e^{u})>\delta \inf_{u\in J_{\ell}}\sigma_n(u)\big)\\ &\lesssim_{\alpha} \max_{i\in \tilde{\mathcal{I}}^{(-)}_p} L(i)\exp\Big(-2\ell\Big(1-\frac{L\log n M^p}{n}\Big)\Big) \frac{(L\ell\log n M^{p+1})^{\alpha+1}}{n^{\alpha+1}}.
 \end{align*}
 By summing the above bounds, we get 
 \begin{align*}
  (\mathbf{I})\lesssim_{\alpha} \frac{(L\log n M^{p+1})^{\alpha+1}}{n^{\alpha+1}}.
 \end{align*}
  Now we turn to bound $(\mathbf{II})$. For $u,v\in A_0=[-\frac{K}{n},\frac{K}{n}]$, we have 
\begin{align*}
\E\big[Q_n^{(-1), p}(e^{u})-Q_n^{(-1), p}(e^{v})\big]^2\le&(u-v)^2 \sup_{\xi\in A_0}\sum_{i\in \tilde{\mathcal{I}}^{(-)}_p} i^{\alpha+2}L(i) e^{2i\xi}\\
\le & (u-v)^2 e^{\frac{2KLM^{p}\log n}{n}}\sum_{i\in \tilde{\mathcal{I}}^{(-)}_p} i^{\alpha+2}L(i)\\
\lesssim &(u-v)^2 \max_{i\in \tilde{\mathcal{I}}^{(-)}_p}L(i) e^{\frac{2KLM^{p}\log n}{n}} (L\log n M^{p})^{\alpha+3}.
\end{align*}

A similar calculation gives
\begin{align*}
\sup_{u\in A_0}\E( Q_n^{(-1), p}(e^{u}))^2\le \max_{i\in \tilde{\mathcal{I}}^{(-)}_p}L(i) e^{\frac{2KLM^{p}\log n}{n}} (L\log n M^{p})^{\alpha+3}.
\end{align*}
Using Lemma \ref{lem:kol_cent} with $\lambda= \delta \inf_{u\in A_0}\sigma_n(-u)$ where  $\sigma_n(u)= e^{2nu} n^{\alpha+1} L(n)$ gives
\begin{align*}
\P( &\sup_{u\in A_0}Q_n^{(-1),p}(e^{-u})>\delta \inf_{u\in J_\ell}\sigma_n(-u))\\ &\lesssim \max_{i\in \tilde{\mathcal{I}}^{(-)}_p}L(i) e^{\frac{2KLM^{p}\log n}{n}+2K} (L\log n M^{p})^{\alpha+3}\frac{1}{n^{\alpha+1}L(n)}
 \end{align*}
 which implies  
 \begin{align*}
 (\mathbf{II})\le_{\alpha} e^{3K} \frac{(L\log n M^{p+1})^{\alpha+1}}{n^{\alpha+1}}.
 \end{align*}
 Combining the bound on $(\mathbf{I})$ and $(\mathbf{II})$ shows the upper bound on $\P(\neg B_{3p})$

\textbf{Bound on $\P(\neg B_{4p})$:} 
%
We now write $A_{-2}=\cup_{\ell=\Gamma}^{\infty}(-\frac{M^{\ell}}{\log n},-\frac{M^{\ell-1}}{\log n})$. Let us denote $W_{\ell}:=(-\frac{M^{\ell}}{\log n},-\frac{M^{\ell-1}}{\log n})$. 
For any $p\in \mathbb{N}$ and $\ell\geq \Gamma$, we note that $-\frac{M^{\ell-1}}{\log n}\leq -\frac{1}{\log n M^{p-\Gamma+1}}$. Applying part (a) of Lemma~\ref{lem:defer}  with $[c,d]=W_{\ell}$ gives
 \begin{align}\label{eq:QsqBd1}
 \P\Big(\sup_{u\in W_\ell}|\tilde{Q}_{n}^{(-1),p}(e^{-u})|>\frac{\delta}{\ell^2\sqrt{s_{W_\ell}^{\alpha+1}}}\Big)\lesssim  \ell^2\frac{M^{(-\ell-p+1)(\alpha+3)}+M^{(-\ell-p+1)(\alpha+1)}}{\delta^2}.
 \end{align}
Therefore by the union bound, we get 
\begin{align*}
    \P\Big(\sup_{u\in A_{-2}}\frac{|\tilde{Q}_{n}^{(-1),p}(e^{-u})|}{\sigma_n(u)}>\delta\Big)& \leq \sum_{\ell=\Gamma}^{\infty} \P\Big(\sup_{u\in W_\ell}|\tilde{Q}_{n}^{(-1),p}(e^{-u})|>\frac{\delta}{\ell^2\sqrt{s_{W_\ell}^{\alpha+1}}}\Big)\\ & \lesssim  \sum_{\ell=\Gamma}^{\infty}\ell^2\frac{M^{(-\ell-p+1)(\alpha+3)}+M^{(-\ell-p+1)(\alpha+1)}}{\delta^2}\\
    & \lesssim \frac{M^{(-\Gamma-p+2)(\alpha+3)}+M^{(-\Gamma-p+2)(\alpha+1)}}{\delta^2}
\end{align*}
This proves the desired claim.

 \end{proof}


\subsubsection{Proof of Proposition~\ref{prop:r=-1}}
 Recall the last line of the inequality \eqref{eq:b-1} which says 
 \begin{align*}
  \P\Big( &\frac{Q^{(-1)}_n(\pm e^{u})}{\sigma_n(u)}\leq -\delta, u \in A_{-1},\frac{Q^{(-1)}_n(\pm e^{u})}{\sigma_n(u)}<\frac{\delta}{4}, u \in A_{-1}\Big)
  \\ &\geq \P(B_{1p}) - \P(\neg B_{2p}) - \P(\neg B_{3p}) - \P(\neg  B_{4p}). 
\end{align*}  
 Substituting the the lower bound on $\P(B_{1p})$ from Lemma~\ref{lem:B1p} and the upper bounds on $\P(\neg B_{2p})$, $\P(\neg B_{3p})$ and $\P(\neg  B_{4p})$ from Lemma~\ref{lem:B2p} and~\ref{lem:B34p} into the right hand side of the above display completes the proof.

 \subsection{Proof of Proposition~\ref{prop:r=1}: $r=1$ Case}\label{sec:r=1}
 We divide $B_{1}=(n-\frac{n}{D}, n-L \log n]$ into $T_n=\lceil \frac{\log \frac{nh}{K}}{\log M}\rceil$ many sub-intervals. Recall the set of intervals $\{\tilde{\mathcal{I}}^{(+)}_{r}\}_{1\leq r\leq T_n}$ where 
$$\tilde{\mathcal{I}}^{(+)}_{r} = \mathbb{Z}\cap (n - L M^{r}\log n , n - LM^{r-1}\log n ], \quad 1\leq r\leq T_n,$$
and note that 
$$B_{1} \subset  \cup^{T_n}_{r=1} \tilde{\mathcal{I}}^{(+)}_{r}.$$

 As in \eqref{eq:tildeQ}, we define 
 \begin{align}
 \tilde{Q}_n^{(+1),p}(x):=\sum_{i\in \tilde{\mathcal{I}}_p^{(+)}} a_i x^i\text{ if }2\le p\le N-2, \quad 1\leq p\leq T_n.
 \end{align}
     As a result, we have $\tilde{Q}^{(+1)}_n(x) = \sum_{p=1}^{T_n} \tilde{Q}^{(+1),p}_n(x)$ and in the same spirit as in \eqref{eq:b-1}, we get 
 \begin{align*}
 \P &\Big(\frac{Q_n^{(+1)}(\pm e^u)}{\sigma_n(u)}<-\delta, u\in A_{1}, \frac{Q_n^{(1)}(\pm e^u)}{\sigma_n(u)}<\delta/4, u\notin A_{1}\Big)
\\ &\ge \prod_{p=1}^N\Big[\P(B^{+}_{1p})-\P(\neg B^{+}_{2p})-\P(\neg B^{+}_{3p})-\P(\neg B^{+}_{4p})\Big]
 \end{align*}
 where  
 \begin{align}
 \label{eq:b1p}\tilde{B}^{+}_{1p}:=&\bigcap_{q:|q-p|\le \Gamma} \tilde{B}^{(q)}_{1,p}, \quad \text{where }\tilde{B}^{(q)}_{1p}:=\Big\{\sup_{u\in I_q}\frac{\tilde{Q}_n^{(+1),p}(\pm e^{-u})}{\sigma_n(-u)}<-2\delta, |p-q|\le \Gamma\},\\
 \label{eq:b2p}\tilde{B}^{+}_{2p}:=&\bigcap_{q:|q-p|>\Gamma}\tilde{B}^{(q)}_{2p}, \quad \text{where }\tilde{B}^{(q)}_{2p}:=\Big\{ \sup_{u\in I_q}\frac{\tilde{Q}_n^{(+1),p}(\pm e^{-u})}{\sigma_n(-u)}<\delta\rho(p-q)\Big\},\\
 \label{eq:b3p}\tilde{B}^{+}_{3p}:=&\Big\{\sup_{u\in A_0\cup A_{-1}\cup A_{-2}}\frac{\tilde{Q}_n^{(+1),)}(\pm e^{u})}{\sigma_n(u)}<\delta\rho(N+1-p)\Big\},\\
\label{eq:b4p}  \tilde{B}^{+}_{4p}:=&\Big\{\sup_{u\in A_{2} }\frac{\tilde{Q}_n^{(+1),p}(\pm e^{-u})}{\sigma_n(-u)}<\delta \rho(p)\Big\}.
 \end{align}
Recall that $\rho(\ell) = \frac{1}{\ell^2}$ for any $\ell \in \mathbb{Z}\backslash \{0\}$.

Proposition~\ref{prop:r=1} will be proved by showing lower bound to $\P(\tilde{B}^{+}_{1p})$ and upper bounds to $\P(\neg \tilde{B}^{+}_{2p})$, $\P(\neg \tilde{B}^{+}_{3p})$ and $\P(\neg \tilde{B}^{+}_{4p})$. We show these as follows by using Lemma~\ref{lem:weak*},~\ref{lem:weak*Alt} and~\ref{lem:kol_cent}.

 \subsubsection{Lower Bound on $\P(\tilde{B}^{+}_{1p})$} 
 
 \begin{lem} 
 Consider the Gaussian process $\{\tilde{Y}^{(+)}_{0,M}(b,t)\}_{b\in \{-,+\}, t\in [1,M]}$ defined in \ref{def:tildeY}. For all large $n$, we have 
 \begin{align}\label{eq:B+1p}
 \P(\tilde{B}^{+}_{1p}) \geq 2^{-4\Gamma } \P\Big(\sup_{t\in [1,M]}\tilde{Y}^{(+)}_{0,M}(b,t)\leq -\delta\Big)^{2}
\end{align}  
 \end{lem}

 \begin{proof}
 Recall the process $\tilde{Y}^{(+1)}_{0,M}(b,t)$ from Definition~\ref{def:tildeY}. By Lemma~\ref{lem:weakAlt}, we know 
 \begin{align}\label{eq:MyEq1}
\Big\{u\in \tilde{J}^{(+)}_{q}, b\in \{+,-\}: \frac{\tilde{Q}^{(+1),p}_n(be^{u})}{\sigma_n(u)}\Big\} \stackrel{d}{\to} \Big\{t\in [1,M],b\in \{+,-\}: \tilde{Y}^{(+)}_{q-p,M}(b,t)\Big\}   
\end{align} 
as $n\to \infty$ where $\tilde{Y}^{(+)}_{q-p,M}(b,\cdot)$ is a centered Gaussian process as defined in \ref{def:tildeY}. We now show \eqref{eq:B+1p} using \eqref{eq:MyEq1}. This proof is very similar to the proof of Lemma~\ref{lem:B1p}. 

Now we proceed to prove \eqref{eq:PrincipleEvent} using \eqref{eq:MyEq}. By the weak convergence, for all large $n$
  \begin{align}
   \mathbb{P} &\Big( \tilde{B}^{+}_{1p} \cap \tilde{B}^{+}_{2p} \cap \tilde{B}^{+}_{3p} \cap \tilde{B}^{+}_{1p} \Big)
   \nonumber
   \\ &\geq \mathbb{P}\Big(\bigcap_{b\in \{+,-\}}\big\{\sup_{t\in [1,M]}\tilde{Y}^{(+1)}_{0,M}(b,t) \leq -\delta\big\}\bigcap \\ & \qquad \bigcap_{b\in \{+,-\}} \bigcap_{q:|q-p|\leq \Gamma} \Big\{\inf_{t \in [1,M]}\tilde{Y}^{(+1)}_{q-p,M}(b,t)\leq \rho(|p-q|)\delta\Big\} \Big)\nonumber\\
 & \geq\mathbb{P}\Big(\bigcap_{b\in \{+,-\}}\Big\{\sup_{t\in [1,M]} \tilde{Y}^{(+1)}_{0,M}\leq -\delta\Big\}\Big) \nonumber\\ &\times \prod_{q:|p-q|\leq \Gamma, p\neq q}     \mathbb{P}\Big(\bigcap_{b\in \{+,-\}}\Big\{\inf_{t \in [1,M]}\tilde{Y}^{(+1)}_{q-p,M}(b,t)\leq \rho(|p-q|)\delta\Big\}\Big)\label{eq:Latline}
  \end{align}
   where the last inequality follows by applying the Slepian's inequality for the Gaussian processes. Since $\tilde{Y}^{(+1)}_{q-p,M}(b,t)$ is centered Gaussian process, the product in the last line of the above display is lower bounded by $2^{-4\Gamma}$. Substituting this into the right hand side of the above inequality completes the proof. 

\end{proof}

\subsubsection{Upper bound on $\P(\neg \tilde{B}^{+}_{2p})$}
\begin{lem}
 Recall the event $\tilde{B}^{+}_{2p}$. For all large $n$,  
  \begin{align}\label{eq:B+2p}
 \P\big(\neg \tilde{B}^{+}_{2p}\big)\lesssim M^{-\Gamma}
  \end{align}  
 \end{lem}

\begin{proof}
Recall that 
$$\tilde{B}^{+}_{2p} =\bigcap_{q:|q-p|>\gamma}\tilde{B}^{(q)}_{2p}, \quad \text{where }\tilde{B}^{(q)}_{2p}:=\Big\{ \sup_{u\in I_q}\frac{\tilde{Q}_n^{(-1),p}(\pm e^{-u})}{\sigma_n(-u)}<\delta\rho(p-q)\Big\}.$$
By the union bound, $\P\big(\neg \tilde{B}^{+}_{2p}\big)$ can be bounded above by $\sum_{q:|p-q|>\Gamma}\P\big(\neg \tilde{B}^{(q)}_{2p}\big)$.
Throughout the rest of the proof, we bound $\P(\tilde{B}^{(q)}_{2p})$. We claim that for all large $n,M$ and uniformly for any $q$ such that $|p-q|>\Gamma$,
\begin{align}\label{eq:tildeBq2p}
\P(\neg \tilde{B}^{(q)}_{2p})\lesssim_{\alpha,\omega} \begin{cases}
\frac{e^{-M^{q-p-1}}}{\delta^2\rho(p-q)^2}\frac{\max_{i \in \tilde{\mathcal{I}}^{(+)}_p} L(i)}{\max_{i \in \tilde{\mathcal{I}}^{(+)}_p} L(i)} & \text{when }q\geq p+2\\
\frac{1}{\delta^2\rho(p-q)^2} M^{-(p-q-1)}\frac{\max_{i \in \tilde{\mathcal{I}}^{(+)}_p} L(i)}{\min_{i \in \tilde{\mathcal{I}}^{(+)}_p} L(i)} & \text{when }p\geq q+\Gamma
\end{cases}
\end{align}  
From \eqref{eq:tildeBq2p}, the inequality in \eqref{eq:B+2p} follows by the union bound. We divide the proof of \eqref{eq:tildeBq2p} into two stage. In \textbf{Stage 1}, we consider the case when $q\geq p+\Gamma$ (where $\Gamma\geq 2$) and in \textbf{Stage 2}, we consider the case when $p\geq q+2$.

\textbf{Stage 1:} For $q\ge p+\Gamma$ and $(u,v)\in \tilde{\mathcal{J}}^{(+)}_q$ we have
\begin{align*}
\E \Big[\frac{\tilde{Q}_n^{(+1),p}(e^u)}{e^{nu}}-\frac{Q_n^{(+1),p}(e^v)}{e^{nv}}\Big]^2=&\sum_{i\in \tilde{\mathcal{I}}^{(+)}_p}i^{\alpha}L(i)\Big[\frac{e^{2ui}}{e^{2nu}}-\frac{e^{2vi}}{e^{2nv}}\Big]^2\\
\leq & n^{\alpha}\max_{i \in \tilde{\mathcal{I}}^{(-)}_p}L(i)\sum_{\tilde{\mathcal{I}}^{(+)}_p} [e^{-2ui}-e^{-2vi}]^2\\
\lesssim & n^{\alpha} (u-v)^2 n^{\alpha}\max_{\xi \in \tilde{\mathcal{J}}^{(+)}_q} \sum_{i\in \tilde{\mathcal{I}}^{(-)}_p} e^{-2i\xi} \\ & \lesssim_{\alpha, \omega}n^{\alpha} (u-v)^2 e^{-M^{q-p-1}}\max_{i \in \tilde{\mathcal{I}}^{(+)}_{q}} L(i).
\end{align*}
A similar calculation gives
\begin{align*}
\E\Big[\frac{\tilde{Q}_n^{(+1),p}(e^u)}{e^{nu}}\Big]^2\lesssim_{\alpha,\omega} n^{\alpha} e^{-M^{q-p-1}}\max_{i \in \tilde{\mathcal{I}}^{(+)}_{q}} L(i).
\end{align*}
Using Lemma \ref{lem:kol_cent} and Chebyshev's inequality, we have
\begin{align*}
\P(\neg \tilde{B}^{(q)}_{2p}) &\leq \P\Big(\sup_{u\in \tilde{\mathcal{J}}^{(+)}_q}|e^{-nu}\tilde{Q}_n^{(+1),p}(e^u)|>\delta \rho(p-q)\inf_{u \in \tilde{\mathcal{J}}^{(+)}_q}e^{-nu}\sigma_n(u)\Big) \\
& \leq \P\Big(\sup_{u,v\in \tilde{\mathcal{J}}^{(+)}_q}\big|e^{-nu}\tilde{Q}_n^{(+1),p}(e^u)- e^{-nv}\tilde{Q}_n^{(+1),p}(e^v)\big|>\delta \rho(p-q)\inf_{u \in  \tilde{\mathcal{J}}^{(+)}_q}e^{-nu}\sigma_n(u)\Big)\\ & + \P\Big(|\tilde{Q}_n^{(+1),p}(e^{M^{-q}/L\log n})|>\delta \rho(p-q)e^{-nM^{-q}/L\log n}\sigma_n(M^{-q}/L\log n)\Big)\\
&\lesssim_{\alpha,\omega}\frac{e^{-M^{q-p-1}}}{\delta^2\rho(p-q)^2}\frac{\max_{i \in \tilde{\mathcal{I}}^{(+)}_p} L(i)}{\min_{i \in \tilde{\mathcal{I}}^{(+)}_p} L(i)}.
\end{align*}

\textbf{Stage 2:} For $p\ge q+\Gamma$ and $(u,v)\in \tilde{\mathcal{J}}^{(+)}_{q}$, we have 
 \begin{align*}
\E \Big[\frac{\tilde{Q}_n^{(+1),p}(e^u)}{e^{nu}}-\frac{\tilde{Q}_n^{(+1),p}(e^v)}{e^{nv}}\Big]^2=&\sum_{i\in \tilde{\mathcal{I}}^{(+)}_{p}}i^{\alpha}L(i)\Big[\frac{e^{2ui}}{e^{2nu}}-\frac{e^{2vi}}{e^{2nv}}\Big]^2\\
\leq &n^{\alpha}\max_{i \in \tilde{\mathcal{I}}^{(-)}_{p}} L(i)\sum_{i\in \tilde{\mathcal{I}}^{(-)}_{p}} \big[e^{-2ui}-e^{-2vi}\big]^2\\
\leq & (u-v)^2n^{\alpha}\max_{i \in \tilde{\mathcal{I}}^{(-)}_{p}} L(i)\max_{\xi \in \tilde{\mathcal{J}}^{(+)}_{p}}  \sum_{i \in \tilde{\mathcal{I}}^{(-)}_{p}} e^{-2i\xi}\\
\lesssim & (u-v)^2n^{\alpha}\max_{i \in \tilde{\mathcal{I}}^{(-)}_{p}} L(i) M^{p}L\log n.
\end{align*}
By a similar computation, we get 
\begin{align*}
\E\Big[\frac{\tilde{Q}_n^{(+1),p}(e^u)}{e^{nu}}\Big]^2\le n^{\alpha}\max_{i \in \tilde{\mathcal{I}}^{(-)}_{p}} L(i) M^{p}L\log n.
\end{align*}

Using Lemma~\ref{lem:kol_cent} and Chebyshev's inequality,  we have 
\begin{align*}
\P(\neg  & \tilde{B}^{(q)}_{2p}) \leq \P\Big(\sup_{u\in \tilde{\mathcal{J}}^{(+)}_q}|e^{-nu}\tilde{Q}_n^{(+1),p}(e^u)|>\delta \inf_{u \in \tilde{\mathcal{J}}^{(+)}_q} e^{-nu}\sigma_n(u)\rho(p-q)\Big)\\ & 
 \leq \P\Big(\sup_{u,v\in \tilde{\mathcal{J}}^{(+)}_q}|e^{-nu}\tilde{Q}_n^{(+1),p}(e^u)- e^{-nv}\tilde{Q}_n^{(+1),p}(e^v)|>\delta \inf_{u \in \tilde{\mathcal{J}}^{(+)}_q} e^{-nu}\sigma_n(u)\rho(p-q)\Big) \\
 & + \P\Big(| e^{-nv}\tilde{Q}_n^{(+1),p}(e^v)|>\delta \inf_{u \in \tilde{\mathcal{J}}^{(+)}_q} e^{-nu}\sigma_n(u)\rho(p-q)\Big)
\\
&\lesssim_{\alpha, \omega} \frac{1}{\delta^2\rho(p-q)^2} M^{-(p-q-1)} \frac{\max_{i \in \tilde{\mathcal{I}}^{(+)}_p} L(i)}{\min_{i \in \tilde{\mathcal{I}}^{(+)}_p} L(i)}. 
\end{align*} 
Now we complete showing the bound on $\P(\neg \tilde{B}^{+}_{2p})$. Recall that $\rho(i)=(3/\pi^2)^{-1}\frac{1}{i^2}$ for $i\in \mathbb{Z}\backslash \{0\}$. By the union bound,
\begin{align*}
\P(\neg \tilde{B}^{+}_{2p}) \leq \sum_{q:|q-p|>\Gamma} \P(\neg \tilde{B}^{(q)}_{2p}) & \lesssim_{\delta} \sum_{q\geq p+\Gamma} |p-q|^2 e^{-M^{q-p-1}} + \sum_{p\geq q+\Gamma} |p-q|^2 M^{-(p-q-1)}\\ & \lesssim_{\delta} M^{-\Gamma}. 
\end{align*}
This completes the proof of \eqref{eq:B+2p}.
\end{proof} 

 \subsubsection{Upper bound on $\P(\neg \tilde{B}^{+}_{3p})$ \& $\P(\neg \tilde{B}^{+}_{4p})$}
 
  \begin{lem}
   For all large $n$, we have 
   \begin{align}\label{eq:B+43p}
   \P\big(\neg \tilde{B}^{+}_{4p}\big)\leq e^{-hD}, \quad \P\big(\neg \tilde{B}^{+}_{3p}\big)\leq \frac{1}{\delta^2} \frac{K}{M^{(p-1)(\alpha+1)}} + \frac{1}{\delta^2 M^{p-1}} e^{-K}.
   \end{align}
  \end{lem}

  \begin{proof}
  The proof is divided in two stages: in \textbf{Stage 1}, we prove the bound on $\P(\tilde{B}^{+}_{4p})$ and \textbf{Stage 2} will contain the bound on $\P(\tilde{B}^{+}_{3p})$. 
  
\noindent \textbf{Stage 1:} From the definition of $\tilde{B}^{+}_{4p}$, 
\begin{align}
\P\big(\neg \tilde{B}^{+}_{4p}\big) \leq \P\Big(\sup_{u\in A_{2} }\frac{\tilde{Q}_n^{(+1),p}}{\sigma_n(u)}(\pm e^{u})>\delta \rho(p)\Big)
\end{align}   
  In what follows, we seek to bound the right hand side of the above inequality. We write $A_2:= \cup^{\infty}_{\ell= \Gamma}\tilde{W}_{\ell} $ where $W_\ell:=[\frac{M^{\ell-1}}{\log n}, \frac{M^{\ell}}{\log n})$. For any fixed $\ell\in \mathbb{Z}_{\geq 1}$ and $u,v\in J_{\ell}$, we have
 \begin{align*}
 \E[e^{-nu}\tilde{Q}_n^{(+1),p}(e^{u})-e^{-nv}\tilde{Q}_n^{(+1),p}(e^{v})]^2\le & n^{\alpha}(u-v)^2 \sup_{\xi\in J_\ell}\sum_{i\in \tilde{\mathcal{I}}^{(+)}_p} L(i)(n-i)^2 e^{-2\xi (n-i)}\\
 \le & (u-v)^2n^{\alpha}\max_{i\in \tilde{\mathcal{I}}^{(+)}_p}L(i)\max_{\xi \in J_{\ell}}\sum_{i\in \tilde{\mathcal{I}}^{(-)}_p}i^2 e^{-2i\xi}\\
 \lesssim &(u-v)^2 n^{\alpha+2}L(n) \frac{e^{2nu}}{u}e^{-2\ell/M^p}.
 \end{align*}
 A similar calculation gives
 \begin{align*}
  \E[Q_n^{(+1),p}(e^{u})]^2\lesssim n^{\alpha}L(n) \frac{e^{2nu}}{u}e^{-2\ell/M^p}.
 \end{align*}
 Using Lemma \ref{lem:kol_cent}, this gives
 \begin{align*}
 \P(\sup_{u\in J_\ell}Q_n^{(+1), p}(e^u)>\delta \sigma_n(u)) \le e^{-2\ell/M^p}.
 \end{align*}
 A union bound then gives
  \begin{align*}
 \P(\sup_{u\in A_2}Q_n^{(+1),p}(e^u)>\delta \sigma_n(u)) \le e^{-2nh/M^p}\le e^{-hD},
 \end{align*}
 from which the desired conclusion follows on noting that $D\gg \frac{1}{h}+M$. This proves the bound on $\P(B^{+}_{4p})$ of \eqref{eq:B+43p}.

\noindent \textbf{Stage 2:} 
 By the union bound, we write 
 \begin{align*}
 \P(\neg B^{+}_{3p}) & \leq \underbrace{\P\Big(\sup_{u\in A_0}\frac{\tilde{Q}_n^{(+1),p}(\pm e^{u})}{\sigma_n(u)}>\delta\rho(N+1-p)\Big)}_{(\mathbf{I})} \\ & + \underbrace{\P\Big(\sup_{u\cup A_{-1}\cup A_{-2}}\frac{\tilde{Q}_n^{(+1),p}(\pm e^{u})}{\sigma_n(u)}>\delta\rho(N+1-p)\Big)}_{(\mathbf{II})}.
 \end{align*}
 In what follows, we seek to bound $(\mathbf{I})$ and $(\mathbf{II})$ separately. We claim and prove that for all large $n$,
 $$(\mathbf{I})\leq \frac{1}{\delta^2}\frac{K}{M^{(p-1)(\alpha+1)}}, \quad (\mathbf{II})\leq \frac{1}{\delta^2} \frac{e^{-K}}{M^{p-1}}.$$

\noindent \textbf{Bound on $(\mathbf{I})$:} For $u,v\in J_{\ell}$ for $\ell\in \mathbb{Z}_{[0,K]}$ we have
 \begin{align*}
 \E[Q_n^{(+1),p}(e^{u})-Q_n^{(+1),p}(e^{v})]^2\le &(u-v)^2 \sup_{\xi\in J_\ell}\sum_{i\in \widetilde{I}_p} i^{\alpha+2} L(i) e^{2\xi i}\\
 \le & (u-v)^2e^{2\ell/M^p}\sum_{i\in \widetilde{I}_p} i^{\alpha+2}L(i)\\ & \lesssim (u-v)^2 \frac{n^{\alpha+3}}{M^{(p-1)(\alpha+3)}}L(n) e^{2\ell/M^p}.
 \end{align*}
  A similar calculation gives
 \begin{align*}
  \sup_{u\in J_{\ell}}\E[Q_n^{(+1),p}(e^{u})^2]\lesssim & \frac{n^{\alpha+1}}{M^{(p-1)(\alpha+1)}}L(n) e^{2\ell/M^p}.
 \end{align*} 
 
 Noting that $\sigma^2_n(u)=n^{\alpha}L(n)e^{2nu}/u$ gives
 \begin{align*}
(\mathbf{I})\leq 2\sum_{\ell\in \mathbb{Z}_{[0,K]}} \P(\sup_{u\in J_{\ell}}Q_n^{(+1),p}(e^{u})>\delta \inf_{u \in J_{\ell}}\sigma_n(u)) &\lesssim \frac{1}{\delta^2}\frac{nu}{M^{(p-1)(\alpha+1)}} \\ & \le \frac{1}{\delta^2}\frac{K}{M^{(p-1)(\alpha+1)}}.
 \end{align*}
The desired bound on $(\mathbf{I})$ follows by taking $p$ large, or equivalently $D\gg K$. 
\vspace{0.2cm}

\noindent \textbf{Bound on $(\mathbf{II})$:} Note that $A_{-1}\cup A_{-2}=(-\infty, -\frac{K}{n}]$. For $u,v\ge \frac{K}{n}$ we have

\begin{align*}
 \E[Q_n^{(+1),p}(e^{-u})-Q_n^{(+1),p}(e^{-v})]^2\le &(u-v)^2 \sup_{\xi\in J_\ell}\sum_{i\in \widetilde{I}_p} i^{\alpha+2} L(i) e^{-2\xi i}\\
 \le & (u-v)^2 \sup_{u\in J_\ell} L(1/u) u^\delta \sum_{i\in \widetilde{I}_p} i^{\alpha+2+\delta} e^{-2ui}\\
 \le& (u-v)^2 L(1/u) u^\delta n^{\alpha+3+\delta} e^{-2u(n-n/M^p)}\frac{1}{M^{p-1}}.
   \end{align*}
  A similar calculation gives
 \begin{align*}
  \E[Q_n^{(+1),p}(e^{-u})^2]\lesssim & L(1/u) u^\delta n^{\alpha+1+\delta} e^{-2u(n-n/M^p)}\frac{1}{M^{p-1}}.
 \end{align*}
 
 Noting that $\sigma^2_n(u)=L(1/u)/u^{\alpha+1}$ gives
 \begin{align*}
 \P(Q_n^{(+1),p}(e^{-u})>\delta \sigma_n(u))\lesssim \frac{1}{\delta^2} (nu)^{\alpha+\delta+1} e^{-nu} \le \frac{1}{\delta^2 M^{(p-1)}} e^{-\ell/2}.
 \end{align*}
Summing over $\ell$ gives the bound $\frac{1}{\delta^2} \frac{e^{-K}}{M^{p-1}}$. 

  \end{proof}

\subsection{Proof of Proposition~\ref{prop:r=0}: $r=0$ Case}\label{sec:r=0}
 Proposition~\ref{prop:r=0} follows from the following theorem.

\begin{thm}
Fix $\theta\in (0,1)$. Consider the following event:
\begin{align}
\Xi_n(K):= \Big\{n^{-1/2}Q^{(0)}_n(e^{x/n}) &\leq -\delta_x,\forall x\in \big[-K,K\big],\nonumber \\   n^{-1/2}Q^{(0)}_n(e^{x/n}) &\leq \delta_x/4, \forall x\in \big(-\infty,-K\big),\nonumber\\
 n^{-1/2}e^{-x}Q^{(0)}_n(e^{x/n}) &\leq \delta_x/4, \forall x\in \big(K,\infty\big)\Big\}
\end{align}
where $\delta_0=\delta$ and $\delta_x=\delta/\sqrt{j}$ if $x\in [-j,j+1)\cup (j-1,j]$ for any $j\in \mathbb{N}$. Then, there exists $\delta_0>0$, $M_0>0$ and $C_1,C_2>0$ such that for all $\delta<\delta_0$ and $M>M_0$, 
\begin{align}\label{eq:0liminf}
\liminf_{n\to \infty}\mathbb{P}(\Xi_n)\geq 4^{-C_1\delta^{-1-\theta}K\log K}- e^{-C_2K^2}
\end{align}
\end{thm}

\begin{proof}
We use the following shorthand notations:
\begin{align}
\Xi^{(1)}_n(K):= \Big\{n^{-1/2}Q^{(0)}_n(e^{x/n}) &\leq -\delta_x,\forall x\in \big[-K,K\big], \nonumber \\ n^{-1/2}Q^{(0)}_n(e^{x/n}) &\leq \delta_x/4, \forall x\in \big[-K^3,-K\big],\nonumber\\ n^{-1/2}e^{-x}Q^{(0)}_n(e^{x/n}) & \leq  \delta_x/4, \forall x\in \big[K,K^3\big] \Big\}\nonumber\\
\Xi^{(2)}_n(K):= \Big\{n^{-1/2}Q^{(0)}_n(e^{x/n}) &\leq \delta_x/4,\forall x\in (-\infty,-K^3], \nonumber \\ n^{-1/2}e^{-x}Q^{(0)}_n(e^{x/n}) & \leq \delta_x/4, \forall x\in [K^3,+\infty])\Big\}\nonumber
\end{align}
It is straightforward to see that $\Xi_n(K)= \Xi^{(1)}_n(K)\cup \Xi^{(2)}_n(K)$. Thus, $\mathbb{P}(\Xi_n(K))$ is bounded below by $\mathbb{P}(\Xi^{(1)}_n(K))-\mathbb{P}(\neg\Xi^{(2)}_n(K))$. In what follows, we show there exist $\delta_0,M_0>0$ and $C_1,C_2>0$ such that for all $\delta<\delta_0$ and $M>M_0$,
\begin{align}\label{eq:SplitBound}
\underbrace{\liminf_{n\to \infty}\mathbb{P}(\Xi^{(1)}_n(K))\geq e^{-C_1\delta^{-1-\theta}K\log K}}_{\mathfrak{LimInf}}, \quad \underbrace{\limsup_{n\to \infty}\mathbb{P}(\neg \Xi^{(2)}_n(K))\leq e^{-C_2K^2}}_{\mathfrak{LimSup}}
\end{align}
where $\neg \Xi^{(2)}_n$ is the complement of the event $\Xi^{(2)}_n$.
Combining the bounds on $\liminf_{n\to \infty}\mathbb{P}(\Xi^{(1)}_n)$ and $\limsup_{n\to \infty}\mathbb{P}(\Xi^{(2)}_n)$ and substituting those into the inequality $\mathbb{P}(\Xi)\geq \mathbb{P}(\Xi^{(1)}_n)-\mathbb{P}(\neg\Xi^{(2)}_n)$ proves \eqref{eq:0liminf}.

\textsf{Proof of $\mathfrak{LimInf}$:}
Lemma~\ref{lem:Convergence} shows that $\{n^{-1/2}Q^{(0)}_n(e^{x/n}): x\in \big[-K^3,K^3\big]\}$ weakly converges to $\{Y^{K}_0(x): x\in \big[-K^3,K^3\big]\}$ as $n\to \infty$. Thus, $\liminf_{n\to \infty}\mathbb{P}(\Xi^{(1)}_n)$ can be bounded below by $\mathbb{P}(\mathfrak{F})$ (see \eqref{eq:FDef} for the definition of $\mathfrak{F}$). The lower bound of $\liminf_{n\to \infty}\mathbb{P}(\Xi^{(1)}_n)$ now follows from Lemma~\ref{lem:ProbLowBd}.

\textsf{Proof of $\mathfrak{LimSup}$:}
Since $\lim_{n\to \infty} L\big(n-\frac{n}{D}\big)/ L\big(n/D\big) = 1$, we get that for any $x,y >0$,
\begin{align*}
\mathbb{E}\Big[n^{-1-\alpha}e^{-x-y}Q^{(0)}_n(e^{x/n})Q^{(0)}_n(e^{y/n})\Big]\sim L(n) e^{-(x+y)/D}\frac{1-e^{-(1-1/D)(x+y)}}{x+y}.
\end{align*} 
Similarly, for any $x,y<0$,
\begin{align*}
\mathbb{E}\Big[n^{-1-\alpha}Q^{(0)}_n(e^{x/n})Q^{(0)}_n(e^{y/n})\Big]\sim D^{-\alpha}e^{(x+y)/D}\frac{1-e^{(1-1/D)(x+y)}}{|x+y|}.
\end{align*}
As a consequence, we get 
\begin{align*}
n^{-1-\alpha}\mathbb{E} &\Big[\Big(\tfrac{Q^{(0)}_n(e^{x/n})}{e^{x}}- \tfrac{Q^{(0)}_n(e^{y/n})}{e^{y}}\Big)^2\Big] \\ & \leq \frac{C}{\min\{x,y\}}e^{-K^3/D-\min\{(x-K^3),(y-K^3)\}/D}|x-y|^2,\quad \forall x,y>K^3\\
n^{-1 - \alpha}\mathbb{E} &\Big[\Big(Q^{(0)}_n(e^{x/n})- Q^{(0)}_n(e^{y/n})\Big)^2\Big] \\ & \leq \frac{C}{\min\{|x|,|y|\}}e^{-K^3/D -\min\{(|x|-K^3),(|y|-K^3)\}/D}|x-y|^2,\quad \forall x,y<-K^3
\end{align*} 
 This upper bound in conjunction with Kolmogorov-Centsov's type argument shows there exists $C_2>0$ such that for all large integer $n$ and $K>0$, 
 \begin{align}
 \mathbb{P}\big(\neg\Xi^{(2)}_n\big)\leq e^{-C_2K^3/D}.
 \end{align}
 This completes the proof of \eqref{eq:SplitBound} and hence, completes the proof the result.
\end{proof}

\begin{lem}\label{lem:Tightness}
Consider the stochastic process $\big\{Q^{(0)}_n(e^{x/n})/\sqrt{n}: x\in [-K^3,K^3]\big\}$. Then, there exists $C=C(M)>0$ such that for all $s>0$, 
\begin{align}
\mathbb{P}\Big(\max_{x\in [-K^3,K^3]} \frac{1}{n^{3/2}}\Big|\frac{d}{dz}Q^{(0)}_n(z)\big|_{z= e^{x/n}}\Big|\geq s\Big)\leq \frac{C}{s^2}.
\end{align}
\end{lem}

\begin{proof}
For any $x\in [-K^3,K^3]$,
\begin{align}
\frac{d}{dz}Q^{(0)}_{n}(z)\big|_{z=e^{x/n}} = \sum_{i=\frac{n}{K}}^{n-\frac{n}{D}}\sum_{j=\frac{n}{D}}^{i}a_i\big(ie^{(i-1)x/n}- (i+1)e^{ix/n}\big) + (n+1)e^{x}\sum^{n-\frac{n}{K}}_{j=\frac{n}{M}}a_j. 
\end{align}
By the Kolmogorov's maxmimal inequality, there exists $c_1=c_1(M)>0$ such that 
\begin{align}
\mathbb{P}\Big(\max_{\frac{n}{D}\leq i\leq n-\frac{n}{D}}\Big|\sum_{j=\frac{n}{D}}^{i} a_i\Big|\geq L(n)s\big(\sum_{i=n/D}^{n - n/D} i^{\alpha}\big)^{\frac{1}{2}}\Big)\leq \frac{c_1}{s^2}.
\end{align}   
Note that there exists $C=C(D, K)>0$ such that $\sup_{x\in [-K^3,K^3]}\max_{1\leq i\leq n}\big|ie^{(i-1)x/n}-(i+1)e^{ix/n}\big|\leq C$. Furthermore, we have  $\sum_{i=n/D}^{n - n/D} i^{\alpha} = (n- \frac{n}{D})^{1+\alpha} - \big(\frac{n}{D}\big)^{1+\alpha}$.
Combining this with the inequality in the above display shows that there exists $c_2=c_2(K, D)>0$ such that 
\begin{align}
\mathbb{P}\Big(\sup_{x\in [-K^3,K^3]}\Big|\frac{d}{dz}Q^{(0)}_{n}(z)\big|_{z= e^{x/n}}\Big|\geq L(n)sn^{(1+\alpha)/2}\Big)\leq \frac{c_2}{s^2}.
\end{align}
This completes the proof.   
\end{proof}

\begin{lem}\label{lem:Convergence}
Then the sequence of stochastic processes $\{Q^{(0)}_n(e^{x/n})/[L(n)n^{\frac{\alpha+1}{2}}]:x\in [-K^3,K^3]\}$ is uniformly equi-continuous and converges to the Gaussian process $\{Y^{D}_0(x):x\in [-K^3,K^3]\}$ as $n\to \infty$ where $Y^{D}_{0}$ has the following covariance structure:
\begin{align*}
\mathbb{E}\Big[Y^{D}_{0}(x)Y^{D}_0(y)\Big]= \int^{1-\frac{1}{D}}_{\frac{1}{D}} t^\alpha e^{t(x+y)}dt.
\end{align*} 

Fix any interval $[a,b]\subset [-K^3,0]$. Then, there exists $C>0$ such that for any $s>0$,
\begin{align}\label{eq:Slope}
\mathbb{P}\Big(\sup_{x\neq y \in [a,b]}\frac{|Y^{D}_0(x) - Y^{D}_0(y)|}{|x-y|}\geq \frac{s}{\max\{|b|,1\}}\Big)\leq \frac{C}{s^2}.
\end{align}
Similarly, for any $[b,a]\subset [0,K^3]$
\begin{align}\label{eq:Slope2}
\mathbb{P}\Big(\sup_{x\neq y \in [a,b]}\frac{|e^{-x}Y^{D}_0(x) - e^{-y}Y^{D}_0(y)|}{|x-y|}\geq \frac{s}{\max\{|b|,1\}}\Big)\leq \frac{C}{s^2}.
\end{align}
\end{lem}

\begin{proof}
By the mean value theorem, for all $n\geq 1$.
\begin{align}
\sup_{x\neq y\in [-K^2,K^2]} \frac{|Q^{(0)}(e^{x/n})- Q^{(0)}_n(e^{y/n})|}{|x-y|} \leq \sup_{x\in [-K^2,K^2]} n^{-1}\Big|\frac{d}{dz} Q^{(0)}(z)\big|_{z=e^{x/n}}\Big|.
\end{align}
Combining this with Lemma~\ref{lem:Tightness}, there exists $C=C(M)>0$ such that for all $s>0$,
\begin{align}\label{eq:EqCont}
\mathbb{P}\Big(\max_{x\neq y\in [-K^2,K^2]}\frac{|Q^{(0)}(e^{x/n})- Q^{(0)}(e^{y/n})|}{|x-y|}\geq L(n)sn^{(1+\alpha)/2} \Big)\leq \frac{C}{s^2}.
\end{align}
This shows that $\{Q^{(0)}_n(x)/[L(n)n^{\frac{1+\alpha}{2}}]\in [-K^3,K^3]\}$ is uniformly equi-continuous as $n$ goes to $\infty$. Moreover, for any $x,y\in [-K^3,K^3]$,
\begin{align}
\mathbb{E}\big[Q^{(0)}_n(e^{x/n})Q^{(0)}(e^{y/n})\big] = \sum_{i=\frac{n}{K}}^{n - \frac{n}{K}} L(i) i^{\alpha} e^{i(x+y)/n} \stackrel{n\to \infty}{\rightarrow} \int^{1-\frac{1}{K}}_{\frac{1}{K}} t^\alpha e^{t(x+y)}dt.
\end{align} 
By Lindeberg-Feller's theorem, for any $x_1,\ldots ,x_n\in [-K^3,K^3]$,
\begin{align}
\big(Q^{(0)}_n(e^{x_1/n}), Q^{(0)}_n(e^{x_2/n}), \ldots , Q^{(0)}_n(e^{x_1/n})\big) \stackrel{d}{\to}\big(Y^{K}_0(x_1), \ldots , Y^{K}_0(x_n)\big)
\end{align}
where $Y^{K}_0$ is the same Gaussian process as stated in the lemma. The above display shows the finite dimensional convergence of process $\{Q^{(0)}_n(e^{x/n}): x\in [-K^3,K^3]\}$. Allying this with the tightness in \eqref{eq:EqCont} yields the weak convergence.  

It remains to show \eqref{eq:Slope} and \eqref{eq:Slope2}. We only show \eqref{eq:Slope}. The proof of \eqref{eq:Slope2} follows from similar argument. By the mean-value theorem, for any $x,y \in [b_1,b_2]\subset [-K^3,0]$ 
\begin{align}
\big|Q^{(0)}_n(e^{x/n})- Q^{(0)}_n(e^{y/n})\big| & \leq \frac{|x-y|}{n}\sup_{w\in [b_1,b_2]} \Big|\frac{d}{dz} Q^{(0)}_n(z)\big|_{z= e^{w/n}}\Big|
\\&\leq \frac{|x-y|}{n}\sup_{n/D\leq i\leq n- n/D} \\ &  \times  |\sum_{j=n/M}^{i} a_i| \sup_{w\in [b_1,b_2]}\Big(\sum_{j=n/D}^{n-n/D}|je^{(j-1)w/n}- (j+1)e^{jw/n}|\Big)\label{eq:SlopeQ}.
\end{align}
where the last inequality follows by expanding $\frac{d}{dz} Q^{(0)}_n(z)\big|_{z= e^{w/n}}$. Approximating the sum $\sum_{j=n/D}^{n-n/D}|je^{(j-1)w/n}- (j+1)e^{jw/n}|$ by the integral $(n\int^{1-1/D}_{1/D}|w t e^{tw}- e^{tw}|dt+o(n))$ and substituting into the right hand side of the above display yields 
\begin{align}
\text{r.h.s. of \eqref{eq:SlopeQ}} & \leq |x-y|\sup_{n/D\leq i\leq n- n/D} |\sum_{j=n/D}^{i} a_i| \sup_{w\in [b_1,b_2]} \\ & \quad \times \big(\int^{1-1/D}_{1/D}|w t e^{tw}- e^{tw}|dt+o(1)\Big).
\end{align}
Note that there exists constant $C_1= C_1(D)>0$ such that $\int^{1-1/D}_{1/D}|w t e^{tw}- e^{tw}|dt$ can be bounded above by $C_1/\max\{|w|,1\}$. The maximum value of this lower bound as $w$ varies in $[b_1,b_2]\subset [-K^3,0]$ is $C/\max\{|b_1|,1\}$. By using this upper bound for $\sup_{w\in [b_1,b_2]} \int^{1-1/D}_{1/D}|w t e^{tw}- e^{tw}|dt$ and Doobs's maximal inequality for the martingale $M_{\ell} = \sum_{i= n/D}^{\ell} a_i$ to control the tail probability of $ \sup_{n/D\leq i\leq n- n/D} |\sum_{j=n/D}^{i} a_i|$ shows 
\begin{align}
\mathbb{P}\Big(\sup_{x\neq y \in [b_1,b_2]}\frac{\big|n^{-(1+\alpha)/2}Q^{(0)}_n(x)- n^{-(1+\alpha)/2}Q^{(0)}_n(y)\big|}{|x-y|}\geq L(n)\frac{s}{|b|}\Big)\leq \frac{C}{s^2}.
\end{align}   
Now \eqref{eq:Slope} follows from the above inequality by letting $n\to \infty$ on both sides.  
\end{proof}

\begin{lem}\label{lem:ProbLowBd}
Fix $\delta>0$. Consider the Gaussian process $\{Y^{D}_0(x):x\in [-K^3,K^3] \}$ as in Lemma~\ref{lem:Convergence}. Define 
\begin{align}\label{eq:FDef}
\mathfrak{F}:=\Big\{Y^{D}_0(x) \leq -\delta_x \forall x\in [-K,K],  &\quad Y^{D}_0(x)  \leq \delta/4 \forall x\in [-K^3,-K], \nonumber\\  e^{-x}Y^{D}_0(x) & \leq \delta_x/4 \forall x\in [K,K^3]\Big\}
\end{align}
where $$\delta_{x}:=\delta/ \sqrt{j}, \text{if }x\in [-j,-j+1)\cup (j-1,j]$$ for $j=1,\ldots \lceil M^2\rceil$. 
Then, there exists $D_0=D_0(\delta)>0$ and $b=b(\delta)>0$ such that for all $D>D_0$ 
\begin{align}\label{eq:YMlow}
\mathbb{P}(\mathfrak{F})\geq e^{-bK\log K}.
\end{align}
\end{lem}

\begin{proof}
Fix some large number $C>0$. We use the following shorthand notations: 
\begin{align*}
\mathfrak{F}_1 &:= \Big\{Y^{D}_0(x)\leq - \delta_x, x\in [-M,M]\Big\}, \\  \mathfrak{F}_2 &:= \Big\{Y^{D}_0(x) \leq \delta_x/4, x\in [-CK\log K,-K]\Big\}, \\ \mathfrak{F}_3 &:= \Big\{e^{-x}Y^{D}_0(x)\leq  \delta_x/4, x\in [K,CK\log K]\}\\
\mathfrak{F}_4 &:= \Big\{Y^{D}_0(x)\leq \delta_x/4, x\in [-K^3, -CK\log K]\Big\}  \\ \mathfrak{F}_5 &:= \Big\{e^{-x}Y^{D}_0(x)\leq \delta_x/4, x\in [CK\log K, K^3]\Big\}
\end{align*}
Notice that $\mathfrak{F}$ is equal to $\mathfrak{F}_1\cap \mathfrak{F}_2\cap \mathfrak{F}_3\cap \mathfrak{F}_4\cap \mathfrak{F}_5$. Since $Y^{M}_0$ is a Gaussian process with positive correlation, we may apply Slepian's inequality to write 
\begin{align}
\mathbb{P}(\mathfrak{F})=\mathbb{P} \Big(\mathfrak{F}_1\cap \mathfrak{F}_2\cap \mathfrak{F}_3\cap \mathfrak{F}_4\cap \mathfrak{F}_5\Big)\geq \prod_{i=1}^{5}\mathbb{P}(\mathfrak{F}_i).\label{eq:MainFSlepian}
\end{align}
Fix $\theta\in (0,1)$. In what follows, we claim and show the following: there exists $\delta_0>0$ such that for all $\delta<\delta_0$
\begin{align}\label{eq:FBounds}
\mathbb{P}(\mathfrak{F}_1)\geq 4^{-\delta^{-1-\theta}K},\quad  \mathbb{P}(\mathfrak{F}_2),\mathbb{P}(\mathfrak{F}_3)\geq 4^{-C\delta^{-1}K\log K},\quad  \mathbb{P}(\mathfrak{F}_4),\mathbb{P}(\mathfrak{F}_5)\geq \frac{1}{4}. 
\end{align}
Substituting these lower bounds on $\{\mathbb{P}(\mathfrak{F}_i)\}_{1\leq i\leq 5}$ to the right hand side of \eqref{eq:MainFSlepian} proves \eqref{eq:YMlow}. In the rest of the proof, we focus on showing \eqref{eq:FBounds}. We only show the lower bound for $\mathbb{P}(\mathfrak{F}_1)$, $\mathbb{P}(\mathfrak{F}_2)$ and $\mathbb{P}(\mathfrak{F}_4)$. The bound for $\mathbb{P}(\mathfrak{F}_3)$ and $\mathbb{P}(\mathfrak{F}_5)$ follow from similar argument for $\mathbb{P}(\mathfrak{F}_2)$ and $\mathbb{P}(\mathfrak{F}_4)$ respectively.

\textsf{Proof $\mathbb{P}(\mathfrak{F}_4)\geq 1/4$}:  Note that there exists $c_1= c_1(D)>0$ such that $\mathbb{E}[(Y^{D}_{0}(x))^2]$ is less than $c_1K^{-C}$ for all $x\in [-K^3D,-CKD\log K]$. Furthermore, it is straightforward to check that for any $x,y\in [-K^3D,-CKD\log K]$ 
\begin{align*}
\mathbb{E} \big[\big(Y^{D}_{0}(x)- Y^{D}_{0}(y)\big)^2\big] \leq c_1\frac{1}{K^{C+1}}|x-y|^2, \qquad \mathbb{E}\big[(Y^{D}_{0}(x))^2\big] \leq c_1\frac{1}{K^{C+1}} 
\end{align*} 
For any $x, y \in [- K^3D, - KD \log K]$, define $d_{Y^D_0}(x,y) = \sqrt{\mathbb{E}[(Y^{D}_{0}(x)- Y^{D}_{0}(y))^2]}$. 
By Dudley's entropy theorem \cite[Theorem~7.1]{Dudley2010}, we have 
\begin{align*}
\mathbb{E} &\Big[\sup_{x\in [-K^3D, -CKD\log K]}|Y^{D}_{0}(x)|\Big] \\ &\leq c_2K^{2-\frac{C+1}{2}}\big(1+\int^{\infty}_{0}\sqrt{\log N\big([-K^3D, -CKD \log K], d_{Y^D_0}, \varepsilon\big)} d\varepsilon\big)\\
& \leq c'_2 K^{2-\frac{C+1}{2}} 
\end{align*} 
for some $c_2,c'_2>0$. Let us define 
$$ \Delta:= - c'_2K^{2-\frac{C+1}{2}}\log K+\mathbb{E}\Big[\sup_{x\in [-K^3D, -CKD\log K]}|Y^{D}_{0}(x)|\Big].$$
Note that $\Delta<0$. As a result, we get 
\begin{align*}
\mathbb{P}& \Big(\sup_{x\in [-K^3D,-CKD\log K]}Y^{D}_0(x)>\delta/4\Big) \\&\leq \mathbb{P}\Big(\sup_{x\in [-K^3D, -CKD\log K]}Y^{D}_{0}(x)\geq\frac{\delta}{4\sqrt{K}}+\Delta\Big)\\&\leq \exp\big(-8^{-2}K^{\frac{C}{4}-3/2}\delta^2\big)
\end{align*}
where the last inequality follows by Borel-Tis inequality.
This shows $$\mathbb{P}\Big(\sup_{x\in [-K^3D,-CKD\log M]}Y^{D}_0(x)\leq \delta/4\Big)\geq 1-\exp(-K^{\frac{C}{4}-3/2}\delta).$$ For $K$ large, this lower bound is bounded below by $\frac{1}{2}$. This proves the lower bound for $\mathbb{P}(\mathfrak{F}_4)$. 

\textsf{Proof $\mathbb{P}\big(\mathfrak{F}_1\big)\geq 4^{-\delta^{-1-\theta}D}$:} By Slepian's inequality which we can apply since $Y^D_0$ is a Gaussian process,
\begin{align}
\mathbb{P}(\mathfrak{F}_1)\geq \mathbb{P}\Big(\big\{Y^{D}_0(x)\leq -\delta_x, x\in [-K,0]\big\}\Big)\mathbb{P}\Big(\big\{Y^{D}_0(x)\leq -\delta_x, x\in [0, K]\big\}\Big).
\end{align}
By symmetry of $Y^{D}_{0}([-K,0])$ between $Y^{D}_{0}([0,K])$ , it suffices to show 
\begin{align}\label{eq:YMineq}
\mathbb{P}\Big(\big\{Y^{D}_0(x)\leq -\delta_x, x\in [-K,0]\big\}\Big)\geq e^{-\delta^{-2}K}
\end{align} 
 for some constant $c>0$. We show this follows. 
 
 
 We divide the interval $[-K,0]$ into $K_1:=\lceil\delta^{-1-\theta}K\rceil$ many intervals of equal length and denote them as $\mathcal{I}_{1}, \ldots , \mathcal{I}_{K_1}$  where $\mathcal{I}_{i} := [-K+(i-1)K/K_1, -K+iK/K_1]$. By Slepian's inequality, 
 \begin{align}
 \mathbb{P}\Big(\big\{Y^{D}_0(x)\leq -\delta_x, x\in [-K,0]\big\}\Big)\geq \prod_{i=1}^{K_1} \mathbb{P}\Big(\big\{Y^{D}_0(x)\leq -\delta_x, x\in \mathcal{I}_i\big\}\Big).\label{eq:LocSlep}
 \end{align}

 We first lower bound each term of the product of the right hand side of the above display.
Applying \eqref{eq:Slope} of Lemma~\ref{lem:Convergence}, we notice
\begin{align*}
\mathbb{P}\Big(\sup_{x\in \mathcal{I}_j}|Y^{D}_0(x)- Y^{D}_0( -K+iK/K_1)|\geq \frac{K}{K_1}\frac{s}{\max\{K(1-i/K_1),1\}}\Big)\leq \frac{C}{s^2}.
\end{align*}
Fix $\xi \in (0,1/2)$. Note that $K/K_1$ is less than $\delta^{1+\theta}$. Letting $s:= \delta^{-\theta/2}\sqrt{C^{-1}}(\max\{K(1-i/K_1),1\})^{1/2-\xi}$, we see that the right hand side of the above inequality is bounded by $\delta^{\theta}(\max\{K(1-i/K_1),1\})^{-1+2\xi}$. This implies
\begin{align*}
\mathbb{P} &\Big(\sup_{x\in \mathcal{I}_j}|Y^{D}_0(x)- Y^{D}( -K+iK/K_1)|\geq \frac{\delta^{1+\theta/2}}{(\max\{K(1-i/K_1),1\})^{1/2+\xi}}\Big) \\ &\leq \delta^{\theta}(\max\{K(1-i/K_1),1\})^{-1+2\xi}. 
\end{align*}
Suppose $-K+iK_1\in [-j-1,-j]$ for some $i\in \mathbb{N}$. Since $\delta/\sqrt{j+1}\geq \delta^{1+\theta/2}/(\max\{K(1-i/K_1),1\})^{1/2+\xi}$ for all large $K$, we write 
\begin{align*}
\mathbb{P} &\Big(Y^{D}_0(x)\leq -\delta_x, x\in [-K+(i-1)K_1, -K+iK_1]\Big)\\
&\geq \mathbb{P}\Big( Y^{D}_0(-M+iK_1)\leq -\frac{2\delta}{\sqrt{j+1}}\Big)\\& - \mathbb{P}\Big(\sup_{x\in \mathcal{I}_j}|Y^{D}_0(x)- Y^{D}( -K+iK/K_1)|\geq \frac{\delta^{1+\theta/2}}{(\max\{K(1-i/K_1),1\})^{1/2+\xi}}\Big)\\
&\geq \Phi\Big(-\frac{e^{(1-i/K_1)}}{\sqrt{j}}\frac{\sqrt{K(1-i/K_1)}}{\sqrt{1-e^{-(K-1)(1-i/K_1)}}}\delta\Big) - \delta^{\theta}(\max\{K(1-i/K_1),1\})^{-1+2\xi}  
\end{align*}
where $\Phi(\cdot)$ is the cumulative distribution function of a standard normal distribution. The last line of the above display can be bounded below by $\Phi(-C\delta)-\delta^{\theta}$ for some constant $C>0$ which does not depend on $\delta$. By taking $\delta$ small, we can lower bound $\Phi(-C\delta)-\delta^{\theta}$ by $1/4$ for some $C^{\prime}>0$. This shows a lower bound to $i$-th term of the product in \eqref{eq:LocSlep}. Substituting all these bounds into the right hand side of \eqref{eq:LocSlep} shows \eqref{eq:YMineq}. This completes showing the lower bound on $\mathbb{P}(\mathfrak{F}_1)$. 

\textsf{Proof of $\mathbb{P}(\mathfrak{F}_2)\geq 4^{-C\delta^{-1}K\log K }$:} We divide the interval $[-CK\log K, -K]$ into $K_2$ many sub-intervals $\mathcal{I}^{\prime}_1,\ldots ,\mathcal{I}^{\prime}_{K_2}$ of equal length for $K_2:=\lceil \delta^{-1}(CK\log K)\rceil$ where $$\mathcal{I}^{\prime}_i:= [-CK\log K + \frac{(i-1)}{K_2}(CK\log K-K),-CK\log K + \frac{i}{K_2}(CK\log K-K) ].$$ Applying the Slepian's inequality, we may write 
\begin{align}\label{eq:F_2slepian}
\mathbb{P}\big(\mathfrak{F}_2\big)\geq \prod_{i=1}^{K_2}\mathbb{P}\Big(Y^{D}_0(x) \leq \delta_x/4, x\in \mathcal{I}^{\prime}_i\Big).
\end{align} 
In what follows, we find a lower bound to the to each term of the product in the right hand side of the above display. Fix $i \in \{1,\ldots , K_2\}$ and suppose $-CK\log K +i(CK\log K-K)/K_2\in [-j,-j+1)$ for some $j\in \mathbb{N}$. Let us denote 
\begin{align*}
U&:=  \Big\{Y^{D}_0(x) \leq \delta_x/4, x\in \mathcal{I}^{\prime}_i\Big\}, \\  U_1&:=\Big\{Y^{D}_0(-CK\log K(1-i/K_2)-Ki/K_2)\leq \delta/8\sqrt{j}\Big\}\\
U_2 &:=  \Big\{\sup_{x\in \mathcal{I}_i}\big|Y^{D}_0(x)- Y^{D}_0(-CK\log K(1-i/K_2)-Ki/K_2)\big|\leq \frac{\delta}{\aleph_i}\Big\}.
\end{align*} 
where $\aleph_i:= 8\sqrt{CK\log K(1-i/K_2)+Ki/K_2}$. 
Notice that $U_1\cap U_2\subset U$. Thus, $\mathbb{P}(U)$ can be bounded below by $\mathbb{P}(U_1)-\mathbb{P}(\neg U_2)$ where $\neg U_2$ denotes the complement of the event $U_2$. Since $Y^{D}_0(-CK\log K(1-i/K_2)-Ki/K_2)$ is a Gaussian r.v. with mean zero, $\mathbb{P}(U_1)$ is bounded below by $1/2$. To obtain a lower bound to $\mathbb{P}(U_2)$, it suffices to bound $\mathbb{P}(\neg U_2)$ from above. Due to \eqref{eq:Slope},
\begin{align*}
\mathbb{P}\Big( &\sup_{x\in \mathcal{I}^{\prime}_i}\big|Y^{D}_0(x)- Y^{D}_0(-CK\log K(1-i/K_2)+Ki/K_2)\big| \\ & \geq |\mathcal{I}_i|\frac{s}{CK\log K(1-i/K_2)-Ki/K_2}\Big)\leq \frac{C}{s^2}.
\end{align*} 
Note that $|\mathcal{I}_i|\frac{s}{CK\log K(1-i/K_2)-Ki/K_2}$ is less than $\delta$ when $$s=\sqrt{CK\log K(1-i/K_2)-Ki/K_2}.$$ This shows $\mathbb{P}(\neg U_2)$ is bounded above by the left hand side of the above display when $s$ is equal to $\sqrt{CK\log K(1-i/K_2)-Ki/K_2}$. Furthermore, by setting the following $s:=\sqrt{CK\log K(1-i/K_2)-Ki/K_2}$, the right hand side is bounded by $C(CK\log K(1-i/K_2)-Ki/K_2)^{-1}$. As a consequence, $\mathbb{P}(\neg U_2)$ is bounded above by  $C(CK\log K(1-i/K_2)-Ki/K_2)^{-1}$. Combining 
\begin{align}
\mathbb{P}(U)\geq \mathbb{P}(U_1)- \mathbb{P}(\neg U_2)\geq \frac{1}{2} - C(CK\log K(1-i/K_2)-Ki/K_2)^{-1}.
\end{align} 
For large $K$, the right side of the last inequality is bounded below by $1/4$. This provides a lower bound to the $i$-th term of the product \eqref{eq:F_2slepian}. Substituting these lower bounds into the right hand side of \eqref{eq:F_2slepian} yields the lower bound of $\mathbb{P}(\mathfrak{F}_2)$ in \eqref{eq:FBounds}.
\end{proof}

 \subsection{Proof of Proposition~\ref{ppn:r=2}: $r=2$ Case}\label{sec:r=2}
Recall that $A_{2}= [\frac{\delta}{\log n}, \infty)$ and $B_2=[0,L\log n]$.

\begin{proof}[Proof of Proposition~\ref{ppn:r=2}]
%

Recall that $\xi_i = a_i/{\sqrt{R(i)}}$. We divide the proof into two cases: $(a)$ when $\P(\xi_i \leq-\rho)>c$ and $\P(\xi_i \in [-\theta, 0] )>c$ for some $\theta \in [0, \rho)$ and $(b)$ when $\P\big(\xi_i\leq -\rho\big)>c$ and $\P\big(a_i\in [0,\rho]\big)>c$.

\noindent \textbf{Case $(a)$:} Consider the following event:
\begin{align*}
\Gamma_{Q,2} := \Big\{\xi_{2i}\leq -\rho, -\theta\leq \xi_{2i+1}\le 0, \text{when }i \text{ is odd}, i \in B_2\Big\}
\end{align*}
for some $\rho>0$ and $\eta \in (0,1)$. 
Since $\xi$'s are independent, we write 
 \begin{align}\label{eq:GammaQ2LowerBd}
 \P(\Gamma_{Q,2})\geq \prod_{i:2i \in B_{2}} \P(\xi_{2i}\leq -\rho)\P(-\theta\leq \xi_{2i+1}\leq 0)\geq c^{L\log n}.
 \end{align} 
where the last inequality follows since $\P(\xi_{i}\leq -\rho)>c,\P(-\theta\leq \xi_{i}\leq 0)>c\in (0,1)$ by our assumption and $\#\{i :i \in B_{2}\}\leq L\log n$. 
We now claim and prove that for all large $n$ and $\delta>0$ such that $\delta\leq \rho f_{\eta,\epsilon}$,
\begin{align}\label{eq:A_2LowerBd}
\P\Big(\frac{Q^{(2)}_n(\pm e^{u})}{\sigma_n(u)}\leq -\delta, \text{ for }u\geq \frac{\delta}{\log n}, \frac{Q^{(2)}_n(\pm e^{u})}{\sigma_n(u)}\leq \delta, \text{ for } u \leq \frac{\delta}{\log n} \Big) \geq \P(\Gamma_{Q,2}).
\end{align}

On the event $\Gamma_{Q,2}$, we have 
 \begin{align*}
 \sum_{i\in B_2} a_i x^i &=\frac{1}{2}\sum_{2i \in B_{2}}x^{2i}\big(\xi_{2i}\sqrt{R(2i)}+2\xi_{2i+1}x\sqrt{R(2i+1)}+\xi_{2i+2}x^2\sqrt{R(2i+1)}\big)\\ &\le- \sum_{i: 2i \in B_2}\sqrt{R(2i)} x^{2i} \big(\rho-2|x|\theta\frac{\sqrt{R(2i+1)}}{\sqrt{R(2i)}}+\frac{\sqrt{R(2i+1)}}{\sqrt{R(2i)}}\rho x^2\big) 
 \end{align*}
where $\mathbb{E}[\xi_i]=0$ and $\mathbb{E}[\xi^2_i]=1$. Fix $\varepsilon>0$ such that $\theta^2(1+\varepsilon)^2< (1-\varepsilon)$. Recall that $\lim_{i\to \infty}R(i+1)/R(i)=1$. Thus for all large $n$, we have 
   $$ (1-\varepsilon)\leq \frac{R(i+1)}{R(i)}\leq (1+\varepsilon).$$
   For large $n$, on the event $\Gamma_{Q,2}$ for $\eta: = \theta\rho^{-1}$, we have 
   \begin{align*}
       \sum_{i \in B_2} a_i x^i &\leq - \frac{\rho}{2}\sum_{i: 2i\in B_2}x^{2i}\big(1- 2\eta |x|(1+\varepsilon)+(1-\varepsilon)x^2\big) \\ &= -\frac{\rho}{2} \big(1- 2\eta |x|(1+\varepsilon)+(1-\varepsilon)x^2\big)\sum_{i: 2i\in B_2}x^{2i}
   \end{align*}
   Note that $\mathcal{P}(x) = 1-2|x|\eta(1+\varepsilon)+(1-\varepsilon)x^2$ is greater than $0$ for all $x\in \mathbb{R}$ since the discriminant of $\mathcal{P}$ is less than $0$ because $\theta^2(1+\varepsilon)^2< (1-\varepsilon)$. Therefore, $\{Q^{(2)}_n(x)\leq -\rho f_{\eta,\varepsilon}\sum_{i: 2i\in B_2} x^{2i}\}$ for all $x\in \mathbb{R}$ where $f_{\eta,\varepsilon}$ is some positive constant which only depends on $\eta$ and $\varepsilon$. 
We start by recalling that $\sigma_n(u)\sim \frac{e^{nu}}{u}$ for $u>0$ and furthermore, for $x=\pm e^{u}$,
$$\sum_{i:2i \in B_2}R(i) x^{2i}\sim \frac{e^{nu}}{u}\sim \sigma_n(u), \quad \text{when  }u\geq \frac{L^{-1}}{\log n}.$$

Recall that $Q^{(2)}_n(x)\leq -\eta f_{\eta,\epsilon}\sum_{i:2i\in B_2} x^{2i}$ for all $x\in \mathbb{R}$ on the event $\Gamma_{Q,2}$. Hence $Q^{(2)}_n(e^u)\leq -\eta f_{\eta,\epsilon}\sigma_{n}(u)$ for all $u\geq \frac{\delta}{\log n}$ and $Q^{(2)}_n(e^{u})\leq 0$ for all $u\leq \frac{\delta}{\log n}$. This shows the claim in \eqref{eq:A_2LowerBd}.

Substituting the lower bound of $\P(\Gamma_{Q,2})$ from \eqref{eq:GammaQ2LowerBd} to the right hand side of \eqref{eq:A_2LowerBd} yields \eqref{eq:Q2LBd}.

\textbf{Case $(b)$:} Consider the event 
$$\tilde{\Gamma}_{Q,2}:= \big\{\xi_{2i}\leq -\rho, 0\leq \xi_{2i+1} \in [0,\rho], \text{ for all s.t. }2i\in B_2\big\}.$$
Note that $\P(\tilde{\Gamma}_{Q,2})$ is bounded above $c^{L\log n}$ since $\xi_i$'s are independent, $\P(\xi_{2i}\leq -\rho)\geq c$ and $\P(0\leq \xi_{2i+1}\leq \rho)\geq c$. In the same spirit of \eqref{eq:A_2LowerBd}, we now show that 
\begin{align}\label{eq:A_2Case2LowerBd}
\P\Big(\frac{Q^{(2)}_n(e^{u})}{\sigma_n(u)}\leq -\delta, \text{ for } u\geq \frac{\delta}{\log n}, \frac{Q^{(2)}_n(e^{u})}{\sigma_n(u)}\leq \delta, \text{ for }u \leq \frac{\delta}{\log n} \Big) \geq \P(\tilde{\Gamma}_{Q,2}).
\end{align}
On the event $\tilde{\Gamma}_{Q,2}$, we have 
\begin{align}
 \sum_{i\in B_2} \xi_i x^i &=\frac{1}{2}\sum_{2i \in B_{2}}x^{2i}\big(a_{2i}\sqrt{R(2i)}+2a_{2i+1}x\sqrt{R(2i+1)}+a_{2i+2}x^2\sqrt{R(2i+1)}\big)\nonumber\\ &\le- \frac{1}{2}\sum_{i: 2i \in B_2}\rho\sqrt{R(2i)} x^{2i} \Big(1-2|x|\frac{\sqrt{R(2i+1)}}{\sqrt{R(2i)}}+\frac{\sqrt{R(2i+1)}}{\sqrt{R(2i)}}x^2\Big).\label{eq:Case_2RHS} 
 \end{align}
 Recall that $\lim_{i\to \infty} R(i+1)/R(i) = 1$ and let $i_0>0$ be such that $R(i+1)/R(i)\in (1-\varepsilon, 1+\varepsilon)$. For all large $n$, we have 
 \begin{align*}
\text{r.h.s. of \eqref{eq:Case_2RHS}}\leq -\frac{\rho}{2}\sum_{i: 2i \in B_2}\sqrt{R(2i)} x^{2i}\Big(1- 2\eta (1+\varepsilon) |x| +  (1-\varepsilon)x^2\Big).      
 \end{align*}
  Note that the discriminant of the quadratic polynomial $\tilde{\mathcal{P}}(x) = 1- 2\eta(1+\varepsilon) x+ (1-\varepsilon)x^2$ is $4\big(\eta^2(1+\varepsilon)-(1-\varepsilon)\big)<0$ since the choice of $\eta$ is made in such a way to satisfy this inequality.
  This shows the right hand side of \eqref{eq:Case_2RHS} is bounded above $-\frac{\rho+\epsilon}{2}f_{\rho,\eta,\epsilon}\sum_{i: 2i\in B_2}\sqrt{R(2i)} x^{2i}$ for some positive constant $f_{\rho,\eta,\epsilon}$. This implies \eqref{eq:A_2Case2LowerBd} and hence, completes the proof.  
\end{proof}
 
\subsection{Proof of Proposition~\ref{ppn:r=-2}: $r=-2$ Case}\label{sec:r=-2}
The dominant interval here is $A_{-2}= (-\infty,-\frac{\delta}{ \log n}]$ and the corresponding interval for the coefficient indices is $B_{-2} = [0,L\log n]\cap \mathbb{Z}$.

\begin{proof}[Proof of Proposition~\ref{ppn:r=-2}] One can prove this proposition by exactly in the same way as in Proposition~\ref{ppn:r=-2}, We skip the details for brevity.

\end{proof}

\subsection{Supporting Lemmas}

 We first state the following two lemmas, which we will use to lower bound the terms in the product \eqref{eq:lower_bound*}. We begin by introducing the following notations. 

\begin{defn}\label{def:tildeY}

Fixing $a\in\{-,+\}$, define a process on $[M^{-1},1]\times \{-,+\}$ by setting
$$\tilde{Y}_{n,r,M}^{(a),p}(b,t):=\frac{\tilde{Q}_n^{(a),p}(b e^{at\hat{\tau}})}{\sigma_n(at\hat{\tau})}, \quad \hat{\tau}:= \hat{\tau}_{n}(r,M) = \frac{1}{M^{r}\log n }.$$
Note that
\begin{align}\label{eq:tildeY_ncovaraince}
\mathrm{Cov}(\tilde{Y}_{n,r,M}^{(a),p}(b,t_1), \tilde{Y}_{n,r,M}^{(a),p}(b,t_2))=\frac{\tilde{h}_{n,r,M}^{(a),p}(t_1+t_2)}{\sqrt{\tilde{h}_{n,r,M}^{(a),p}(2t_1) \tilde{h}_{n,r,M}^{(a),p}(2t_2)}}
\end{align}
where $\tilde{h}$ is defined as 
 \begin{align}
 \tilde{h}_{n,r,M}^{(a), p}(t) = \sum_{i \in I^{-1}_p} R(i)e^{it\hat{\tau}}. 
 \end{align}

For any integer $s\in \Z$ and $t\ge 0$ set $$\tilde{h}^{(-1)}_{s,M}(t):=\int_{M^{s}}^{M^{s+1}} x^\alpha e^{-xt} dx, \quad \tilde{h}^{(+1)}_{s,M}(t):=\int_{M^{s}}^{M^{s+1}}  e^{-xt} dx.$$

Also let $\{\tilde{Y}_{s,M}^{(a)}(b,.)\}_{b\in \{+,-\}}$ be i.i.d.~centered Gaussian processes defined on the interval $[M^{-1},1]$, with
\begin{align}\label{eq:tildeYcovaraince}
\mathrm{Cov}(\tilde{Y}_{s,M}^{(a)}(b,t_1), \tilde{Y}_{s,M}^{(a)}(b,t_2))=\frac{\tilde{h}_{s,M}^{(a)}(t_1+t_2)}{\sqrt{\tilde{h}_{0,M}^{(a)}(2t_1) \tilde{h}_{0,M}^{(a)}(2t_2)}}
\end{align}

\end{defn}

 \begin{lem}\label{lem:weakAlt}
For any  $a\in \{+,-\}$ and $(r_n,\ell_n)\in [K]$ with $p_n-r_n\to s\in \Z$, we have 
$$\big\{\tilde{Y}_{n,r_n,M}^{(a),p_n}(b,t), b\in \{-,+\}, t\in [1,M]\big\}\stackrel{D}{\to} \big\{\tilde{Y}_{s,M}^{(a)}(b,t), b\in \{-,+\}, t\in [1,M]\big\}.$$
Here the convergence is in the topology of $\mathcal{C}[1,M]^{\otimes 2}$.
\end{lem}

\begin{lem}\label{lem:weak*Alt}


Recall the functions $\tilde{h}_{n,r,M}^{(a), p}(t)$ and $\tilde{h}^{(-1)}_{s,M}(t)$.   
\begin{enumerate}

\item[(1)]  Suppose $a=-$.

\begin{itemize}
\item[(i)] For all $n$ large enough (depending on $\delta$ and $L(.)$) we have
 $$M^{-|\ell-r|\delta}\frac{L(n^\delta M^\ell)}{\hat{\tau}^{\alpha+1} }\tilde{h}_{\ell-r, M, \alpha}(t) \lesssim \tilde{h}_{n,r,M}^{(-),\ell}(t)\lesssim M^{|\ell-r|\delta}\frac{L(n^\delta M^\ell)}{\tau^{\alpha+1} }\tilde{h}_{\ell-r,M,\alpha}(t).$$

\item[(ii)]  For any positive integer $\kappa$, we have 
\begin{align}
\lim_{n\to\infty}\max_{r,\ell\in [K]: |r-\ell|\le \kappa}\left|\frac{\hat{\tau}^{\alpha+1} \tilde{h}_{n,r,M}^{(-),\ell}(t)}{L(n^\delta M^r)\tilde{h}_{\ell-r,M,\alpha}(t)}-1\right|=0.
\end{align}

\item[(iii)] If $b_1=b_2$, for any positive integer $\kappa$ and $t_1,t_2\in[1,M]$ we have
$$\lim_{n\to\infty}\max_{r,\ell\in [K], |r-\ell|\le \kappa}\left|\frac{Cov\Big( \frac{\tilde{Q}_n^{(-),\ell}(b_1 e^{at_1 \hat{\tau}})}{\sigma_n(at_1 \hat{\tau})},  \frac{\tilde{Q}_n^{(-),\ell}(b_2 e^{at_2\hat{\tau}})}{\sigma_n(at_2\hat{\tau})}\Big)}{\tilde{C}_{s,M,\alpha}(t_1,t_2)}-1\right|=0,$$
{ where } $\tilde{C}_{s,M,\alpha}(t_1,t_2):= \frac{\tilde{h}_{s,M,\alpha}(t_1+t_2)}{\sqrt{\tilde{h}_{0,M,\alpha }(2t_1) \tilde{h}_{0,M,\alpha }(2t_2)}}$ for $t_1,t_2\in [1,M]$.

\end{itemize}

\item[(2)] Suppose $a=+$.

\begin{itemize}

\item[(i)] For all $n$ large enough (depending on $\delta$ and $L(.)$) we have
 $$ \tilde{h}_{n,r,M}^{(+),\ell}(t)\asymp  \frac{R(n) e^{nt\hat{\tau}}}{\hat{\tau}} \tilde{h}_{\ell-r,M,\alpha}(t).$$

\item[(ii)] For any $M>0$,  $t\in [1,M]$ and positive integer $\kappa$  we have 
$$\lim_{n\to\infty}\max_{r,\ell\in [K]: |r-\ell|\le \kappa}\left|\frac{\hat{\tau} \tilde{h}_{n,r,M}^{(+),\ell}(t)}{R(n)e^{nt\hat{\tau}}\tilde{h}_{\ell-r,M,0}(t)}-1\right|=0.$$

\item[(iii)] If $b_1=b_2$, for any positive integer $\kappa$ and $t_1,t_2\in [1,M]$ we have
$$\lim_{n\to\infty}\max_{r,\ell\in [K]:|r-\ell|\le \kappa}\left|\frac{Cov\Big( \frac{\tilde{Q}_n^{(+),\ell}(b_1 e^{at_1 \hat{\tau}})}{\sigma_n(at_1\hat{\tau})},  \frac{\tilde{Q}_n^{(+),\ell}(b_2 e^{at_2\hat{\tau}})}{\sigma_n(at_2\hat{\tau})}\Big)}{C_{\ell-r,M,0}(t_1,t_2)}-1\right|=0. $$

\end{itemize}

\item[(3)] If $b_1\ne b_2$ then for any $t_1,t_2\in [1,M]$ and positive integer $\kappa$ we have
$$\lim_{n\to \infty}\max_{r,\ell\in [K]:|r-\ell|\le \kappa}\Big|Cov\Big( \frac{\tilde{Q}_n^{(a),\ell}(b_1 e^{at_1\hat{\tau}})}{\sigma_n(at_1\hat{\tau})},  \frac{\tilde{Q}_n^{(a),\ell}(b_2 e^{at_2\hat{\tau}})}{\sigma_n(at_2\hat{\tau})}\Big)\Big|= 0.$$

\end{enumerate} 

\end{lem}

The proof of Lemma~\ref{lem:weakAlt} and~\ref{lem:weak*Alt} are very similar to that of Lemma~\ref{lem:weak} and Lemma~\ref{lem:weak*Alt}. We will skip the details.

\begin{lem}\label{lem:kol_cent}
 Suppose $\{X(t),t\in [c,d]\}$ is a continuous time stochastic process with continuous sample paths, such that $\E X(t)=0$ and $\E(X(s)-X(t))^2\le C^2(s-t)^2$. Then the following holds:
 \begin{align}\label{eq:kol_cent1}
 \P(\sup_{t\in [c,d]}|X(t)-X(d)|>\lambda)\lesssim \frac{C^2(c-d)^2}{\lambda^2} &\\
\label{eq:kol_cent2}\lim_{\delta\rightarrow0}\lim_{n\rightarrow\infty}\P(\sup_{s,t\in [c,d],|s-t|\le \delta }|X_n(s)-X_n(t)|>\varepsilon)&\stackrel{p}{\rightarrow}0,\text{ for any }\varepsilon>0.
\end{align}

 \end{lem}
  
  \begin{proof}[Proof of Lemma \ref{lem:kol_cent}]
  For a non negative integer $m$ and $\ell\in[2^m]$, setting $r_m(\ell):=c+(d-c)\frac{\ell}{2^m}$ we can write $[c,d]=\cup_{\ell=1}^{2^m} [r_m(\ell-1),r_m(\ell)]$. 
For $m\ge 0$ and $u\in I_a$  let $\pi_m(u):=\ell$ if $u\in [r_m(\ell-1),r_m(\ell)]$. Then, using continuity of sample paths and noting that $\pi_0(t)=d$ we can write
\begin{align}\label{eq:chaining}
X(t)-X(d)=&\sum_{m=1}^\infty \Big(X(\pi_m(t))-X(\pi_{m-1}(t))\Big).
\end{align}
To prove \eqref{eq:kol_cent1}, use \eqref{eq:chaining} along with a union bound to get
\begin{align*}
\P(\sup_{t\in [c,d]}|X(t)-X(d)|>\lambda)\le &\sum_{m=1}^\infty \sum_{\ell=1}^{2^m}\P\Big(2^{m/4}|X(r_m(\ell))-X(r_m(\ell-1))|>\lambda\Big)\\
\lesssim &\frac{C^2(c-d)^2}{\lambda^2}\sum_{m=1}^\infty 2^{-m/4},
\end{align*}
where the last line uses Chebyshev's inequality. The desired conclusion is immediate from this.
\end{proof}

\begin{lem}\label{lem:tilde_Y}
Consider the centered Gaussian processes $\{\tilde{Y}^{(a)}_{s,M}(b,\cdot)\}_{b \in \{-1,+1\}}$ of Definition~\ref{def:tildeY}. For any given $\varepsilon>0$, there exists $\theta_0=\theta_0(\varepsilon)\in (0,\frac{1}{2})$ such that for any $0<\theta<\theta_0$ one gets $M_0= M_0(\epsilon,\theta)>0$ satisfying 
\begin{align}\label{eq:LowBdOfLimit1}
   \frac{1}{\log M} \log \P\Big(\sup_{t\in [\frac{1}{M},1]}\tilde{Y}^{(-1)}_{0,M}(t)\leq -\delta\Big) &\geq \frac{1}{\log M}\log \mathbb{P}\Big(\sup_{t\in [0, (1-2\theta)\log M]} Y^{(\alpha)}_t\leq -\delta\Big) - \varepsilon, \\ \frac{1}{\log M}\log \P\Big(\sup_{t\in [\frac{1}{M},1]}\tilde{Y}^{(+1)}_{0,M}(t)\leq -\delta\Big) &\geq \frac{1}{\log M}\log \mathbb{P}\Big(\sup_{t\in [0,(1-2\theta)\log M]} Y^{(0)}_t\leq -\delta\Big) - \varepsilon\label{eq:LowBdOfLimit2}
\end{align}
for all $M>M_0$.
\end{lem}
\begin{proof}
Here we show how to prove \eqref{eq:LowBdOfLimit1}. Proof of \eqref{eq:LowBdOfLimit2} follows from similar arguments. Fix $\theta \in (0,\frac{1}{2})$ small. To this end, we divide $[0,M]$ into three sub-intervals; $\mathfrak{A}_1:= [M^{-1},M^{-1+\theta}), \mathfrak{A}_2:= [M^{-1+\theta}, M^{-\theta}]$ and  $(M^{-\theta}, 1]$. Recall that the covariance function $\tilde{h}^{(-1)}_{0,M}$ of $\tilde{Y}^{(-1)}_{0,M}(t)$ from \eqref{eq:tildeYcovaraince}.
Since the covariance is non-negative, by Slepian's inequality of the Gaussian process, we get 
\begin{align*}
   \log \P\Big(\sup_{t\in [\frac{1}{M},1]}\tilde{Y}^{(-1)}_{0,M}(t)\leq -\delta\Big) &\geq  \log \P\Big(\sup_{t\in \mathfrak{A}_2}\tilde{Y}^{(-1)}_{0,M}(t)\leq -\delta\Big)\\ &+ \log \P\Big(\sup_{t\in \mathfrak{A}_1\cup\mathfrak{A}_3}\tilde{Y}^{(-1)}_{0,M}(t)\leq -\delta\Big)
\end{align*}
For any fixed $s_1,s_2>0$ such that $t_1:=e^{s_1},t_2:=e^{s_2}\in \mathfrak{A}$, we have $\lim_{M\to \infty} \tilde{h}^{(-1)}_{0,M}(t_1) = \int^{\infty}_{1} x^{\alpha} e^{-t_1x}dx $ and $\lim_{M\to \infty} \tilde{h}^{(-1)}_{0,M}(t_2) = \int^{\infty}_{1} x^{\alpha} e^{-t_2x}dx $. This shows $$\mathrm{Corr}(\tilde{Y}^{(-1)}_{0,M}(e^{s_1}), \tilde{Y}^{(-1)}_{0,M}(e^{s_2}))\to \mathrm{sech}(s_1-s_2)^{\alpha+1}$$ as $M$ approaches $\infty$. Recall that $\mathrm{sech}(s_1-s_2)^{\alpha+1}$ is the correlation function of the stationary Gaussian process $Y^{(\alpha)}$. We now intend to apply Theorem~1.6 of \cite{DemboMukherjee2015} which will imply that 
\begin{align} \label{eq:ExpoConverge}
\lim_{M\to \infty} &\frac{1}{\log M}\log \P\Big(\sup_{t\in \mathfrak{A}_2}\tilde{Y}^{(-1)}_{0,M}(t)\leq -\delta\Big)\nonumber\\ &= \lim_{M\to \infty} \frac{1}{\log M}\log \P\Big(\sup_{s\in [0,(1-2\theta)\log M]}Y^{(\alpha)}_s\leq -\delta\Big). 
\end{align}
However, according to \cite[Theorem~1.6]{DemboMukherjee2015}, the above limit holds when the covariance function of the Gaussian processes $Y^{(-1)}_{0,M}(e^{s})$ satisfy condition 
(1.15) of \cite[Theorem~1.6]{DemboMukherjee2015}.  For this it suffices to check Condition~(1.23) of \cite[Theorem~1.6]{DemboMukherjee2015} which is given as 
\begin{align}\label{eq:DM_Condition}
\limsup_{u\to 0} &|\log u|^{\eta}\sup_{M\geq 1} \Big\{2 \nonumber\\ &- 2\inf_{s\geq 0, \tau\in [0,u]}\frac{\int^{M}_1 x^{\alpha}\exp(-(e^{s}+e^{s+\tau})x)dx}{\sqrt{\int^{M}_1 x^{\alpha}\exp(-2e^{s}x)x)dx}\sqrt{\int^{M}_1 x^{\alpha}\exp(-2e^{s+\tau}x))dx}} \Big\}<\infty
\end{align}
for some $\eta>0$. To show the above condition, we first claim and prove that there exists $\gamma>0$ such that 
\begin{align}\label{eq:BoundByStationary}
\frac{\int^{M}_1 x^{\alpha}\exp(-(e^{s}+e^{s+\tau})x)dx}{\sqrt{\int^{M}_1 x^{\alpha}\exp(-2e^{s}x)x)dx}\sqrt{\int^{M}_1 x^{\alpha}\exp(-2e^{s+\tau}x))dx}}\geq e^{\tau\gamma}
\end{align}
for all $M,s,\tau>0$. 
 Recall that for any $t_1,t_2\in [\frac{1}{M},1]$
\begin{align}
    \mathrm{Cov}\Big(\tilde{Y}^{(-1)}_{0,M}(t_1), \tilde{Y}^{(-1)}_{0,M}(t_1)\Big) =\frac{\int^{M}_1 x^{\alpha} e^{-(t_1+t_2)x}dx}{\sqrt{\int^{M}_1 x^{\alpha} e^{-2t_1x}dx}\sqrt{\int^{M}_1 x^{\alpha} e^{-2t_2x}dx}}. \label{eq:CovOftildeY} 
\end{align}
Suppose that $t_1\leq t_2$. In that case, we get the following following bound 
\begin{align}\label{eq:NewInequality}
\text{r.h.s. of \eqref{eq:CovOftildeY}} \geq \frac{\sqrt{\int^{M}_1 x^{\alpha} e^{-2t_2x}dx}}{\sqrt{\int^{M}_1 x^{\alpha} e^{-2t_1x}dx}}\geq \frac{\sqrt{\int^{\infty}_1 x^{\alpha} e^{-2t_2x}dx}}{\sqrt{\int^{\infty}_1 x^{\alpha} e^{-2t_1x}dx}}
\end{align}
where the second inequality follows since $x^{\alpha} e^{-tx}$ strictly decreases as $x$ increases to $\infty$ and $x^{\alpha}e^{-t_2x}\leq x^{\alpha}e^{-t_1x}$ for all $x>0$. Now we claim and prove that there exists some $\gamma>0$ such that 
$$\text{r.h.s. of \eqref{eq:NewInequality}}\geq \big(\frac{t_1}{t_2}\big)^{\gamma}$$
for all $M^{-1}\leq t_1\leq t_2\leq 1$. To show this, we consider the function $g:[M^{-1},1]\to \mathbb{R}_{>0}$ defined as $g(t) = t^{2\gamma}\int^{\infty}_1 x^{\alpha} e^{-tx}dx$. Note that $$\frac{g'(t)}{g(t)} = \frac{2\gamma}{t} - \frac{\int^{\infty}_1 x^{\alpha+1} e^{-tx}dx}{ \int^{\infty}_1 x^{\alpha+1} e^{-tx}dx}.$$
This shows that if $\gamma$ is chosen such that 
$$2\gamma \geq \sup_{t\in [0,1]}t\frac{\int^{\infty}_1 x^{\alpha+1} e^{-tx}dx}{ \int^{\infty}_1 x^{\alpha+1} e^{-tx}dx},$$
$g(\cdot)$ will be strictly increasing on the interval $[M^{-1},1]$ for all $M$. Hence, for all $\gamma>0$ satisfying the above inequality, we get for all $M^{-1}\leq t_1\leq t_2\leq 1$
$$\text{r.h.s. of \eqref{eq:NewInequality}}\geq \Big(\frac{t_1}{t_2}\Big)^{\gamma}.$$
Taking $t_1=e^{s}$ and $t_2=e^{s+\tau}$ shows \eqref{eq:BoundByStationary}. From \eqref{eq:BoundByStationary}, it follows that 
$$2- 2\inf_{s\geq 0, \tau\in [0,u]}\frac{\int^{M}_1 x^{\alpha}\exp(-(e^{s}+e^{s+\tau})x)dx}{\sqrt{\int^{M}_1 x^{\alpha}\exp(-2e^{s}x)x)dx}\sqrt{\int^{M}_1 x^{\alpha}\exp(-2e^{s+\tau}x))dx}}\leq 2(1-e^{u\gamma})$$
uniformly for all $M>0$. The above inequality shows that \eqref{eq:DMCondition} holds for any $\eta>0$ since $|\log u|^{\eta} (2-e^{u\gamma})\to 0$ as $u\to 0$.

To complete the proof, it suffices to show that there exists $\theta_0\in (0,\frac{1}{2})$ such that for all $0<\theta<\theta_0$ and $M$ large, one has 
\begin{align}\label{eq:ClippedOutPart}
    \frac{1}{\log M}\log \P\Big(\sup_{t\in \mathfrak{A}_1}\tilde{Y}^{(-1)}_{0,M}(t)\leq -\delta\Big) &\geq -\frac{\varepsilon}{2}, \nonumber\\ \frac{1}{\log M}\log \P\Big(\sup_{t\in\mathfrak{A}_3}\tilde{Y}^{(-1)}_{0,M}(t)\leq -\delta\Big) &\geq -\frac{\varepsilon}{2}. 
\end{align}
We only show the first inequality. The second follows from similar arguments.
Note that $\mathfrak{A}_2$ transforms to $[-(1-\theta) \log M, -\theta \log M]$ under the change of variable $t\mapsto \log t$. Furthermore, by the stationarity of $Y^{(\alpha)}$, we have  $$\sup_{[-(1-\theta) \log M, -\theta\log M]} Y^{(\alpha)}_s \stackrel{d}{=}\sup_{s\in [0, (1-2\theta)\log M]} Y^{(\alpha)}_s.$$ This shows why $s$ takes values on the interval $[0,(1-2\delta)\log M]$ in the right hand of \eqref{eq:ExpoConverge}. Now we we lower bound $\log \P\Big(\sup_{t\in \mathfrak{A}_1\cup\mathfrak{A}_3}\tilde{Y}^{(-1)}_{0,M}(t)\leq -\delta\Big)$.
By Slepian's inequality, we know 
\begin{align*}
\log \P\Big(\sup_{t\in \mathfrak{A}_1\cup\mathfrak{A}_3}\tilde{Y}^{(-1)}_{0,M}(t)\leq -\delta\Big) &\geq \log\P\Big(\sup_{t\in \mathfrak{A}_1}\tilde{Y}^{(-1)}_{0,M}(t)\leq -\delta\Big)\\ &+ \log \P\Big(\sup_{t\in\mathfrak{A}_3}\tilde{Y}^{(-1)}_{0,M}(t)\leq -\delta\Big).
\end{align*}

Consider a mean zero Gaussian process $\{\mathfrak{Z}^{(\gamma)}_t\}_{t\geq 0}$ such that $\mathrm{Cov}(\mathfrak{Z}^{(\gamma)}_{t_1}, \mathfrak{Z}^{(\gamma)}_{t_2})=(t_1/t_2)^{\gamma}$ for $t\leq t_2$. Due to \eqref{eq:BoundByStationary} and Slepian's inequality, we get 
\begin{align}\label{eq:BoundWithStationary}
    \mathbb{P}\big(\sup_{t\in \mathfrak{A}_1}\tilde{Y}^{(-1)}_{0,M}(t)\leq -\delta\big)\geq \P\big(\sup_{t\in \mathfrak{A}_1}\mathfrak{Z}^{(\gamma)}_t\leq -\delta\big) = \P\Big(\sup_{t\in [-\log, -(1-\theta)\log M  ]}\mathfrak{Z}^{(\gamma)}_{e^t}\leq -\delta\Big)
\end{align}
Note that $\mathfrak{Z}^{(\gamma)}_{e^t}$ is a stationary Gaussian process. By \cite[Theorem~1.6]{DemboMukherjee2015}, 
$$-\lim_{M\to \infty}\frac{1}{\theta \log M }\log P\Big(\sup_{t\in [-\log, -(1-\theta)\log M  ]}\mathfrak{Z}^{(\gamma)}_{e^t}\leq -\delta\Big)<\infty.$$
Combining the above fact with \eqref{eq:BoundWithStationary} shows that there exists $\theta_0=\theta_0(\epsilon)>0$ such that for all $\theta<\theta_0$ and large $M$ large, first inequality of \eqref{eq:ClippedOutPart} holds. The proof of the second inequality is similar. 
\end{proof}

\bibliographystyle{alpha}
	\bibliography{frt}


\end{document}